\documentclass[a4paper,11pt]{amsart}

\pdfoutput=1

\usepackage[text={400pt,660pt},centering]{geometry}

\usepackage{amsthm, amssymb, amsmath, amsfonts, mathrsfs}
\usepackage{mathtools}
\usepackage{scalerel} 
\usepackage[colorlinks=true, pdfstartview=FitV, linkcolor=blue, citecolor=blue, urlcolor=blue,pagebackref=false]{hyperref}

\usepackage{esint} 
\usepackage{MnSymbol} 

\usepackage{microtype}
\usepackage{cleveref}

\newtheorem{proposition}{Proposition}
\newtheorem{theorem}[proposition]{Theorem}
\newtheorem{lemma}[proposition]{Lemma}
\newtheorem{corollary}[proposition]{Corollary}

\theoremstyle{remark}
\newtheorem{remark}[proposition]{Remark}

\theoremstyle{definition}

\numberwithin{equation}{section}
\numberwithin{proposition}{section}

\renewcommand{\le}{\leqslant}
\renewcommand{\ge}{\geqslant}
\renewcommand{\leq}{\leqslant}
\renewcommand{\geq}{\geqslant}

\renewcommand{\subset}{\subseteq}
\newcommand{\mcl}{\mathcal}
\newcommand{\msf}{\mathsf}
\newcommand{\msc}{\mathscr}
\newcommand{\A}{{\mathsf{A}}}

\newcommand{\G}{\mathcal{G}}
\renewcommand{\S}{\mathsf{S}}

\newcommand{\e}{\mathbf{e}}
\newcommand{\E}{\mathbb{E}}
\newcommand{\Er}{\mathbb{E}_{\rho}}
\renewcommand{\Pr}{\mathbb{P}_{\rho}}

\renewcommand{\a}{\mathbf{a}}
\newcommand{\ab}{{\overbracket[1pt][-1pt]{\a}}}
\newcommand{\at}{\tilde{\ab}}

\newcommand{\ahom}{{\overbracket[1pt][-1pt]{\a}}}

\newcommand{\N}{\mathbb{N}}

\newcommand{\Ll}{\left}
\newcommand{\Rr}{\right}

\newcommand{\1}{\mathbf{1}}
\newcommand{\R}{\mathbb{R}}

\newcommand{\Q}{\mathbb{Q}}
\newcommand{\Z}{\mathcal{Z}}
\newcommand{\Zd}{{\mathbb{Z}^d}}
\renewcommand{\P}{\mathbb{P}}
\newcommand{\ov}{\overline}
\renewcommand{\bar}{\overline}

\newcommand{\td}{\widetilde}
\renewcommand{\tilde}{\widetilde}
\newcommand{\al}{\alpha}
\newcommand{\de}{\delta}

\newcommand{\ep}{\varepsilon}
\renewcommand{\d}{{\mathrm{d}}}

\newcommand{\dr}{\partial}

\renewcommand{\epsilon}{\varepsilon}

\newcommand{\cu}{{\scaleobj{1.2}{\square}}}

\renewcommand{\fint}{\strokedint}
\newcommand{\Rd}{{\mathbb{R}^d}}
\newcommand{\n}{\mathbf{n}}
\newcommand{\mmd}{\mathcal{M}_\delta}
\newcommand{\mmb}{\mathcal{M}_\bullet}

\DeclareMathOperator{\dist}{dist}
\DeclareMathOperator{\supp}{supp}
\DeclareMathOperator{\diam}{diam}

\newcommand{\cC}{\msc{C}}   
\newcommand{\cfC}{\mcl{F}\msc{C}} 
\newcommand{\cL}{\msc{L}}   
\newcommand{\cH}{\msc{H}}   
\newcommand{\acL}{\underline{\msc{L}}}
\newcommand{\acH}{\underline{\msc{H}}}

\newcommand{\poi}{\text{Poi}}
\newcommand{\Ind}[1]{\mathbf{1}_{\left\{#1\right\}}}
\newcommand{\id}{\mathsf{Id}}
\newcommand{\norm}[1]{\left\Vert{#1}\right\Vert}

\renewcommand{\n}{\overrightarrow{\mathbf{n}}}

\newcommand{\mres}{\mathbin{\vrule height 1.4ex depth 0pt width
0.13ex\vrule height 0.13ex depth 0pt width 1ex}}

\newcommand{\mressmall}{\mathbin{\vrule height 1ex depth 0pt width
0.1ex\vrule height 0.1ex depth 0pt width 0.7ex}}

\title[Quantitative homogenization of interacting particle systems]{Quantitative homogenization of \\ interacting particle systems}
\author{Arianna Giunti, Chenlin Gu, Jean-Christophe Mourrat}

\address[Arianna Giunti]{Department of Mathematics, Imperial College London, London, United Kingdom}

\address[Chenlin Gu]{DMA, Ecole normale sup\'erieure, PSL University, Paris, France}

\address[Jean-Christophe Mourrat]{Courant Institute of Mathematical Sciences, New York University, New York, New York, USA; CNRS, France}

\begin{document}

\begin{abstract}
For a class of interacting particle systems in continuous space, we show that finite-volume approximations of the bulk diffusion matrix converge at an algebraic rate. The models we consider are reversible with respect to the Poisson measures with constant density, and are of non-gradient type. 
Our approach is inspired by recent progress in the quantitative homogenization of elliptic equations. 
Along the way, we develop suitable modifications of the Caccioppoli and multiscale Poincar\'e inequalities, which are of independent interest. 

\bigskip


\medskip


\end{abstract}
\maketitle


%
%
%
%
%
%
%
%

\section{Introduction}

The goal of this paper is to make progress on the quantitative analysis of interacting particle systems. We consider a class of models in which each particle follows a random evolution on $\R^d$ which is influenced by the configuration of neighboring particles. The models we consider are reversible with respect to the Poisson measures with constant density, uniformly elliptic, and of non-gradient type.
For similar models in this class, the hydrodynamic limit and the equilibrium fluctuations have been identified rigorously. In both these results, the limit object is described in terms of the \emph{bulk diffusion matrix}. The main result of this paper is a proof that finite-volume approximations of this diffusion matrix converge at an algebraic rate.

Our strategy is inspired by recent developments in the quantitative analysis of elliptic equations with random coefficients, and in particular on the renormalization approach developed in \cite{armstrong2016quantitative, armstrong2016lipschitz, armstrong2016mesoscopic, armstrong2017additive, AKMbook,ferg1, ferg2}; see also \cite{informal} for a gentle introduction, and \cite{NS,vardecay, gloria2011optimal,gloria2012optimal,gloria2015quantification,
GO3,gloria2020regularity} for another approach based on concentration inequalities. This renormalization approach has shown its versatility in a number of other settings, covering now the homogenization of parabolic equations \cite{armstrong2018parabolic},  finite-difference equations on percolation clusters \cite{armstrong2018elliptic,dario2018optimal,dario2019quantitative}, differential forms \cite{dario2018differential}, the ``$\nabla \phi$'' interface model \cite{dario2019phi,gradphi2}, and the Villain model \cite{dario2020massless}. 

Here as in the other settings mentioned above, we start from a representation of the finite-volume approximation of the bulk diffusion matrix as a family of variational problems, denoted by $\nu(U,p)$, where $U \subset \Rd$ and $p \in \Rd$ encodes a slope parameter. This quantity is subadditive as a function of the domain $U$. We then identify another subadditive quantity, denoted by $\nu^*(U,q)$, with $U \subset \Rd$ and $q \in \Rd$, such that $\nu^*(U,\cdot)$ is approximately convex dual to $\nu(U,\cdot)$. These quantities $\nu$ and $\nu^*$ provide with finite-volume lower and upper approximations of the limit diffusion matrix. Roughly speaking, the algebraic rate of convergence is obtained by showing that the defect in the convex duality between $\nu$ and $\nu^*$ can be controlled by the variation of $\nu$ and $\nu^*$ between two scales; we refer to \cite[Section~3]{informal} for some intuition as to why a control of this sort is plausible.

Besides the identification of the most appropriate subadditive quantities $\nu$ and~$\nu^*$, one of the main difficulties we encounter relates to the development of certain functional inequalities. As is to be expected, we will make use of Poincar\'e inequalities, which allow to control the $L^2$ oscillation of a function by the $L^2$ norm of its gradient. However, we will need to be more precise than this. Indeed, we want to be able to assert that if the gradient of a function is small in some weaker norm, then we can control the $L^2$ oscillation of the function more tightly. In other words, we need some analogue of the inequality $\|u\|_{L^2} \le C \|\nabla u\|_{H^{-1}}$. Recall that in the current paper, the functions of interest are defined over the space of all possible particle configurations. The precise statement of our ``multiscale Poincar\'e inequality'' is in \Cref{prop:Multi}. 

Another crucial ingredient we need is a version of the Caccioppoli inequality. In the standard setting of elliptic equations, this inequality states that the $L^2$ norm of the gradient of a harmonic function can be controlled by the $L^2$ norm of its oscillation on a larger domain; one can think of this inequality as a ``reverse Poincar\'e inequality'' for harmonic functions. If $u$ denotes the harmonic function, then a standard proof of this inequality consists in testing the equation for $u$ with $u\phi$, where $\phi$ is a smooth cutoff function which is equal to $1$ in the inner domain, and is equal to $0$ outside of the larger domain. 

In our context, we need to ``turn off'' the influence of \emph{any} particle that would come too close to the boundary of the larger domain. In this case, a naive modification of the standard elliptic argument is inapplicable. This comes from the fact that, as the domains become large, there will essentially always be many particles that come dangerously close to the boundary of the larger domain; so the cutoff function $\phi$ would essentially always have to vanish, except on an event of very small probability. We therefore need to identify a different approach. In fact, we settle for a modified form of the Caccioppoli inequality, in which we control the $L^2$ norm of the gradient of a solution by the $L^2$ norm of the solution on a larger domain, plus \emph{a fraction} of the $L^2$ norm of its gradient on the larger domain; see \Cref{prop:Caccioppoli} for the precise statement.

At present, we think that the results presented here should allow to derive a quantitative version of the hydrodynamic limit, as well as to derive ``near-equilibrium'' fluctuation results. To be more precise, for a domain of side length $R$ and an initial density profile varying macroscopically, it should be possible to control the convergence to the hydrodynamic limit at a precision of $R^{-\alpha}$, for some $\al > 0$. Conversely, starting from a density profile that has variations of size bounded by $R^{-\frac d 2 + \alpha}$, it should  be possible to identify the asymptotic fluctuations of the density field. These would represent first steps towards bridging the gap between these two results. 

By analogy with the results obtained for elliptic equations and other contexts, see in particular \cite[Section~3]{informal} and \cite[Chapter~2 and following]{AKMbook}, we hope that the results obtained here will provide the seed for more refined, and hopefully sharp, quantitative results. This will hopefully allow to improve the exponent $\alpha > 0$ appearing in the previous paragraph to some explicit exponent (ideally $\alpha = \frac d 2$), and thereby to bring us closer to a full understanding of non-equilibrium fluctuations.

We now turn to a brief overview of related works on interacting particle systems. The result in the literature that is possibly closest to ours is that of \cite{lov2}. In this work, the authors consider the diffusion matrix associated with the long-time behavior of a tagged particle in the symmetric simple exclusion process, which is called the \emph{self-diffusion matrix}. The main result of \cite{lov2} is a proof that finite-volume approximations of the self-diffusion matrix converge to the correct limit. However, no rate of convergence could be obtained there. The qualitative result of \cite{lov2} was extended to the mean-zero simple exclusion process, and to the asymmetric simple exclusion process in dimension $d \ge 3$, in \cite{jar06}.

An easy consequence of the results of the present paper is that the bulk diffusion matrix is H\"older continuous as a function of the density of particles. However, for related models, it was shown in \cite{var-reg, lov-reg,bernardin,lov-reg2,sued,naga1,naga2,naga3} that the diffusion matrix depends smoothly on the density of particles. The situation seems comparable to that encountered when considering Bernoulli perturbations of the law of the coefficient field for elliptic equations, see \cite{dl-diff,dg1}. Possibly more difficult situations for obtaining regularity results on the homogenized parameters, with less independence built into the nature of the perturbation, include the $\nabla \phi$ model \cite{gradphi2}, and nonlinear elliptic equations \cite{ferg1, ferg2}. 

Two classical approaches to the identification of the hydrodynamic limit have been developed. The first, called the entropy method, was introduced in \cite{gpv}, and extended to certain non-gradient models in \cite{varadhanII, quastel}. The second, called the relative entropy method, was introduced in \cite{yau}, and was extended to a non-gradient model in \cite{fuy}. 

The asymptotic description of the fluctuations of interacting particle systems at equilibrium has been obtained in \cite{broxrost,spo86,mpsw86,cha94,chayau92}. The extension of this result to non-gradient models was obtained in \cite{lu94,cha96,fun96}. 

We are not aware of any results concerning the non-equilibrium fluctuations of a non-gradient model. For gradient models (or small perturbations thereof), we refer in particular to \cite{prespo83,mfl86, fpv88,chayau92, jarmen18}. We also refer to the books \cite{spohn2012large, kipnis1998scaling, komorowski2012fluctuations} for much more thorough expositions on these topics, and reviews of the literature. 

In relation to the purposes of the present paper, several works considered the problem of obtaining a rate of convergence to equilibrium for a system of interacting particles \cite{lig91, deu94, berzeg,jlqy,lanyau,ccr,gu2020decay}. Heat kernel bounds for the tagged particle in a simple exclusion process were obtained in \cite{hk-ips}.

In all likelihood, the results presented here can be extended to other reversible models of non-gradient type, provided that the invariant measures satisfy some mixing condition (an algebraic decay of correlations would suffice, see \cite{armstrong2016lipschitz}). More challenging directions include dynamics that are not uniformly elliptic, such as hard spheres. Extensions to situations in which the noise only acts on the velocity variable are likely to also be very challenging. Even further away are purely deterministic dynamics of hard spheres, as considered for instance in \cite{bgss20}. For any of these models, it would of course also be desirable to make progress on the quantitative analysis of the large-scale behavior of a tagged particle.

The rest of the paper is organized as follows. In \Cref{sec:notation}, we introduce some notation and state the main result precisely, see \Cref{thm:main}. We then prove several functional inequalities in \Cref{sec:func}, including the multiscale Poincar\'e inequality and the modified Caccioppoli inequality. In \Cref{sec:Subadditive}, we define the subadditive quantities, and establish their elementary properties. Finally, in \Cref{subsec:ConvergenceJ} we prove \Cref{thm:main}.

\section{Notation and main result}
\label{sec:notation}

In this section, we introduce some notation and state our main result. 

Let $\mmd(\Rd)$ be the set of $\sigma$-finite measures that are sums of Dirac masses on $\Rd$, which we think of as the space of configurations of particles.  We denote by $\Pr$ the law on $\mmd(\Rd)$ of the Poisson point process of density $\rho \in (0,\infty)$, with $\Er$ the associated expectation. We denote by $\mcl F_U$ the $\sigma$-algebra generated by the mappings $\mu \mapsto \mu(V)$, for all Borel sets $V \subset U$, completed with all the $\Pr$-null sets, and we set $\mcl F := \mcl F_{\Rd}$. We give ourselves a function $\a_\circ : \mmd(\Rd) \to \R^{d\times d}_{\mathrm{sym}}$, where $\R^{d\times d}_{\mathrm{sym}}$ is the set of $d$-by-$d$ symmetric matrices. We assume that this mapping satisfies the following properties:
\begin{itemize}
\item \emph{uniform ellipticity}: there exists $\Lambda < \infty$ such that for every $\mu \in \mmd(\Rd)$,
\begin{equation}
\label{e.unif.ell}
\forall \xi \in \Rd, \quad |\xi|^2 \le \xi \cdot \a_\circ(\mu) \xi \le \Lambda |\xi|^2 \; ;
\end{equation}
\item \emph{finite range of dependence}: denoting by $B_1$ the Euclidean ball of radius $1$ centered at the origin, we assume that $\a_\circ$ is $\mcl F_{B_1}$-measurable. 
\end{itemize}
We denote by $\tau_{-x} \mu$ the translation of the measure $\mu$ by the vector $-x \in \Rd$; explicitly, for every Borel set $U$, we have $(\tau_{-x} \mu)(U) = \mu(x + U)$. We extend $\a_\circ$ by stationarity by setting, for every $\mu \in \mmd(\Rd)$ and $x \in \Rd$, 
\begin{equation*}  
\a(\mu,x) := \a_\circ(\tau_{-x} \mu).
\end{equation*}
While it would be possible to provide with a direct definition of the asymptotic bulk diffusion matrix, see for instance \cite[Chapter~7]{kipnis1998scaling}, our purposes require that we identify suitable finite-volume versions of this quantity. Accordingly, for every bounded open set $U \subset \Rd$, we define the matrix $\ab(U) \in \R^{d\times d}_{\mathrm{sym}}$ to be such that, for every $p \in \Rd$,
\begin{multline}\label{eq:defVariation}
\frac{1}{2} p \cdot \ab(U) p 
\\
=  \inf_{\phi \in \cH^1_0(U)} \Er\Ll[\frac{1}{\rho \vert U \vert} \int_{U} \frac{1}{2} (p + \nabla \phi(\mu, x)) \cdot \a(\mu, x) (p + \nabla \phi(\mu, x)) \, \d \mu(x)\Rr].
\end{multline}
In this expression, the gradient $\nabla \phi(\mu,x)$ is such that, for any sufficiently smooth function $\phi$, $x \in \supp \mu$, and $k \in \{1,\ldots,d\}$,
\begin{equation}  
\label{e.def.deriv}
\e_k \cdot \nabla \phi(\mu, x) = \lim_{h \to 0} \frac{\phi(\mu - \delta_x + \delta_{x + h \e_k}) - \phi(\mu)}{h},
\end{equation}
with $(\e_1,\ldots, \e_d)$ being the canonical basis of $\Rd$. As will be explained in more details below, the space $\cH^1_0(U)$ is a completion of a space of functions that are $\mcl F_K$-measurable for some compact set $K \subset U$. The expectation $\E_\rho$ is taken with respect to the variable $\mu$, a notation we will always use to denote the canonical random variable on $(\mmd(\Rd),\mcl F,\Pr)$ (an explicit writing of $\int_U \cdots \, \d \mu(x)$ would actually involve a summation over every point in the intersection of $U$ and the support of $\mu$). For every $m \in \N$, we let $\cu_m = Q_{3^m}$ denote the cube of side length $3^m$. We define the \emph{bulk diffusion matrix} $\ab$ as
\begin{align}\label{eq:ab}
\ab := \lim_{m \to \infty} \ab(\cu_m).
\end{align}
Although we keep this implicit in the notation, we point out that the matrices~$\ab(U)$ and~$\ab$ depend on the density $\rho$ of particles, which we keep fixed throughout the paper. The fact that this definition of $\ab$ coincides with the more classical definition, which is directly stated in infinite volume, is explained in Appendix~\ref{sec:app.equiv.def} below. Our main result is to obtain an algebraic rate for the convergence in \eqref{eq:ab}.
\begin{theorem}\label{thm:main}
The limit in \eqref{eq:ab} is well-defined. Moreover, there exist an exponent $\alpha(d,\Lambda,\rho) > 0$ and a constant $C(d,\Lambda,\rho) < \infty$ such that for every $m \in \N$,
\begin{align}\label{eq:main}
\Ll| \ab(\cu_m) - \ab \Rr| \leq C 3^{-\alpha m}.
\end{align}
\end{theorem}

In the remainder of this section, we clarify some of the definitions appearing earlier, and introduce some more useful notation. 

\subsection{Continuum configuration space}
For the purposes of the present paper, we will not need to construct the stochastic process of interacting particles whose large-scale behavior is captured by the bulk diffusion matrix $\ahom$, so we contend ourselves with brief remarks here. Intuitively, the dynamics is a cloud of particles, which we can denote by 
\begin{equation*}  
\mu(t) = \sum_{i = 1}^\infty \de_{X_i(t)} \in \mmd(\Rd), \qquad t \ge 0,
\end{equation*}
and each coordinate $(X_i(t))_{t \ge 0}$ performs a diffusion with local diffusivity matrix given by $\a(\mu(t), X_i(t))$. 
General properties of diffusions on the space $\mmd(\Rd)$ have been studied using Dirichlet forms in \cite{albeverio1996canonical, albeverio1996differential, albeverio1998analysis, albeverio1998analysis2}; see also the survey \cite{rockner1998stochastic}. In our current setup, for a finite $N$ number of particles, the diffusion process can be defined in the standard way (say, using De Giorgi-Nash regularity results on the heat kernel, and Kolmogorov's theorems) as a diffusion on $(\R^d)^N$. For $\Pr$-almost every $\mu \in \mmd(\Rd)$, one can then define the dynamics of the entire cloud of particles using finite-volume approximations. 

Although we have defined $\a(\mu,x)$ for every $x \in \Rd$, we will in fact only need to appeal to this quantity in the case when $x$ is in the support of $\mu$. One possible example of local diffusivity function is $\a_\circ(\mu) := (1 + \1_{\{ \mu(B_1) = 1 \}}) \id$. For this example, a particle at position $x \in \Rd$ follows a Brownian motion with variance $2$ whenever there are no other particles in the unit ball centered at $x$, while it follows a Brownian motion with unit variance whenever there is at least one additional particle in this ball (there are also reflection effects at the transition between these two situations).

For every Borel set $U \subset \Rd$, we denote by $\mcl B_U$ the set of Borel subsets of $U$. For every $\mu \in \mmd(\Rd)$, we denote by $\supp \mu$ the support of $\mu$, and by $\mu \mres U \in \mmd(\Rd)$ the measure such that, for every Borel set $V \subset \Rd$,
\begin{equation*}  
(\mu \mres U)(V) = \mu(U \cap V). 
\end{equation*}
We will often use the following ``disintegration'' lemma for functions defined on~$\mmd(\Rd)$. For definiteness, we state it for functions taking values in $\R$, but this plays no particular role. Its proof is deferred to Appendix~\ref{sec:Appendix}. Whenever $U \subset \Rd$, we write $U^c$ to denote the complement of $U$ in $\Rd$. 
\begin{lemma}[Canonical projection]
\label{lem:Projection}
Let $f : \mmd(\Rd) \to \R$ be a function, and for every Borel set $U$, measure $\mu \in \mmd(\Rd)$, and $n \in \N$, let $f_n(\cdot, \mu \mres U^c)$ denote the (permutation-invariant) function
\begin{equation*}  
f_n(\cdot, \mu \mres U^c) :
\Ll\{
\begin{array}{rcl}  
U^n  & \to & \R \\
(x_1, \ldots, x_n) & \mapsto & f \Ll( \sum_{i = 1}^n \de_{x_i} + \mu \mres U^c \Rr). 
\end{array}
\Rr.
\end{equation*}

The following statements are equivalent. 

(1) The function $f$ is $\mcl F$-measurable. 

(2) For every $n \in \N$, the function $f_n$ is $\mcl B_U^{\otimes n} \otimes \mcl F_{U^c}$-measurable. 

\end{lemma}


\subsection{Lebesgue and Sobolev function spaces}
We define $\cL^2$ to be the space of $\mcl F$-measurable functions $f$ such that $\Er[f^2]$ is finite. 

Recall that for sufficiently smooth $f : \mmd(\Rd) \to \R$, $\mu \in \mmd(\Rd)$ and $x \in \supp \mu$, we define $\nabla f(\mu,x)$ according to the formula in \eqref{e.def.deriv}. We write $\nabla f = (\dr_1 f, \ldots, \dr_d f)$.

For every open set $U \subset \Rd$, we define the sets of smooth functions $\cC^\infty(U)$ and $\cC^\infty_c(U)$ in the following way. We have that $f \in \cC^{\infty}(U)$ if and only if $f$ is an $\mcl F$-measurable function, and for every bounded open set $V \subset U$, $\mu \in \mmd(\Rd)$ and $n \in \N$, the function $f_n(\cdot, \mu \mres V^c)$ appearing in Lemma~\ref{lem:Projection} is infinitely differentiable on $V^n$. 
The space $\cC^{\infty}_c(U)$ is the subspace of $\cC^{\infty}(U)$ of functions that are $\mcl F_K$-measurable for some compact set $K \subseteq U$.

We now define $\cH^1(U)$, an infinite dimensional analogue of the classical  Sobolev space $H^1$. For every $f \in \cC^\infty(U)$, we set
\begin{align*}
\norm{f}_{\cH^1(U)} = \Ll(\Er[f^2(\mu)] + \Er\Ll[\int_U \vert \nabla f(\mu, x) \vert^2 \, \d \mu(x) \Rr]\Rr)^{\frac{1}{2}}.
\end{align*}
The space $\cH^1(U)$ is the completion, with respect to this norm, of the space of functions $f \in \cC^\infty(U)$ such that $\norm{f}_{\cH^1(U)}$ is finite (elements in this function space that coincide $\Pr$-almost surely are identified). As in classical Sobolev spaces, for every $f \in \cH^1(U)$, we can interpret $\nabla f(\mu,x)$, with $x \in U$, in some weak sense. We stress that \emph{functions in $\cH^1(U)$ need not be $\mcl F_U$-measurable}. Indeed, the function $f$ can depend on $\mu \mres U^c$ in a relatively arbitrary (measurable) way, as long as $f \in \cL^2$. If $V \subset U$ is another open set, then $\cH^1(U) \subset \cH^1(V)$. 

We also define the space $\cH^1_0(U)$ as the closure in $\cH^1(U)$ of the space of functions $f \in \cC^\infty_c(U)$   such that $\norm{f}_{\cH^1(U)}$ is finite. Notice in particular that, in stark contrast with functions in $\cH^1(U)$, a function in $\cH^1_0(U)$ does not depend on $\mu \mres U^c$. In the notation of Lemma~\ref{lem:Projection}, when $f \in \cH^1_0(U)$, certain compatibility conditions between the functions $(f_n)_{n \in \N}$ also have to be satisfied. If $V \subset U$ is another open set, we have that $\cH^1_0(V) \subset \cH^1_0(U)$ (notice that the inclusion is in the opposite direction to that for $\cH^1$ spaces). We also have the following result.
\begin{lemma}\label{lem:Integral0}
For every bounded open set $U \subset \Rd$ with Lipschitz boundary and $f \in \cH^1_0(U)$, we have 
\begin{align}\label{eq:Integral0}
\Er\Ll[\int_U \nabla f(\mu, x) \, \d \mu(x)\Rr] = 0.
\end{align}
\end{lemma}
\begin{proof}
By density, we can assume that $f \in \cC^{\infty}_c(U)$. We use the functions $(f_n)_{n \in \N}$ appearing in Lemma~\ref{lem:Projection}; moreover, since $f(\mu)$ does not depend on $\mu \mres U^c$, we simply write $f_n(x_1,\ldots, x_n)$ in place of $f_n(x_1,\ldots, x_n, \mu \mres U^c)$. For every $k \in \{1,\ldots, d\}$, we have
\begin{align*}
\Er\Ll[\int_U \partial_k f(\mu, x) \, \d \mu(x)\Rr] = \sum_{n=1}^{\infty}\Pr[\mu(U) = n] \sum_{i=1}^n \fint_{U^n} \e_k \cdot \nabla_{x_i} f_n(x_1, \cdots, x_n) \, \d x_1 \cdots \d x_n.
\end{align*}
We use Green's formula for the integral $\e_k \cdot \nabla_{x_i} f(x_1, \cdots, x_n)$ with respect to $x_i$
\begin{align*}
\int_{U} \e_k \cdot \nabla_{x_i} f_n(x_1, \cdots, x_n) \, \d x_i = \int_{\partial U} f_n(x_1, \cdots, x_n) \e_k \cdot \mathbf{n}(x_i) \, \d x_i,
\end{align*}
where $\mathbf{n}(x_i)$ is the unit outer normal. Since $f \in \cC^{\infty}_c(U)$, the quantity $f_n(x_1,\ldots, x_n)$ remains constant when $x_i$ moves along the boundary $\partial U$. Denoting this constant (which depends on $(x_j)_{j \neq i}$) by $c$, we apply once again Green's formula to get
\begin{align*}
\int_{U} \e_k \cdot \nabla_{x_i} f_n(x_1, \cdots, x_n) \, \d x_i =\int_{\partial U} c \e_k \cdot \mathbf{n}(x_i) \, \d x_i = \int_{U} \e_k \cdot \nabla_{x_i} c \, \d x_i = 0. 
\end{align*}
This proves the desired result.
\end{proof}

\subsection{Localization operators}
\label{subsubsec:LocReg}
We now introduce families of operators that allow to localize a function defined on $\mmd(\Rd)$. We state some properties of these operators without proof, and refer to \cite[Section 4.1]{gu2020decay} for more details. 

Recall that for every $s > 0$, we write by $Q_s := \Ll(-\frac{s}{2}, \frac{s}{2}\Rr)^d$. We denote the closure of the cube $Q_s$ by $\bar Q_s$, and define $\A_s f := \Er[f \, \vert \, \mcl F_{\bar Q_s}]$. For any ${f \in \cL^2}$, the process $(\A_s f)_{s \geq 0}$ is a c\`adl\`ag $\cL^2$-martingale with respect to $(\mmd(\Rd),(\mcl F_{\bar Q_s})_{s \geq 0}, \Pr)$. We denote the jump at time $s$ by
\begin{equation*}  
\Delta_s(\A f) := \A_{s}f - \A_{s-}f = \A_{s}f - \lim_{t < s, t \to s} \A_{t}f .
\end{equation*}
We can have $\Delta_s(\A f) \neq 0$ only on the event where the support of the measure $\mu$ intersects the boundary~$\dr Q_s$. The bracket process $([\A f]_s)_{s \geq 0}$ is defined by 
\begin{align}\label{eq:Bracket}
[\A f]_s := \sum_{0\leq \tau \leq s} \Delta_\tau (\A f).
\end{align}
We have that $((\A_s f)^2 - [\A f]_s)_{s \geq 0}$ is a martingale with respect to $(\mcl F_{\bar Q_s})_{s \geq 0}$. 

Notice that the operator $\A_s$ can be interpreted as an averaging of the variable $\mu \mres \ov{Q}_s^c$, keeping $\mu \mres \bar Q_s$ fixed. As a consequence, for every open set $Q_s \subset U$, if $f \in \cH^1(U)$ and $x \in Q_s \cap \supp(\mu)$, there is no ambiguity in considering the quantity $\A_s(\nabla f)(\mu,x)$. Moreover,
\begin{align}\label{eq:Commute}
    \nabla \A_s f(\mu, x) = \A_s (\nabla f)(\mu, x),
\end{align}
and $\A_s f$ belongs to $\cH^1(Q_s)$, by Jensen's inequality; see \Cref{prop:AppendixLocalization} for details. However, in general, this function does not belong to $\cH^1_0(\Rd)$, or any other $\cH^1_0$ space. This comes from the fact that the function $\A_s f$ may be discontinuous as a particle enters or leave $\bar Q_s$. To solve this problem, we regularize this conditional expectation in the following way. For any $s, \epsilon > 0$, we define 
\begin{align}
\label{eq:defA}
\A_{s,\epsilon} f := \frac{1}{\epsilon} \int_{0}^{\epsilon} \A_{s + t} f \, \d t.
\end{align}
As above, for every open set $U$ containing $Q_{s+\ep}$, $f \in \cH^1(U)$, and ${x \in Q_{s+\ep} \cap \supp(\mu)}$, the quantity $\A_{s,\ep}(\nabla f)(\mu,x)$ is well-defined. Irrespectively of the position of the point $x \in \supp(\mu)$, the gradient of $\A_{s,\ep} f$ can be calculated explicitly. Indeed, writing ${\tau(x) := \inf\{r \in \R \, : \ x \in Q_r\}}$, and $\n(x)$ for the outer unit normal to $Q_{\tau(x)}$ at the point $x$, we have
\begin{align}
\label{eq:ADerivative}
\nabla \A_{s,\epsilon} f(\mu, x) = \left\{ 
	\begin{array}{ll}
         \A_{s,\epsilon} \Ll(\nabla f\Rr)(\mu, x) & \text{if } x \in Q_{s};\\
         \frac{1}{\epsilon} \int_{\tau(x)-s}^{\epsilon} \A_{s+t} \Ll(\nabla f(\mu, x)\Rr)\, \d t - \frac{\n(x)}{\epsilon}\Delta_{\tau(x)} (\A f)  & \text{if } x \in \Ll( {Q}_{s+\epsilon} \backslash Q_{s} \Rr);\\
         0 & \text{if } x \in Q_{s+\epsilon}^c.
         \end{array} \right. 
\end{align}
Recalling that $Q_{s+\epsilon} \subset U$, one can check that $\A_{s,\epsilon} f \in \cH^1_0(U)$. Similarly, one can define another regularized localization operator $\tilde{\A}_{s,\epsilon}$ 
\begin{align}\label{eq:Atilde}
\tilde{\A}_{s,\epsilon} f :=  \frac{2}{\epsilon^2} \int_{0}^{\epsilon} (\epsilon - t) \A_{s+ t} f \, \d t,
\end{align}
which can be obtained by applying $\A_{s,\ep}$ twice: $\tilde{\A}_{s,\epsilon} = \A_{s,\epsilon} \circ \A_{s,\epsilon}$. We have the identity 
\begin{align}\label{eq:AtildeIdentity}
\Er[(\A_{s,\epsilon} f)^2] = \Er[f (\tilde{\A}_{s,\epsilon} f)] = \Er\Ll[\frac{2}{\epsilon^2} \int_{0}^{\epsilon} (\epsilon - t) (\A_{s+ t} f)^2 \, \d t\Rr].
\end{align}
The operator $\tilde{\A}_{s,\epsilon}$ satisfies properties similar to those of $\A_{s,\epsilon}$,  and we have 
\begin{multline}
\label{eq:AtildeDerivative}
\nabla \tilde{\A}_{s,\epsilon}f(\mu, x) =
\\
\left\{ \begin{array}{ll}
         \tilde{\A}_{s,\epsilon} \Ll(\nabla f(\mu, x)\Rr) & x \in Q_{s};\\
         \frac{2}{\epsilon^2} \Ll(\int_{\tau(x)-s}^{\epsilon}(\epsilon - t)\A_{s+t} \Ll(\nabla f(\mu, x)\Rr)\, \d t - (s+ \epsilon  - \tau(x))\Delta_{\tau(x)} (\A f) \n(x)\Rr) & x \in \Ll( {Q}_{s+\epsilon} \backslash Q_{s} \Rr);\\
         0 & x \in Q_{s+\epsilon}^c.\end{array} \right. 
\end{multline}

\section{Functional inequalities}
\label{sec:func}

The goal of this section is to derive functional inequalities that will be fundamental to the proof of our main result. The first crucial estimate is a multiscale Poincar\'e inequality, see Proposition~\ref{prop:Multi}. This inequality is an improvement over the standard Poincar\'e inequality that substitutes the $L^2$ norm of the gradient of the function of interest by a weighted sum of spatial averages of this gradient. It has a structure comparable to that of $\|u\|_{L^2} \lesssim \|\nabla u\|_{H^{-1}}$, where we moreover decompose the $H^{-1}$ norm into a series a scales, in analogy with the standard definition of Besov spaces, or the equivalent definition of $H^{-1}$ norm in terms of spatial averages, see for instance \cite[Appendix~D]{AKMbook}. The proof of this estimate is based on an $H^2$ estimate for solutions of ``$-\Delta u = f$'', with ``$\Delta$'' being the relevant Laplacian adapted to our setting; see Proposition~\ref{prop:H2}. 

The second crucial functional inequality derived here is a Caccioppoli inequality, see Proposition~\ref{prop:Caccioppoli}. In the standard elliptic setting, the Caccioppoli inequality allows to control the $L^2$ norm of the gradient of a solution by the $L^2$ norm of the function itself, on a larger domain; it can thus be thought of as a reverse Poincar\'e inequality for solutions. In our context, we are not able to prove such a strong estimate, but prove instead a weaker version of this inequality that allows to control the $\cL^2$ norm of the gradient of a solution by the $\cL^2$ norm of the function itself, plus \emph{a fraction} of the $\cL^2$ norm of the gradient on a larger domain. 

For every $k \le n \in \N$, we define $\Z_{n,k} := 3^k \Zd \cap \cu_n$. Up to a set of null measure, the family $(z + \cu_k)_{z \in \Z_{n,k}}$ forms a partition of $\cu_n$. For any $y \in \Rd$, we write $\cu_n(y)$ to denote the unique cube containing $y$ that can be written in the form $z + \cu_n$ for some $z \in 3^n \Zd$. This is well-defined except for some $y$'s in a set of null measure; we can decide on an arbitrary convention for these remaining cases. We also write $\Z_{n,k}(y) := 3^k \Zd \cap \cu_n(y)$. 

The following ``multiscale spatial filtration'' will be useful in the rest of the paper: for every $n,k \in \N$ with $k \le n$, and $y \in \Rd$, we define the $\sigma$-algebra $\G_{n,k}^y$ by
\begin{align}\label{eq:FilG}
\G_{n,k}^y := \sigma\Ll( \{\mu(z + \cu_k)\}_{z \in \Z_{n,k}(y)} ,\mu \mres (\Rd \backslash \cu_n(y)) \Rr).
\end{align}
We use the shorthand $\G_{n,k} := \G_{n,k}^0$ and $\G_n := \G_{n,n}$. One can verify that, for every $n,n',k,k' \in \N$ and $y,y' \in \Rd$, 
\begin{align}\label{eq:FilGInclusion}
n \leq n', k \leq k' \text{ and } \cu_n(y) \subset \cu_{n'}(y')  \quad \implies \quad \G^{y'}_{n',k'} \subset \G^y_{n,k}.
\end{align}
We also define the analogue of $\G_n$ for a general Borel set $U \subset \Rd$ as
\begin{align}\label{eq:FilGU}
\G_{U} := \sigma\Ll( \mu(U) ,\mu \mres (\Rd \backslash U) \Rr).
\end{align}
The condition $\Er[f \vert \G_U] = 0$ will appear many times in this paper, usually in the context of centering a function in $\cH^1(U)$. Using the functions $(f_n)$ defined defined in \Cref{lem:Projection}, we can rewrite the condition $\Er[f \vert \G_U] = 0$ as: for every $n \in \N$ and $\Pr$-almost every $\mu \in \mmd(\Rd)$,
\begin{align}\label{eq:ConditionGU}
\int_{U^n} f_n(x_1, \cdots, x_n, \mu \mres U^c) \, \d x_1 \cdots \, \d x_n =0.   
\end{align}

\subsection{Poincar\'e inequality}\label{subsec:Poincare}
We present two types of Poincar\'e inequalities: one for the space $\cH^1_0(U)$, and one for the space $\cH^1(U)$. We first state an elementary version for product spaces and functions in the standard Sobolev $H^1$ space. The proof is classical and can be found for instance in \cite[Theorem~13.36 and Proposition~13.34]{firstcourse}. For any bounded Borel set $U \subset \Rd$, we write $\diam(U)$ to denote the diameter of $U$, and for every $f \in L^1(U)$, we denote the Lebesgue integral of $f$, normalized by the Lebesgue measure of $U$, by
\begin{equation*}  
\fint_U  f := |U|^{-1} \int_U f.
\end{equation*}

\begin{proposition}[Poincar\'e inequality in classical Sobolev spaces]
There exists a constant $C(d) < \infty$ such that for every bounded convex open set $U \subset \Rd$, $n \in \N$, and $f \in H^1(U^n)$, we have
\label{prop:Poincare0}
\begin{align}
\fint_{U^n} \Ll(f - \Ll(\fint_{U^n} f \Rr)\Rr)^2 \leq C \diam(U)^2 \sum_{i=1}^n \fint_{U^n} \vert \nabla_{x_i} f \vert^2 .
\end{align}
\end{proposition}
A direct application of \Cref{prop:Poincare0} gives the following proposition.
\begin{proposition}[Poincar\'e inequality in $\cH^1(U)$]
\label{prop:Poincare2}
There exists a constant ${C(d) < \infty}$ such that for every bounded convex open set and $f \in \cH^1(U)$, we have
\begin{align}\label{eq:Poincare2}
\Er\Ll[(f - \Er[f \, \vert \, \G_U])^2\Rr] \leq  C \diam(U)^2 \Er\Ll[\int_{U} \vert\nabla f \vert^2 \, \d \mu\Rr].
\end{align}
\end{proposition}
\begin{proof}
Without loss of generality, we may assume that $\Er[f \vert \G_U] = 0$; subtracting $\Er[f \vert \G_U]$ from $f$ does not change the right side of \cref{eq:Poincare2}. We use the functions $(f_n)$ from \Cref{lem:Projection},  and recall that since ${\Er[f \vert \G_U] = 0}$, we have that every function $f_n$ is centered; see \cref{eq:ConditionGU}. We can apply \Cref{prop:Poincare0} to every $f_n$: for a constant $C < \infty$ independent of $n$, we have
\begin{multline*}
\fint_{U^n} \vert f_n(x_1, \cdots, x_n, \mu \mres U^c)\vert^2 \, \d x_1 \cdots \d x_n \\
\leq C \diam(U)^2 \sum_{i=1}^n \fint_{U^n} \vert \nabla_{x_i} f_n(x_1, \cdots, x_n,\mu \mres U^c) \vert^2  \, \d x_1 \cdots \d x_n.
\end{multline*}
We then sum over $n$ and take the expectation to obtain the result.
\end{proof}

Functions in the space $\cH^1_0(U)$ enjoy certain continuity properties as particles enter and leave the domain $U$. For this reason, it suffices to center the function by its mean value to have a Poincar\'e inequality. 
\begin{proposition}[Poincar\'e inequality in $\cH^1_0(U)$]
There exists a constant $C(d) < \infty$ such that for every bounded open set $U\subset \R^d$, and every $f \in \cH^1_0(U)$, 
\label{prop:Poincare1}
\begin{equation}\label{eq:Poincare1}
\Er\Ll[(f - \Er[f])^2\Rr] \le C \,\diam(U)^2 \, \Er\Ll[ \int_{U} |\nabla f|^2 \, \d \mu\Rr].
\end{equation}
\end{proposition}
\begin{proof}
Without loss of generality, we assume that $\Er [f]=0$. By density, we may restrict to $f\in \cC^\infty_c(U)$.
Applying \cite[Theorem 18.7]{bookPoisson} to $f$, we have that
\begin{align*}
\Er \Ll[ f^2 \Rr] \leq  \rho \int_{\R^d} \Er \Ll[ (f(\mu + \delta_x ) - f( \mu) )^2 \Rr] \d x.
\end{align*}
By the Fubini-Tonelli theorem, and since $f$ is $\mathcal{F}_U$-measurable, this reduces to
\begin{align*}
\Er \Ll[ f^2 \Rr] \leq  \rho  \Er \Ll[ \int_{U} (f(\mu + \delta_x) - f(\mu) )^2 \d x \Rr].
\end{align*}
To establish Proposition \ref{prop:Poincare1}, it thus only remains to show that 
\begin{align}\label{Poi.1}
 \Er \Ll[ \int_{U} (f(\mu + \delta_x) - f(\mu) )^2 \d x \Rr] \leq \frac{C(d)}{\rho} \, \Er \Ll[ \int_U |\nabla f|^2 \d\mu \Rr].
\end{align}
We recall that
\begin{multline}
\label{Poi.2}
\int_{U} \Er \Ll[ (f(\cdot + \delta_x) - f(\cdot) )^2 \Rr] \d x \\
=\sum_{n\in \N} \P(\mu(U)=n) \fint_{U^n} \Ll(\int_U |f_{n+1}(x_1, \cdots, x_n, x) - f_{n}(x_1, \cdots, x_n)|^2 \d x \Rr)  \d x_1\cdots \d x_n,
\end{multline}
where we used the notation (similar but simpler than in \Cref{lem:Projection})
\begin{equation}\label{n.dimensional}  
f_n(x_1,\ldots, x_n) := f \Ll( \sum_{k = 1}^n \de_{x_k} \Rr), \ \ \ \ x_1, \cdots, x_n \in U.
\end{equation}
Let $n\in \N$ be fixed. Since $f\in \cC^\infty_c(U)$, for every $\bar{x} \in \partial U$ we have that
$$
f_n( x_1, \cdots, x_n) = f_{n+1}(x_1, \cdots, x_n, \bar x).
$$
That is, for every $x_1, \cdots, x_n \in U^n$, the (smooth) function 
$$
G: U \to \R, \ \ \ \ \ \ G(\cdot) := f_{n+1}(x_1, \cdots, x_n,  \cdot ) - f_{n}(x_1, \cdots, x_n)
$$
belongs to the (standard) Sobolev space $H^1_0(U)$. We may thus apply the standard Poincar\'e inequality for functions in $H^1_0(U)$ to infer that 
\begin{align*}
\int_U |f_{n+1}(x_1, \cdots, x_n, x) -& f_{n}(x_1, \cdots, x_n)|^2 \d x \\
& \leq C(d) \diam (U)^2 \int_U |\nabla_{x} f_{n+1}(x_1, \cdots, x_n, x)|^2 \d x.
\end{align*}
Inserting this into \eqref{Poi.2}, using that $\P(\mu(U)=n)= e^{-\rho |U|} \frac{(\rho |U|)^n}{n!}$ and relabelling $n+1$ as $n$, yields that
\begin{align*}
\int_{U}& \Er \Ll[ (f(\cdot + \delta_x) - f(\cdot) )^2 \Rr] \d x\\
& \leq \frac{C(d)}{\rho} \diam(U)^2 \sum_{n\in \N} \P(\mu(U)=n) n \fint_{U^{n}} |\nabla_{x_n} f_{n}(x_1, \cdots, x_n) |^2 \d x_1\cdots \,  \d x_n.
\end{align*}
To establish \eqref{Poi.1} from this, it only remains to observe that, by definition \eqref{n.dimensional} each function $f_{n}$ is invariant under permutations, we have
\begin{align*}
\fint_{U^{n}} |\nabla_{x_1} f_{n}|^2 = \fint_{U^{n}} |\nabla_{x_i} f_{n} |^2  \ \ \ \text{for all $i= 1, \cdots, n$.}
\end{align*}
This concludes the proof of \eqref{Poi.1} and establishes Proposition \ref{prop:Poincare1}.
\end{proof}

\subsection{\texorpdfstring{$\cH^2$}{H2} estimate for the homogeneous equation}

When the diffusion matrix $\a$ is a constant, the solutions to the corresponding equation have a better regularity than otherwise, and in particular,  the following $\cH^2$ estimate holds. One can define the function with higher derivative iteratively: for $x, y \in \supp(\mu), x \neq y$ 
\begin{align*}
\partial_j \partial_k f(\mu, x, y) := \lim_{h \to 0} \frac{\partial_k f(\mu - \delta_y + \delta_{y + h \e_j}, x) - \partial_k f(\mu, x)}{h},
\end{align*}
and for the case $x=y$, it makes sense as 
\begin{align*}
\partial_j \partial_k f(\mu, x, x) := \lim_{h \to 0} \frac{\partial_k f(\mu - \delta_x + \delta_{x + h \e_j}, x + h \e_j) - \partial_k f(\mu, x)}{h}.
\end{align*}
We also denote by $\nabla^2 f(\mu, x, y)$ the matrix $\{\partial_j \partial_k f(\mu, x, y)\}_{1 \leq j,k \leq d}$, and its norm is defined as 
\begin{align*}
    \vert \nabla^2 f(\mu, x, y) \vert^2 := \sum_{1 \leq j, k \leq d} \vert \partial_j \partial_k f(\mu, x, y)\vert^2.
\end{align*}
\begin{proposition}[$\cH^2$ estimate]\label{prop:H2}
Let $f \in \cL^2$, and let $u \in \cH^1(Q_r)$ solve ``$- \Delta u = f$'' in the sense that for any $v \in \cH^1(Q_r)$,
\begin{align}\label{eq:H2Laplace}
\Er\Ll[\int_{Q_r} \nabla u(\mu, x) \cdot \nabla v(\mu, x) \, d\mu \Rr] = \Er[f v].
\end{align}
We have the $\cH^2(Q_r)$ estimate 
\begin{align}\label{eq:H2}
\Er\Ll[\int_{(Q_r)^2}   \vert \nabla^2 u(\mu,x,y)\vert^2 \, \d \mu(x) \d \mu(y) \Rr] \leq \Er[f^2].
\end{align}
\end{proposition}
\begin{remark}
By testing \cref{eq:H2Laplace} with $v = \Ind{\mu(Q_r) = n, \mu \mressmall Q_r^c (V) = m}$, we see that $f$ has to satisfy $\Er[f \, \vert \,  \G_{Q_r}] = 0$ as a condition of compatibility. 
\end{remark}
\begin{proof}[Proof of Proposition~\ref{prop:H2}]
Although this is not really part of the statement, we start by showing that for every $f \in \cL^2$ satisfying the compatibility condition $\E_\rho [f \, \vert \, \G_{Q_r}] = 0$, there exists a solution $u$ to \cref{eq:H2Laplace}, and we will show its link with the classical elliptic equation. At first, we notice that the problem can be studied on the space of functions 
\begin{align*}
W = \{g \in \cH^1(Q_r) : \Er[g \, \vert \,  \G_{Q_r}] = 0\}.
\end{align*}
Because for a general function $v \in \cH^1(Q_r)$, $\Er[v \vert \G_{Q_r}]$ can be seen as a constant in \cref{eq:H2Laplace}: its derivative is $0$ so the left-hand side of \cref{eq:H2Laplace} is $0$. For the right-hand side, we have 
\begin{align*}
    \Er[f \Er[v \vert \G_{Q_r}]] = \Er\Ll[\Er[f \vert \G_{Q_r}] \Er[v \vert \G_{Q_r}]\Rr] = 0.
\end{align*}
Thus when applying the operation $v \mapsto v - \Er[v \vert \G_{Q_r}]$, we do not change \cref{eq:H2Laplace} and we can restrict the Laplace equation on $W$. Moreover, with the notation in \Cref{lem:Projection}, $\Er[v \vert \G_{Q_r}]=0$ implies every $v_n$ is centered; see \cref{eq:ConditionGU}.

Secondly, we test \cref{eq:H2Laplace} with $v \Ind{\mu(Q_r)=n}\Ind{(\mu \mres Q_r^c)(V)=m}$, which is conditioning the number of particles $\mu(Q_r)$, and also $(\mu \mres Q_r^c)(V)$ for some bounded Borel set $V$ as an environment outside $Q_r$. Then for arbitrary choices of $n, m, V$, in fact we have a classical elliptic equation using the canonical projection \Cref{lem:Projection} 
\begin{multline}\label{eq:H2LaplaceRNd}
\int_{(Q_r)^n} \sum_{k=1}^n \nabla_{x_k}u_n(x_1, \cdots, x_n, \mu\mres Q_r^c) \cdot \nabla_{x_k}v_n(x_1, \cdots, x_n, \mu\mres Q_r^c) \, \d x_1 \cdots \d x_n \\
= \int_{(Q_r)^n} f_n(x_1, \cdots, x_n, \mu\mres Q_r^c)  v_n(x_1, \cdots, x_n, \mu\mres Q_r^c) \, \d x_1 \cdots \d x_n.
\end{multline}
Thus the solution $u$ can be described as follows: we sample the environment outside $Q_r$ and fix the number of particle $\mu(Q_r) = n$ at first, then solve the classical elliptic equation in $H^1(\R^{nd})$ with mean zero. Finally we combine all the $u_n$ and this gives the solution of \cref{eq:H2Laplace}. In other words, the statement of \cref{eq:H2Laplace} can be reinforced as   
\begin{align*}
\forall v \in W , \qquad \Er\Ll[\int_{Q_r} \nabla u(\mu, x) \cdot \nabla v(\mu, x) \, d\mu  \ \Bigg \vert \  \G_{Q_r} \Rr] = \Er\Ll[f v \, \big\vert \,  \G_{Q_r}\Rr].
\end{align*}
We now turn to study the $\cH^2$ estimate. We apply the classical $H^2(\R^{nd})$ estimate for \cref{eq:H2LaplaceRNd} (see for instance \cite[Lemma~B.19]{AKMbook} and its proof) 
\begin{multline}\label{eq:H2RNd}
\int_{(Q_r)^n} \sum_{1 \leq i , j \leq n} \vert \nabla_{x_i} \nabla_{x_j} u_n\vert^2(x_1, \cdots, x_n, \mu \mres Q_r^c) \, \d x_1 \cdots \d x_n \\ \leq  \int_{(Q_r)^n} \vert f_n \vert^2 (x_1, \cdots, x_n, \mu \mres Q_r^c) \, \d x_1 \cdots \d x_n,
\end{multline}
Taking the expectation of \cref{eq:H2RNd} then gives the result. 
\end{proof}

\subsection{Multiscale Poincar\'e inequality}
For cubes of size $3^n$, the Poincar\'e inequalities derived in the previous subsection (say with $k = n$ in Proposition \ref{prop:Poincare2}) have a right-hand side that scales like $3^{2n}$. In this subsection, we derive a multiscale version of the Poincar\'e inequality, that aims to improve upon this scaling, provided that some local average of the gradient of the function is not too large. 
We recall that the multiscale spatial filtration $\G_{n,k}^y$ is defined in \cref{eq:FilG}. For every $k \le n \in \N$, $x, y \in \Rd$ such that $x \in \cu_n(y)$, open set $U$ containing $\cu_k(x)$, and $f \in \cH^1(U)$, the following quantity is well defined
\begin{align}\label{eq:defSnk}
(\S^y_{n,k} \nabla f)(\mu,x) := \Er\Ll[\fint_{\cu_k(x)} \nabla f \, \d\mu \ \Bigg\vert \ \G_{n,k}^y  \Rr],
\end{align}
where we use the notation, for every Borel set $V$ such that $\mu(V) \in (0, \infty)$ and function $g$ defined on $\supp(\mu) \cap V$,
\begin{equation}  
\label{eq:defAverageInt}
\fint_V g \, \d \mu := 
\frac{1}{\mu(V)} \int_V g \, \d \mu ,
\end{equation}
and for definiteness, we also set $\fint_V g \, \d \mu = 0$ if $\mu(V) = 0$. 
We use the shorthand notation $\S_{n,k} := \S^0_{n,k}$ and $\S_n := \S_{n,n}$. This operator has a convenient spatial martingale structure, as displayed in the following lemma.

\begin{lemma}
[Martingale structure for $\S_{n,k}$]
For every ${n, n', k, k' \in \mathbb{N}}$, $y, y' \in \Rd$ satisfying
\begin{align*}
n \leq n', \quad k\leq k',  \quad \cu_n(y) \subset \cu_{n'}(y'), 
\end{align*}
every $x \in \cu_{k'}(y')$, and $f \in \cH^1(\cu_{n'}(y'))$, we have
\begin{align}\label{eq:SnkMt}
\S_{n',k'}^{y'} \nabla f(\mu, x) = \Er\Ll[\fint_{\cu_{k'}(x)} (\S_{n,k}^y \nabla f) \, \d\mu \ \Bigg\vert \ \G_{n',k'}^{y'}  \Rr]. 
\end{align}
\end{lemma}
\begin{proof} 
The key observation is \cref{eq:FilGInclusion}, stating that $\G_{n,k}^{y}$ is a finer $\sigma$-algebra than $\G_{n',k'}^{y'}$, so that
\begin{align*}
& \S_{n',k'}^{y'} \nabla f (\mu,x) 
\\
& \qquad = \Er\Ll[\frac{1}{\mu(\cu_{k'}(x))}  \int_{\cu_{k'}(x)} \nabla f \, \d\mu \ \Bigg\vert \ \G_{n',k'}^{y'}  \Rr]\\
& \qquad = \Er\Ll[\sum_{z \in \Z_{n,k} \cap \cu_{k'}(x)}\frac{\mu(z+\cu_k)}{\mu(\cu_{k'}(x))}  \Er\Ll[\frac{1}{\mu(z+\cu_k)}\int_{\cu_{k}(z)} \nabla f \, \d\mu \ \Bigg\vert \ \G_{n,k}^y \Rr]\ \Bigg\vert \ \G_{n',k'}^{y'}  \Rr].
\end{align*}
By the definition of $\S_{n,k}^y \nabla f(\mu, z)$, we obtain
\begin{align*}
\S_{n',k'}^{y'} \nabla f (\mu,x) &= \Er\Ll[\sum_{z \in \Z_{n,k} \cap \cu_{k'}(x)}\frac{\mu(z+\cu_k)}{\mu(\cu_{k'}(x))} (\S_{n,k}^y\nabla f)(\mu, z) \ \Bigg\vert \ \G_{n',k'}^{y'}  \Rr] \\
&= \Er\Ll[\frac{1}{\mu(\cu_{k'}(x))} \int_{\cu_{k'}(x)} \S_{n,k}^y \nabla f \, \d \mu \ \Bigg\vert \ \G_{n',k'}^{y'}  \Rr]. 
\end{align*}
This is \cref{eq:SnkMt}. 
\begin{figure}
\includegraphics[scale=0.5]{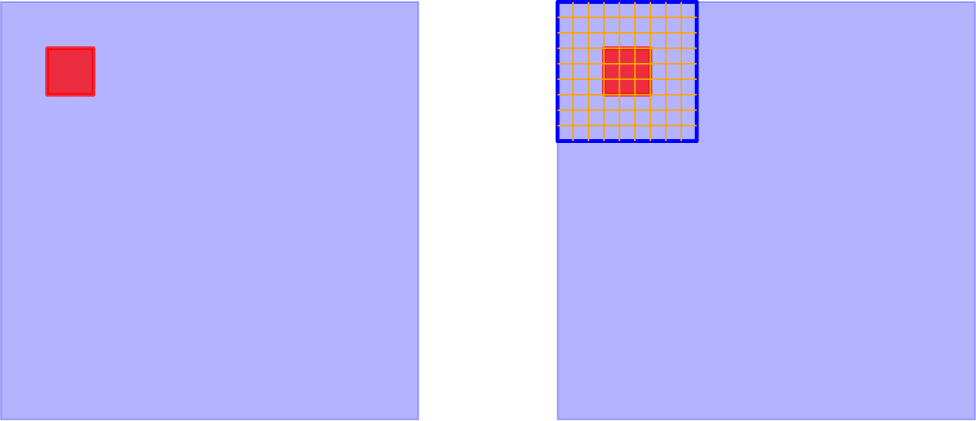}
\caption{The largest cube on this figure is $\square_{n'}(y')$. The operator $\S_{n',k'}^{y'}$ computes the spatial average in every subcubeof size $3^{k'}$, for example the cube in red in the image. We can apply at first the operator $\S_{n,k}^{y}$, which works on the finer scales $3^k$ and $3^n$, represented by the cubes with orange and blue boundaries respectively.}
\end{figure}
\end{proof}

To prepare further for the  multiscale Poincar\'e inequality, we also give the following explicit expression for $\S_{n,k}^{y} \nabla f$. We use the notation 
\begin{align*}
\fint_{(z_i + \cu_k)_{1 \leq i \leq N}} := \fint_{z_i + \cu_k} \cdots \fint_{z_N + \cu_k}.
\end{align*}
\begin{lemma}
Using the notation of \Cref{lem:Projection} with ${\mu \mres \cu_n(y) = \sum_{i=1}^N \delta_{x_i}}$, for any $x \in \cu_n(y)$ and any $f \in \cH^1(\cu_n(y))$, we have
\begin{multline}\label{eq:SnkExplicit}
    (\S_{n,k}^{y} \nabla f)(\mu, x) \prod_{z \in \Z_{n,k}(y)}\Ind{\mu(z + \cu_k) = N_z} \\
= \frac{1}{\mu(\cu_k(x))} \sum_{j : x_j \in \cu_k(x)} \fint_{(z_i + \cu_k)_{1 \leq i \leq N}} \nabla_{x_j} f_N(\cdot, \mu \mres \cu_n^c) \prod_{z \in \Z_{n,k}(y)}\Ind{\mu(z + \cu_k) = N_z} ,
\end{multline}
with $N = \sum_{z \in \Z_{n,k}(y)} N_z$ and $\{z_i\}_{1 \leq i \leq N}$ any fixed sequence such that
\begin{align}
\label{eq:cond_zi}
\forall z \in \Z_{n,k}(y), \qquad    \vert \{i \in \{1,\ldots, N\} \ : \  z_i = z\}\vert = N_z.
\end{align}
Moreover, for every $j,j'$ such that $x_j, x_{j'} \in \cu_k(x)$, we have 
\begin{align}\label{eq:SnkSymmetry}
     \fint_{(z_i + \cu_k)_{1 \leq i \leq N}} \nabla_{x_j} f_N(\cdot , \mu \mres \cu_n^c)  = \fint_{(z_i + \cu_k)_{1 \leq i \leq N}} \nabla_{x_{j'}} f_N(\cdot, \mu \mres \cu_n^c).
\end{align}
\end{lemma}
\begin{proof}
Without loss of generality, we set $y = 0$. Then let $N = \sum_{z \in \Z_{n,k}} N_z$ and we use the canonical projection 
\begin{align*}
&(\S_{n,k} \nabla f)(\mu, x)\prod_{z \in \Z_{n,k}(y)}\Ind{\mu(z + \cu_k) = N_z} \\
= & \frac{1}{\mu(\cu_k(x))}\Er\Ll[ \sum_{x_j \in \cu_k(x)} \nabla_{x_j} f_N(\cdot, \mu \mres \cu_n^c)  \prod_{z \in \Z_{n,k}}\Ind{\mu(z + \cu_k) = N_z} \Bigg\vert \G_{n,k}\Rr].
\end{align*}
The key point is to write $ \prod_{z \in \Z_{n,k}}\Ind{\mu(z + \cu_k) = N_z}$ with respect to $\{x_i\}_{1 \leq i \leq N}$ such that ${\mu \mres \cu_n = \sum_{i=1}^N \delta_{x_i}}$. Let $\{z_i\}_{1 \leq x_i \leq N}$ be any fixed sequence so that every $z$ in $\Z_{n,k}$ appears exactly $N_z$ times, as displayed in \cref{eq:cond_zi}. We have
\begin{align*}
\prod_{z \in \Z_{n,k}}\Ind{\mu(z + \cu_k) = N_z} = \sum_{\sigma \in S_N} \prod_{i=1}^N\Ind{x_{\sigma(i)} \in z_i + \cu_k},
\end{align*}
where $S_N$ is the symmetric group. Moreover, under $\G_{n,k}$ every permutation has equal probability, and then each $x_i$ is uniformly distributed in the associated cube $z_{\sigma(i)} + \cu_k$. Thus, we have 
\begin{align*}
&(\S_{n,k} \nabla f)(\mu, x)\prod_{z \in \Z_{n,k}}\Ind{\mu(z + \cu_k) = N_z} \\
= & \frac{1}{\mu(\cu_k(x))} \frac{1}{\vert S_N \vert}\sum_{\sigma \in S_N} \fint_{(z_{\sigma(i)} + \cu_k)_{1 \leq i \leq N}} \sum_{x_j \in \cu_k(x)}\nabla_{x_j} f_N(\cdot, \mu \mres \cu_n^c)\prod_{i=1}^N\Ind{x_i \in z_{\sigma(i)} + \cu_k}. 
\end{align*}
Notice that for every $1 \leq i \leq N$, $x_i \in z_{\sigma(i)} + \cu_k$ means $x_{\sigma^{-1}(i)} \in z_{i} + \cu_k$, and $\sum_{x_j \in \cu_k(x)}\nabla_{x_j} f_N(\cdot, \mu \mres \cu_n^c)$ is permutation-invariant. So we have 
\begin{align*}
&\sum_{x_j \in \cu_k(x)}\nabla_{x_j} f_N(x_1, \cdots, x_N, \mu \mres \cu_n^c)\prod_{i=1}^N\Ind{x_i \in z_{\sigma(i)} + \cu_k}\\
=&\sum_{x_j \in \cu_k(x)}\nabla_{x_j} f_N(x_1, \cdots, x_N, \mu \mres \cu_n^c)\prod_{i=1}^N\Ind{x_{\sigma^{-1}(i)} \in z_i + \cu_k} \\
=&\sum_{x_{\sigma^{-1}(j)} \in \cu_k(x)}\nabla_{x_{\sigma^{-1}(j)}} f_N(x_{\sigma^{-1}(1)}, \cdots, x_{\sigma^{-1}(N)}, \mu \mres \cu_n^c)\prod_{i=1}^N\Ind{x_{\sigma^{-1}(i)} \in z_i + \cu_k} \\
= &\sum_{x_{j} \in \cu_k(x)}\nabla_{x_{j}} f_N(x_{1}, \cdots, x_{N}, \mu \mres \cu_n^c)\prod_{i=1}^N\Ind{x_i \in z_i + \cu_k}.
\end{align*}
Therefore, the term for each permutation has the same contribution, and we thus obtain \cref{eq:SnkExplicit}. 

Then we prove \cref{eq:SnkSymmetry}. To avoid possible confusion in the notation, we let $y_j, y_{j'}$ be the $j$-th and $j'$-th coordinates, then we exchange them and use the invariance under permutation of $f_N$, 
\begin{equation}\label{eq:SnkSymmetryStep1}
\begin{split}
\e_k \cdot \nabla_{x_j}f_N(\cdots, y_j, \cdots y_{j'}, \cdots) &= \lim_{h \to 0} \frac{f_N(\cdots, y_j + h \e_k, \cdots y_{j'}, \cdots) - f_N(\cdots, y_j, \cdots y_{j'}, \cdots)}{h} \\
&= \lim_{h \to 0} \frac{f_N(\cdots, y_{j'}, \cdots y_{j}+ h \e_k, \cdots) - f_N(\cdots, y_{j'}, \cdots y_{j}, \cdots)}{h} \\
&= \e_k \cdot \nabla_{x_{j'}}f_N(\cdots, y_{j'}, \cdots y_{j}, \cdots).
\end{split}
\end{equation}
Moreover, the condition $x_j, x_{j'} \in \cu_k(x)$ implies that $z_j = z_{j'}$ and 
\begin{align}\label{eq:SnkSymmetryStep2}
\Ind{y_j \in z_j + \cu_k}\Ind{y_{j'} \in z_{j'} + \cu_k}
 = \Ind{y_{j'} \in z_{j} + \cu_k}\Ind{y_j \in z_{j'} + \cu_k}.
\end{align}
We combine \cref{eq:SnkSymmetryStep1} and \cref{eq:SnkSymmetryStep2} to conclude \cref{eq:SnkSymmetry}.
\end{proof}

We now use the operators $\S_{n,k}^y$ as our locally averaged gradient to obtain the following multiscale Poincar\'e inequality. Notice in particular the factor of $3^k$ inside the sum on the right side of \cref{eq:Multi}, which we aim to leverage upon later by combining this with information on the smallness of $\S_{n,k} \nabla u$ for $k$ close to $n$.
\begin{proposition}[Multiscale Poincar\'e inequality]
There exists a constant ${C(d) < \infty}$ such that for every function $u \in \cH^1(\cu_n)$ satisfying $\Er[u \, \vert \, \G_n] = 0$, we have
\label{prop:Multi}
\begin{align}\label{eq:Multi}
\norm{u}_{\cL^2} \leq C \Ll(\Er\Ll[\int_{\cu_n} \vert\nabla u \vert^2 \, \d \mu \Rr]\Rr)^{\frac{1}{2}} 
 + C \sum_{k=0}^n 3^k \Ll(\Er\Ll[ \int_{\cu_n} \vert \S_{n,k} \nabla u\vert^2 
 \, \d \mu \Rr]\Rr)^{\frac{1}{2}}.
\end{align}
\end{proposition}
\begin{proof}
Let $w \in \cH^1(\cu_n)$ be such that $\Er[w \, \vert \, \G_n] = 0$ and that solves ``$-\Delta w = u$'', in the sense that
\begin{align}\label{eq:MultiLaplace}
\forall v \in \cH^1(\cu_n), \qquad \Er\Ll[\int_{\cu_n} \nabla w \cdot \nabla v \, \d \mu\Rr] = \Er[uv],
\end{align}
and this relation also holds conditionally on $\G_n$:
\begin{align}\label{eq:MultiLaplace2}
\forall v \in \cH^1(\cu_n), \qquad \Er\Ll[\int_{\cu_n} \nabla w \cdot \nabla v \, \d \mu \ \Bigg \vert \ \G_n \Rr] = \Er[uv \, \vert \, \G_n].
\end{align}
Thanks to the condition $\Er[u \, \vert \, \G_n] = 0$, these equations are well-defined; see the proof of \Cref{prop:H2} for a detailed discussion. This proposition asserts that
\begin{align}\label{eq:MultiH2}
\Er\Ll[\int_{(\cu_n)^2}   \vert \nabla^2 w(\mu,x,y)\vert^2 \, \d \mu(x) \d \mu(y) \Rr] \leq  \Er[u^2].
\end{align}
We test \cref{eq:MultiLaplace} with $u$ and write a telescopic sum with $(\S_{n,k}\nabla w)_{0 \leq k \leq n}$ to get
\begin{equation}\label{eq:MultiDecom}
\begin{split}
\Er[u^2] &= \Er\Ll[\int_{\cu_n} \nabla w \cdot \nabla u \, \d \mu\Rr] = \text{\cref{eq:MultiDecom}-a} + \text{\cref{eq:MultiDecom}-b} + \text{\cref{eq:MultiDecom}-c}, \\
\text{\cref{eq:MultiDecom}-a} &= \Er\Ll[\int_{\cu_n} (\nabla w - \S_{n,0}\nabla w) \cdot \nabla u \, \d \mu\Rr], \\
\text{\cref{eq:MultiDecom}-b} &= \sum_{k=0}^{n-1} \Er\Ll[\int_{\cu_n}  (\S_{n,k}\nabla w - \S_{n, k+1}\nabla w) \cdot \nabla u \, \d \mu\Rr], \\
\text{\cref{eq:MultiDecom}-c} &= \Er\Ll[\int_{\cu_n} (\S_{n,n} \nabla w) \cdot \nabla u \, \d \mu\Rr].
\end{split}
\end{equation}
We treat each of these three terms in turn. For \cref{eq:MultiDecom}-a, we use the Cauchy-Schwarz inequality to write
\begin{align*}
\text{\cref{eq:MultiDecom}-a} \leq \Ll(\Er\Ll[\int_{\cu_n} \vert\nabla w - \S_{n,0}\nabla w\vert^2 \, \d \mu\Rr]\Rr)^{\frac{1}{2}} \Ll(\Er\Ll[\int_{\cu_n} \vert \nabla u \vert^2 \, \d \mu\Rr]\Rr)^{\frac{1}{2}}.
\end{align*}
The first term on the right side above can be rewritten as
\begin{align}\label{eq:MultiDecomAPoicare}
\Er\Ll[\int_{\cu_n} \vert\nabla w - \S_{n,0}\nabla w\vert^2 \, \d \mu\Rr] = \Er\Ll[ \sum_{z \in \Z_{n,0}}  \Er\Ll[\int_{z + \cu_0} \vert\nabla w - \S_{n,0}\nabla w\vert^2 \, \d \mu \ \Bigg \vert \ \G_{n,0}\Rr]\Rr].
\end{align}
We use the canonical projection \Cref{lem:Projection} for $w$ with $\mu \mres \cu_n = \sum_{i=1}^N \delta_{x_i}$, and the do the decomposition conditioned on $\G_{n,0}$ that 
\begin{align*}
w(\mu) = \sum_{N=0}^{\infty} \sum_{\substack{\sum_{z \in \Z_{n,0}}N_z = N}}  w_N(x_1, \cdots, x_N, \mu \mres \cu_n^c) \prod_{z \in \Z_{n,0}}\Ind{\mu(z + \cu_0) = N_z}.
\end{align*}
It suffices to study one term $w_N(x_1, \cdots, x_N, \mu \mres \cu_n^c) \prod_{z \in \Z_{n,0}}\Ind{\mu(z + \cu_0) = N_z}$. We can apply \cref{eq:SnkExplicit}: let $\{z_i\}_{1\leq i \leq N}$ be a fixed sequence such that \cref{eq:cond_zi} holds (with $y = 0$ there). For any $x \in \cu_n$ we have
\begin{multline}\label{eq:MultiProjection}
\S_{n,0}\nabla w(\mu, x) \prod_{z \in \Z_{n,0}}\Ind{\mu(z + \cu_0) = N_z} \\ 
= \frac{1}{\mu(\cu_0(x))} \sum_{x_j \in \cu_0(x)} \fint_{(z_i + \cu_0)_{1 \leq i \leq N}} \nabla_{x_j} w_N(\cdot, \mu \mres \cu_n^c) \prod_{z \in \Z_{n,0}}\Ind{\mu(z + \cu_0) = N_z}.
\end{multline}
We apply \cref{eq:MultiProjection} in \cref{eq:MultiDecomAPoicare} and just study the sum over one $z'$ in $\Z_{n,0}$
\begin{align*}
&\Er\Ll[\int_{z' + \cu_0} \vert\nabla w - \S_{n,0}\nabla w\vert^2 \, \d \mu  \prod_{z \in \Z_{n,0}}\Ind{\mu(z + \cu_0) = N_z} \ \Bigg \vert \ \G_{n,0}\Rr] \\
=&\sum_{x_j \in z' + \cu_0} \fint_{(z_i + \cu_0)_{1 \leq i \leq N}}  \Ll\vert \nabla_{x_j} w_N(\cdot, \mu \mres \cu_n^c) - \frac{1}{\mu(z'+ \cu_0)}\sum_{x_{j'} \in z' + \cu_0} \fint_{(z_i + \cu_0)_{1 \leq i \leq N}} \nabla_{x_{j'}} w_N(\cdot, \mu \mres \cu_n^c)\Rr\vert^2 \\
& \qquad \times \prod_{z \in \Z_{n,0}}\Ind{\mu(z + \cu_0) = N_z}.
\end{align*} 
Then we use the symmetry proved in \cref{eq:SnkSymmetry}, that in fact every $\nabla_{x_j} w_N$ has the same contribution for all $x_j \in z' + \cu_0$,
\begin{align*}
&\Er\Ll[\int_{z' + \cu_0} \vert\nabla w - \S_{n,0}\nabla w\vert^2 \, \d \mu  \prod_{z \in \Z_{n,0}}\Ind{\mu(z + \cu_0) = N_z} \ \Bigg \vert \ \G_{n,0}\Rr] \\
=&\sum_{x_j \in z' + \cu_0} \fint_{(z_i + \cu_0)_{1 \leq i \leq N}}  \Ll\vert \nabla_{x_j} w_N(\cdot, \mu \mres \cu_n^c) -  \fint_{(z_i + \cu_0)_{1 \leq i \leq N}} \nabla_{x_{j}} w_N(\cdot, \mu \mres \cu_n^c)\Rr\vert^2  \prod_{z \in \Z_{n,0}}\Ind{\mu(z + \cu_0) = N_z}.
\end{align*}
For the equation above, we can use the Poincar\'e inequality \Cref{prop:Poincare0} because it is centered and every $x_i$ lives uniformly in its associated small cube $z_i + \cu_0$. We remark that the  constant $C$ here is independent of $N$
\begin{align*}
&\Er\Ll[\int_{z' + \cu_0} \vert\nabla w - \S_{n,0}\nabla w\vert^2 \, \d \mu  \prod_{z \in \Z_{n,0}}\Ind{\mu(z + \cu_0) = N_z} \ \Bigg \vert \ \G_{n,0}\Rr] \\
& \qquad \leq  C\sum_{1 \leq i \leq N}\sum_{x_j \in z' + \cu_0} \fint_{(z_i + \cu_0)_{1 \leq i \leq N}}  \Ll\vert \nabla_{x_i} \nabla_{x_j} w_N(\cdot, \mu \mres \cu_n^c) \Rr\vert^2  \prod_{z \in \Z_{n,0}}\Ind{\mu(z + \cu_0) = N_z}.
\end{align*} 
We put this estimate back to \cref{eq:MultiDecomAPoicare}, do the sum over all $z' \in \Z_{n,0}$
\begin{align*}
&\sum_{z' \in \Z_{n,0}}  \Er\Ll[\int_{z' + \cu_0} \vert\nabla w - \S_{n,0}\nabla w\vert^2 \, \d \mu \prod_{z \in \Z_{n,0}}\Ind{\mu(z + \cu_0) = N_z} \ \Bigg \vert \ \G_{n,0}\Rr]\\
& \qquad \leq  C\sum_{1 \leq i,j \leq N} \fint_{(z_i + \cu_0)_{1 \leq i \leq N}}  \Ll\vert \nabla_{x_i} \nabla_{x_j} w_N(\cdot, \mu \mres \cu_n^c) \Rr\vert^2  \prod_{z \in \Z_{n,0}}\Ind{\mu(z + \cu_0) = N_z} \\
& \qquad =C  \Er\Ll[\int_{(\cu_n)^2} \Ll\vert\nabla^2  w(\mu, x, y) \Rr\vert^2 \, \d \mu(x) \d \mu(y) \prod_{z \in \Z_{n,0}}\Ind{\mu(z + \cu_0) = N_z} \ \Bigg \vert \ \G_{n,0}\Rr].
\end{align*}
Finally, we do the expectation and the sum over all ${\prod_{z \in \Z_{n,0}}\Ind{\mu(z + \cu_0) = N_z}}$, and use the $\cH^2$-estimate \cref{eq:MultiH2} to obtain that 
\begin{equation}\label{eq:MultiScale0}
\begin{split}
&\Er\Ll[\int_{\cu_n} \vert\nabla w - \S_{n,0}\nabla w\vert^2 \, \d \mu\Rr] \\
& \qquad \leq  C  \Er\Ll[\int_{(\cu_n)^2} \Ll\vert\nabla^2  w(\mu, x, y) \Rr\vert^2 \, \d \mu(x) \d \mu(y) \Rr] \\
& \qquad \leq  C \Er[u^2],
\end{split}
\end{equation}
and this concludes that 
\begin{align}\label{eq:MultiDecomABound}
\text{\cref{eq:MultiDecom}-a} \leq C (\Er[u^2])^{\frac{1}{2}} \Ll(\Er\Ll[\int_{\cu_n} \vert \nabla u \vert^2 \, \d \mu\Rr]\Rr)^{\frac{1}{2}}.
\end{align}

The term \cref{eq:MultiDecom}-b can be treated similarly. For every $k$, we apply at first the conditional expectation with respect to $\G_{n,k}$ 
\begin{align*}
&\Er\Ll[\int_{\cu_n}  (\S_{n,k}\nabla w - \S_{n, k+1}\nabla w) \cdot \nabla u \, \d \mu\Rr] \\
& \qquad = \sum_{z \in \Z_{n,k}} \Er\Ll[ \Er\Ll[\int_{z+\cu_k}  (\S_{n,k}\nabla w - \S_{n, k+1}\nabla w) \cdot \nabla u \, \d \mu  \ \Bigg \vert \  \G_{n,k}\Rr]\Rr]\\
& \qquad =  \Er\Ll[ \sum_{z \in \Z_{n,k}} \int_{z+\cu_k}  (\S_{n,k}\nabla w - \S_{n, k+1}\nabla w) \cdot (\S_{n,k}\nabla u) \, \d \mu \Rr].
\end{align*}
Then we use the Cauchy-Schwarz inequality to obtain that 
\begin{multline*}
\Er\Ll[\int_{\cu_n}  (\S_{n,k}\nabla w - \S_{n, k+1}\nabla w) \cdot \nabla u \, \d \mu\Rr] \\
\leq  \Ll(\Er\Ll[ \sum_{z \in \Z_{n,k}} \int_{z+\cu_k}  \vert\S_{n,k}\nabla w - \S_{n, k+1}\nabla w \vert^2 \, \d \mu \Rr]\Rr)^{\frac{1}{2}} \Ll(\Er\Ll[\sum_{z \in \Z_{n,k}} \int_{z+\cu_k} \vert \S_{n,k}\nabla u \vert^2 \, \d \mu \Rr]\Rr)^{\frac{1}{2}}.
\end{multline*}
We use the definition in \cref{eq:defSnk} and Jensen's inequality for $\vert \S_{n,k}\nabla w - \S_{n, k+1}\nabla w \vert^2$. For every $z \in \Z_{n,k}$, since $(\S_{n, k+1}\nabla w)(\mu,z)$ is $\G_{n,k}$-measurable, 
\begin{equation}\label{eq:SnkTwoScale}
\begin{split}
\vert \S_{n,k}\nabla w - \S_{n, k+1}\nabla w \vert^2(\mu,z)
& = \Ll(\Er\Ll[ \fint_{z+\cu_k} (\nabla w - \S_{n, k+1}\nabla w) \, \d\mu \ \Bigg\vert \ \G_{n,k} \Rr]\Rr)^2 \\
& \leq  \Er\Ll[ \fint_{z+\cu_k} \vert\nabla w - \S_{n, k+1}\nabla w \vert^2 \, \d\mu \ \Bigg\vert \ \G_{n,k} \Rr].
\end{split}
\end{equation}
Then we sum over all $z \in \Z_{n,k}$, and we can treat it like \cref{eq:MultiDecom}-a and \cref{eq:MultiScale0} with the Poincar\'e inequality in the scale $3^k$ and the $\cH^2$-estimate \cref{eq:MultiH2}, yielding
\begin{align*}
\Er\Ll[ \sum_{z \in \Z_{n,k}} \int_{z+\cu_k}  \vert\S_{n,k}\nabla w - \S_{n, k+1}\nabla w \vert^2 \, \d \mu \Rr] &\leq \Er\Ll[ \int_{\cu_n}  \vert \nabla w - \S_{n, k+1}\nabla w \vert^2 \, \d \mu \Rr] \\
&\leq C3^{2k}\Er[u^2].
\end{align*}
We have thus shown that
\begin{align}\label{eq:MultiDecomBBound}
\text{\cref{eq:MultiDecom}-b} \leq C (\Er[u^2])^{\frac{1}{2}} \Ll(\sum_{k=0}^{n-1} 3^k \Ll(\Er\Ll[\int_{\cu_n} \vert\S_{n,k}\nabla u \vert^2 \, \d \mu \Rr]\Rr)^{\frac{1}{2}}\Rr).
\end{align}

For \cref{eq:MultiDecom}-c, we use \cref{eq:defSnk} and the Cauchy-Schwarz inequality to get that 
\begin{align*}
\text{\cref{eq:MultiDecom}-c} &= \Er\Ll[\int_{\cu_n} (\S_{n,n} \nabla w) \cdot (\S_{n,n} \nabla u) \, \d \mu\Rr] \\
&\leq \Ll(\Er\Ll[\int_{\cu_n} \vert\S_{n,n} \nabla w\vert^2 \, \d \mu\Rr]\Rr)^{\frac{1}{2}} \Ll(\Er\Ll[\int_{\cu_n} \vert \S_{n,n} \nabla u\vert^2 \, \d \mu\Rr]\Rr)^{\frac{1}{2}}.  
\end{align*}
To treat the term $\Er\Ll[\int_{\cu_n} \vert\S_{n,n} \nabla w\vert^2 \, \d \mu\Rr]$, we define the random affine function 
\begin{align}\label{eq:MultiVector}
\msf{p} := \frac{(\S_{n,n} \nabla w)(\mu, 0)}{\vert (\S_{n,n} \nabla w)(\mu, 0)\vert}, \qquad \ell_{\msf{p}, \cu_n} := \int_{\cu_n}  \msf{p} \cdot x \, \d \mu(x).
\end{align}
Notice that here $\msf{p}$ is random, but when the particles in $\cu_n$ move within $\cu_n$, it does mot change the value; more precisely, the slope $\msf p$ is $\G_{n,n}$-measureable. We test $\ell_{\msf{p}, \cu_n}$ with \cref{eq:MultiLaplace2},  
\begin{align*}
\Er\Ll[u \ell_{\msf{p}, \cu_n} \, \big \vert \, \G_{n,n}\Rr] &= \Er\Ll[\int_{\cu_n} \nabla w \cdot \msf{p} \, \d \mu  \ \Bigg \vert \ \G_{n,n} \Rr]\\
&= \Er\Ll[\int_{\cu_n} \nabla w \, \d \mu \ \Bigg \vert \ \G_{n,n} \Rr]\cdot \msf{p} \\
&= \int_{\cu_n}(\S_{n,n} \nabla w)\cdot \msf{p}  \, \d \mu .
\end{align*}
Recalling the definition in \cref{eq:MultiVector}, we obtain that 
\begin{align*}
\int_{\cu_n} \vert \S_{n,n} \nabla w\vert  \, \d \mu  &= \Er\Ll[u \ell_{\msf{p}, \cu_n} \, \big \vert \, \G_{n,n}\Rr] \\
&\leq  \Ll(\Er\Ll[u^2 \, \big \vert \, \G_{n,n}\Rr]\Rr)^\frac{1}{2}
\Ll(\Er\Ll[ \ell^2_{\msf{p}, \cu_n} \, \big \vert \, \G_{n,n}\Rr]\Rr)^{\frac{1}{2}}\\
&\leq C\sqrt{\mu(\cu_n)}3^n \Ll(\Er\Ll[u^2 \, \big \vert \, \G_{n,n}\Rr]\Rr)^\frac{1}{2},
\end{align*} 
where in the last step, we use a direct calculation of $\Ll(\Er\Ll[ \ell^2_{\msf{p}, \cu_n} \, \big \vert \, \G_{n,n}\Rr]\Rr)^{\frac{1}{2}}$, and where the constant $C$ may depend on $d$. Since  $\S_{n,n} \nabla w$ is constant for every point in $\cu_n$, we have shown that
\begin{equation*}  
\sqrt{\mu(\cu_n)} \vert \S_{n,n} \nabla w\vert(\mu, 0) \leq C3^n \Ll(\Er\Ll[u^2 \, \big \vert \, \G_{n,n}\Rr]\Rr)^\frac{1}{2}.
\end{equation*}
We thus obtain that
\begin{align*}
\Er\Ll[\int_{\cu_n} \vert\S_{n,n} \nabla w\vert^2 \, \d \mu\Rr] &= \Er\Ll[\mu(\cu_n) \vert\S_{n,n} \nabla w\vert^2(\mu, 0) \Rr]\\
&\leq C 3^{2n} \Er\Ll[\Er\Ll[u^2 \, \big \vert \, \G_{n,n}\Rr]\Rr]\\
&=C 3^{2n} \Er[u^2], 
\end{align*}
and therefore
\begin{align}\label{eq:MultiDecomCBound}
\text{\cref{eq:MultiDecom}-c} \leq C 3^n (\Er[u^2])^{\frac{1}{2}}  \Ll(\Er\Ll[\int_{\cu_n} \vert \S_{n,n} \nabla u\vert^2 \, \d \mu\Rr]\Rr)^{\frac{1}{2}}.
\end{align}
We now combine \cref{eq:MultiDecom}, \eqref{eq:MultiDecomABound}, \eqref{eq:MultiDecomBBound}, and \eqref{eq:MultiDecomCBound}, to obtain  \cref{eq:Multi}.
\end{proof}

\subsection{Caccioppoli inequality}
For every bounded open set $U \subset \Rd$, we define the space of $\a$-harmonic functions on $\mmd(\Rd)$ by
\begin{align}
\mcl A(U) := \Ll\{u \in \cH^1(U) : \  \forall \varphi \in \cH^1_0(U), \ \Er\Ll[\int_U \nabla u \cdot \a \nabla \varphi \, \d \mu \Rr] = 0\Rr\}.
\end{align}
Recalling that, for any two bounded open sets $V \subset U$, we have $\cH^1(U) \subset \cH^1(V)$ and $\cH^1_0(V) \subset \cH^1_0(U)$, so we see that $\mcl A(U) \subset \mcl A (V)$. 
For the classical Caccioppoli inequality, a standard proof is as follows: we multiply the harmonic function by a cutoff function, and then use this as a test function against the harmonic function itself. Adapting this argument to our space of particle configurations is not immediate. A naive approach would be to introduce a cutoff that brings the value of the function to zero whenever a particle approaches the boundary of the domain. But proceeding in this way is a very bad idea, since as we increase the size of the domain, there will essentially always be some particles near the boundary. We will instead rely on a suitable averaging procedure for particles that fall outside of a given region, using the localization operators defined in Subsection \ref{subsubsec:LocReg}. Notice that our goal thus is not to bring the function to zero as a particle approaches the boundary of the box. Rather, it is only to produce a function that stops depending on the position of a particle that progressively approaches the boundary of the domain, in agreement with our definition of the space $\cH^1_0(U)$ (and departing from the traditional definition of the Sobolev $H^1_0$ spaces). 

\begin{proposition}[Modified Caccioppoli inequality]
There exist ${\theta(d, \Lambda) \in (0,1)}$, ${C(d, \Lambda) < \infty}$, and ${R_0(d, \Lambda) < \infty}$ such that for every $r \ge R_0$ and ${u \in \mcl A(Q_{3r})}$, we have
\label{prop:Caccioppoli}
\begin{multline}
\label{eq:Caccioppoli}
\Er\Ll[\frac{1}{\rho \vert Q_r \vert}\int_{Q_r} \nabla (\A_{r+2} u) \cdot \a \nabla (\A_{r+2} u) \, \d \mu\Rr] \\
\leq  \frac{C}{r^2 \rho \vert Q_{3r} \vert} \Er[u^2] + \theta \Er\Ll[\frac{1}{\rho \vert Q_{3r} \vert}\int_{Q_{3r}} \nabla u \cdot \a \nabla u \, \d \mu\Rr].
\end{multline}
\end{proposition}
\begin{remark}
Inequality \cref{eq:Caccioppoli} controls the norm of the gradient of a harmonic function in the small cube $Q_r$ by a sum of terms involving the norm of the gradient in the larger cube $Q_{3r}$. This does not seem to be useful at first glance. However, the key point is that the multiplicative factor $\theta$ is smaller than one.
\end{remark}
The proof of Proposition~\ref{prop:Caccioppoli} will be divided into two steps. In the first step, provided by the lemma below, we prove a weaker Caccioppoli inequality,  without the normalization of the volume. In the second step we use an iterative argument to improve the result and obtain Proposition~\ref{prop:Caccioppoli}.

Recall that $\A_{s,\epsilon}$ is the regularized localization operator defined in \cref{eq:defA}.
\begin{lemma}[Weak Caccioppoli inequality]\label{lem:WeakCaccioppoli}
Fix $\theta'(\Lambda) := \frac{2\Lambda}{2\Lambda + 1} \in (0,1)$. For every $r > 0$, $s \geq r+2, \epsilon > 0$ and ${u \in \mcl A(Q_{s+\epsilon})}$, we have 
\begin{multline}\label{eq:WeakCaccioppoli}
\frac{\theta'}{2 \epsilon^2}\Er[(\A_{s} u)^2] + \Er\Ll[\int_{Q_r} \nabla(\A_{s, \epsilon} u) \cdot \a \nabla(\A_{s, \epsilon} u) \, \d \mu\Rr] \\\leq \theta' \Ll(\frac{1}{2 \epsilon^2}\Er[(\A_{s+\epsilon}u)^2] + \Er\Ll[\int_{Q_{s+\epsilon}} \nabla u \cdot \a \nabla u \, \d \mu\Rr] \Rr).
\end{multline}
\end{lemma}
\begin{proof}
The proof of this lemma borrows some elements from \cite[Lemma 4.8]{gu2020decay}; in both settings, the main point is to construct and analyze an appropriate ``cut-off'' version of the function $u$. We use the function ${\tilde{\A}_{s,\epsilon}u \in \cH^1_0(Q_{s+\epsilon})}$ defined in \cref{eq:Atilde} as a cut-off of the function $u$ and test it against $u \in \mcl A(Q_{s+\ep})$ to get
\begin{align}\label{eq:CaccioTest}
\Er\Ll[\int_{Q_{s+\epsilon}} \nabla(\tilde{\A}_{s,\epsilon}u) \cdot \a \nabla u \, \d\mu\Rr] = 0.
\end{align}
Combining this with the decomposition
\begin{equation}\label{eq:PerThreeTerms}
\begin{split}
\Er\Ll[\int_{Q_{s+\epsilon}} \nabla(\tilde{\A}_{s,\epsilon}u) \cdot \a \nabla u \, \d\mu\Rr] & =  \underbrace{\Er\Ll[ \int_{Q_{s-2}} \nabla (\tilde{\A}_{s,\epsilon} u) \cdot \a \nabla u \, \d\mu\Rr]}_{\text{\cref{eq:PerThreeTerms}-a}} \\
& \qquad  +  \underbrace{\Er\Ll[ \int_{Q_{s} \backslash Q_{s-2}}  \nabla (\tilde{\A}_{s,\epsilon} u) \cdot \a \nabla u   \, \d\mu\Rr]}_{\text{\cref{eq:PerThreeTerms}-b}} \\
& \qquad + \underbrace{\Er\Ll[ \int_{Q_{s+\epsilon} \backslash Q_{s}}  \nabla (\tilde{\A}_{s,\epsilon} u) \cdot \a \nabla u   \, \d\mu\Rr]}_{\text{\cref{eq:PerThreeTerms}-c}},
\end{split}
\end{equation}
we obtain that
\begin{align}\label{eq:CaccioTestABC}
\text{\cref{eq:PerThreeTerms}-a} \leq \vert\text{\cref{eq:PerThreeTerms}-b} \vert + \vert \text{\cref{eq:PerThreeTerms}-c}\vert.
\end{align}
We now study each of these three terms. For the first term \cref{eq:PerThreeTerms}-a, since $x \in Q_{s-2}$, the coefficient $\a$ is $\mcl F_{Q_{s}}$-measurable. We can thus use \cref{eq:Commute}, \cref{eq:Atilde}, \eqref{eq:AtildeDerivative} and \eqref{eq:AtildeIdentity} to get 
\begin{align*}
\text{\cref{eq:PerThreeTerms}-a} &= \frac{2}{\epsilon^2}\Er\Ll[ \int_{Q_{s-2}}   \int_{0}^{\epsilon} (\epsilon - t)  \A_{s + t} (\nabla u)  \cdot \a \nabla u \, \d t  \, \d\mu\Rr] \\
&= \frac{2}{\epsilon^2}\Er\Ll[ \int_{Q_{s-2}}   \int_{0}^{\epsilon} (\epsilon - t) \Er\Ll[\A_{s + t} (\nabla u) \cdot \a \A_{s + t} (\nabla u) \, \vert \mcl{F}_{\bar Q_{s+t}}\Rr] \, \d t  \, \d\mu\Rr] \\
&= \Er\Ll[ \int_{Q_{s-2}}  \nabla (\A_{s, \epsilon}u) \cdot \a  \nabla (\A_{s, \epsilon}u)   \, \d\mu\Rr].
\end{align*}
We then apply \cref{eq:AtildeDerivative} for the second term \cref{eq:PerThreeTerms}-b. We notice that in $Q_{s} \backslash Q_{s-2}$, $\a$ is no longer $\mcl{F}_{Q_{s}}$-measurable. So we use Young's inequality and the bound $\a \leq \Lambda \id $
\begin{align*}
\vert \text{\cref{eq:PerThreeTerms}-b}\vert &= \frac{2}{\epsilon^2}\Er\Ll[ \int_{Q_{s} \backslash Q_{s-2}}   \int_{0}^{\epsilon} (\epsilon - t)  \A_{s + t} (\nabla u)  \cdot \a \nabla u \, \d t  \, \d\mu\Rr] \\
&\leq \frac{\Lambda}{\epsilon^2}\Er\Ll[ \int_{Q_{s} \backslash Q_{s-2}}   \int_{0}^{\epsilon} (\epsilon - t)  \Ll( \Ll\vert\A_{s + t} (\nabla u)\Rr\vert^2 +  \Ll\vert\nabla u \Rr\vert^2 \Rr) \, \d t  \, \d\mu\Rr].
\end{align*}
For the part with conditional expectation, we use Jensen's inequality and the uniform bound $\id \leq \a \leq \Lambda \id$
\begin{align*}
\frac{\Lambda}{\epsilon^2}\Er\Ll[ \int_{Q_{s} \backslash Q_{s-2}}   \int_{0}^{\epsilon} (\epsilon - t) \Ll\vert\A_{s + t} (\nabla u)\Rr\vert^2  \, \d t  \, \d\mu\Rr] & \leq \frac{\Lambda}{2}\Er\Ll[ \int_{Q_{s} \backslash Q_{s-2}}  \vert \nabla u \vert^2  \, \d\mu\Rr]  \\
& \leq \frac{\Lambda}{2}\Er\Ll[ \int_{Q_{s} \backslash Q_{s-2}}  \nabla u \cdot \a  \nabla u  \, \d\mu\Rr]. 
\end{align*}
This concludes that $|\text{\cref{eq:PerThreeTerms}-b}| \leq \Lambda \Er\Ll[ \int_{Q_{s} \backslash Q_{s-2}}  \nabla u \cdot \a  \nabla u  \, \d\mu\Rr]$.

For the third term \cref{eq:PerThreeTerms}-c, we use \cref{eq:AtildeDerivative} and obtain 
\begin{align*}
\vert \text{\cref{eq:PerThreeTerms}-c} \vert &\leq \text{\cref{eq:PerThreeTerms}-c1} + \text{\cref{eq:PerThreeTerms}-c2} \\
\text{\cref{eq:PerThreeTerms}-c1} &= \frac{2}{\epsilon^2}\Ll\vert\Er\Ll[ \int_{Q_{s+\epsilon} \backslash Q_{s}}   \int_{\tau(x)-s}^{\epsilon} (\epsilon - t) \A_{s+ t} (\nabla u)  \cdot \a \nabla u \, \d t  \, \d\mu\Rr] \Rr\vert\\
\text{\cref{eq:PerThreeTerms}-c2} &= \frac{2}{\epsilon^2}\Ll\vert\Er\Ll[\int_{Q_{s+\epsilon} \backslash Q_{s}}  (s+\epsilon - \tau(x))\Delta_{\tau(x)} (\A u) \n(x) \cdot \a \nabla u \, \d\mu\Rr] \Rr\vert.
\end{align*}
The part of \cref{eq:PerThreeTerms}-c1 can be treated as that of \cref{eq:PerThreeTerms}-b, so that
\begin{align*}
\text{\cref{eq:PerThreeTerms}-c1} \leq \Lambda \Er\Ll[ \int_{Q_{s+\epsilon} \backslash Q_{s}}  \nabla u \cdot \a  \nabla u  \, \d\mu\Rr].
\end{align*}
We study the part \cref{eq:PerThreeTerms}-c2 using Young's  inequality with a parameter $\beta > 0$ to be fixed later:
\begin{equation}\label{eq:PerThreeTermsC2}
\begin{split}
&\frac{2}{\epsilon^2}\Ll\vert\Er\Ll[\int_{Q_{s+\epsilon} \backslash Q_{s}}  (s+\epsilon - \tau(x))\Delta_{\tau(x)} (\A u) \n(x) \cdot \a \nabla u \, \d\mu \Rr] \Rr\vert \\
\leq & \frac{\Lambda}{\beta \epsilon^2}\Er\Ll[\int_{ Q_{s+\epsilon} \backslash Q_{s}}   (s+\epsilon - \tau(x))\vert \Delta_{\tau(x)} (\A u) \vert^2 \, \d\mu\Rr] + \frac{\beta\Lambda}{\epsilon^2}\Er\Ll[\int_{Q_{s+\epsilon} \backslash Q_{s}}   (s+\epsilon - \tau(x)) \vert \nabla u \vert^2 \, \d\mu \Rr] \\
\leq & \frac{\Lambda}{\beta \epsilon^2}\Er\Ll[\int_{Q_{s+\epsilon} \backslash Q_{s}}   (s+\epsilon - \tau(x))\vert \Delta_{\tau(x)} (\A u) \vert^2 \, \d\mu\Rr] + \frac{\beta\Lambda}{\epsilon}\Er\Ll[\int_{ Q_{s+\epsilon} \backslash Q_{s} } \nabla u \cdot \a  \nabla u \, \d\mu \Rr].
\end{split}
\end{equation}
The first term above will be responsible for producing the $\cL^2$ term on the right side of \cref{eq:WeakCaccioppoli}. We start by writing
\begin{multline*}
\frac{\Lambda}{\beta \epsilon^2}\Er\Ll[\int_{Q_{s+\epsilon} \backslash Q_{s} }   (s+\epsilon - \tau(x))\vert \Delta_{\tau(x)} (\A u) \vert^2 \, \d\mu\Rr] 
\\
=  \frac{\Lambda}{\beta \epsilon^2}\Er\Ll[\sum_{s \leq \tau \leq s+\epsilon}  (s+\epsilon - \tau) \vert \Delta_\tau (\A u) \vert^2 \Rr],
\end{multline*}
where on the right side, the sum is over all $\tau$'s that are jump discontinuites for $(\A_s u)_{s \geq 0}$. Recalling the definition of the bracket process $([\A u]_s)_{s \geq 0}$ defined in \cref{eq:Bracket}, we use Fubini's lemma and the $\cL^2$ isometry ${\Er\Ll[\Ll[\A u \Rr]_{s}\Rr] = \Er\Ll[(\A_s u)^2\Rr]}$:
\begin{align*}
\frac{\Lambda}{\beta \epsilon^2}\Er\Ll[\sum_{s \leq \tau \leq s+\epsilon}  (s+\epsilon - \tau) \vert \Delta_\tau (\A u) \vert^2 \Rr] &= \frac{\Lambda}{\beta \epsilon^2}\Er\Ll[\sum_{s \leq \tau \leq s+\epsilon}  \int_{s}^{s+\epsilon} \Ind{\tau \leq t \leq s+\epsilon} \, \d t \vert \Delta_\tau (\A u) \vert^2 \Rr]\\
&= \frac{\Lambda}{\beta \epsilon^2}\Er\Ll[\int_{s}^{s+\epsilon} \sum_{s \leq \tau \leq t}   \vert \Delta_\tau (\A u) \vert^2 \, \d t\Rr]\\
&= \frac{\Lambda}{\beta \epsilon^2}\Er\Ll[\int_{s}^{s+\epsilon} \Ll(\Ll[\A u\Rr]_{t} - \Ll[\A u \Rr]_{s}\Rr) \, \d t\Rr] \\
&= \frac{\Lambda}{\beta \epsilon^2}\int_{s}^{s+\epsilon} \Ll(\Er\Ll[(\A_{t} u)^2\Rr] - \Er\Ll[(\A_{s} u)^2\Rr] \Rr)\, \d t\\
&\leq \frac{\Lambda}{\beta \epsilon} \Ll(\Er\Ll[(\A_{s+\epsilon} u)^2\Rr] - \Er\Ll[(\A_{s} u)^2\Rr]\Rr).
\end{align*}
Putting this estimate back into \cref{eq:PerThreeTermsC2}, we conclude the estimating of the term \cref{eq:PerThreeTerms}-c2, obtaining
\begin{align*}
\text{\cref{eq:PerThreeTerms}-c2} \leq \frac{\Lambda}{\beta \epsilon} \Ll(\Er\Ll[(\A_{s+\epsilon} u)^2\Rr] - \Er\Ll[(\A_{s} u)^2\Rr]\Rr) + \frac{\beta\Lambda}{\epsilon}\Er\Ll[\int_{ Q_{s+\epsilon} \backslash Q_{s} } \nabla u \cdot \a  \nabla u \, \d\mu \Rr].
\end{align*} 
By choosing $\beta = \epsilon$, recalling \cref{eq:CaccioTestABC}, and that $r \le s-2$, we can combine this estimate with those of \cref{eq:PerThreeTerms}-a, \cref{eq:PerThreeTerms}-b, and \cref{eq:PerThreeTerms}-c1 to get
\begin{multline*}
\frac{\Lambda}{\epsilon^2} \Er\Ll[(\A_{s} u)^2\Rr] + \Er\Ll[ \int_{Q_{r}}  \nabla (\A_{s, \epsilon}u) \cdot \a  \nabla (\A_{s, \epsilon}u)   \, \d\mu\Rr] \\\
\leq \frac{\Lambda}{\epsilon^2} \Er\Ll[(\A_{s+\epsilon} u)^2\Rr] + 2\Lambda \Er\Ll[\int_{ Q_{s+\epsilon} \backslash {Q_{r}} } \nabla u \cdot \a  \nabla u \, \d\mu \Rr]. 
\end{multline*}
We now proceed with a hole-filling argument: adding $2\Lambda \Er\Ll[ \int_{Q_{r}}  \nabla \A_{s, \epsilon}u \cdot \a  \nabla \A_{s, \epsilon}u   \, \d\mu\Rr]$ to both sides of the equation above, and using Jensen's inequality, we obtain 
\begin{multline*}
\frac{\Lambda}{\epsilon^2} \Er\Ll[(\A_{s} u)^2\Rr] + (2\Lambda + 1)\Er\Ll[ \int_{Q_{r}}  \nabla (\A_{s, \epsilon}u) \cdot \a  \nabla (\A_{s, \epsilon}u)   \, \d\mu\Rr] \\
 \leq \frac{\Lambda}{\epsilon^2} \Er\Ll[(\A_{s+\epsilon} u)^2\Rr] + 2\Lambda \Er\Ll[\int_{ Q_{s+\epsilon}} \nabla u \cdot \a  \nabla u \, \d\mu \Rr]. 
\end{multline*}
Dividing both sides by $(2\Lambda + 1)$, and setting $\theta' := \frac{2\Lambda}{2\Lambda+1}$, we obtain the desired inequality \cref{eq:WeakCaccioppoli}.
\end{proof}

We remark that \cref{eq:WeakCaccioppoli} does not imply directly \cref{eq:Caccioppoli}. For example, let $r > 2$ and we choose $s=2r$ and $\epsilon=r$ in \cref{eq:WeakCaccioppoli}, then with a normalization of volume we get 
\begin{multline*}
\frac{\theta'}{2 r^2 \rho \vert Q_r \vert}\Er[(\A_{2r} u)^2] + \Er\Ll[\frac{1}{\rho \vert Q_r \vert}\int_{Q_r} \nabla(\A_{2r, r} u) \cdot \a \nabla(\A_{2r, r} u) \, \d \mu\Rr] \\ \leq 3^d\theta' \Ll(\frac{1}{2 r^2 \rho\vert Q_{3r} \vert}\Er[(\A_{3r}u)^2] + \Er\Ll[\frac{1}{\rho\vert Q_{3r} \vert}\int_{Q_{3r}} \nabla u \cdot \a \nabla u \, \d \mu\Rr] \Rr).
\end{multline*}
Then another factor $3^d$ will be added, and we typically do not have $3^d \theta' \in (0,1)$, since we recall that $\theta' = \frac{2\Lambda}{2\Lambda + 1}$.

\begin{proof}[Proof of \Cref{prop:Caccioppoli}]  
We apply Lemma~\ref{lem:WeakCaccioppoli} iteratively, with very small increments of the volume. Let $\delta > 0$ to be fixed later, and choose ${s = (1+\delta)r}, \epsilon = \delta r$. For convenience, we assume that $r$ is sufficiently large that 
\begin{align}\label{eq:CaccioppoliCondition1}
s = (1+\delta)r \geq r+2, \quad \text{that is} \quad r \ge 2\delta^{-1}.
\end{align}
Equation \eqref{eq:WeakCaccioppoli} and Jensen's inequality give us that, provided $(1+2\de)r \le 3r$,
\begin{multline}\label{eq:CaccioppoliBootstrap}
\Er\Ll[\frac{1}{\rho \vert Q_r \vert}\int_{Q_r} \nabla(\A_{(1+\delta)r, \delta r} u) \cdot \a \nabla(\A_{(1+\delta)r, \delta r} u) \, \d \mu\Rr] \\ \leq \tilde{\theta} \Ll(\frac{1}{2 (\de r)^2 \rho\vert Q_{(1+2\delta)r} \vert}\Er[u^2] + \Er\Ll[\frac{1}{\rho\vert Q_{(1+2\delta)r} \vert}\int_{Q_{(1+2\delta)r}} \nabla u \cdot \a \nabla u \, \d \mu\Rr] \Rr),
\end{multline}
with $\tilde{\theta} = (1+2\delta)^d \theta'$. We choose the constant $\delta > 0$ sufficiently small that $\tilde{\theta} < 1$. In order to obtain \cref{eq:Caccioppoli}, we will now apply \cref{eq:CaccioppoliBootstrap} iteratively, from the cube $Q_r$ to the larger cube $Q_{3r}$. 

We give the details for this argument---see also \Cref{fig:bootstrap} for an illustration. We plan to use \cref{eq:CaccioppoliBootstrap} $(N+1)$ times, and  and let $\delta \in (0,1)$, $N \in \mathbb{N}$ satisfy
\begin{align}\label{eq:CaccioppoliCondition2}
\tilde{\theta} = (1+2\delta)^d \theta' < 1, \qquad (1+2\delta)^{N+1} = 3.
\end{align}
Then we set the scale and the $\a$-harmonic functions in every scale
\begin{align}\label{eq:CaccioppoliIteration}
\left\{
	\begin{array}{ll}
	r_n = (1+2\delta)^n r & 0 \leq n \leq N+1, \\
	u_{N+1} = u &, \\
	u_n = \A_{(1+\delta)r_n, \delta r_n} u_{n+1} & 0 \leq n \leq N.
	\end{array}
\right.
\end{align}
We can prove by induction that $u_n \in \mcl A(Q_{r_n})$ under the condition \cref{eq:CaccioppoliCondition1}. Then for every $0 \leq n \leq N$, we apply \cref{eq:CaccioppoliBootstrap} from $u_n$ on $Q_{r_n}$ to $u_{n+1}$ on $Q_{r_{n+1}}$ 
\begin{multline}\label{eq:CaccioppoliBootstrap2}
\Er\Ll[\frac{1}{\rho \vert Q_{r_n} \vert}\int_{Q_{r_n}} \nabla u_n \cdot \a \nabla u_n \, \d \mu\Rr] \\ \leq \tilde{\theta} \Ll(\frac{1}{2 (\delta r_n)^2 \rho\vert Q_{(1+2\delta)r_n} \vert}\Er[(u_{n+1})^2] + \Er\Ll[\frac{1}{\rho\vert Q_{r_{n+1}} \vert}\int_{Q_{r_{n+1}}} \nabla u_{n+1} \cdot \a \nabla u_{n+1} \, \d \mu\Rr] \Rr).
\end{multline}
Iterating on \cref{eq:CaccioppoliBootstrap2} until $u_{N+1} = u$ on $Q_{3r}$, we get 
\begin{multline*}
\Er\Ll[\frac{1}{\rho \vert Q_r \vert}\int_{Q_r} \nabla u_0 \cdot \a \nabla u_0 \, \d \mu\Rr] \\
\leq  \Ll(\frac{3^d}{2}\sum_{n=0}^N(1+2\delta)^{-2n}\Rr) \frac{1}{(\de r)^2 \rho \vert Q_{3r} \vert} \Er[u^2] + (\tilde{\theta})^{N+1} \Er\Ll[\frac{1}{\rho \vert Q_{3r} \vert}\int_{Q_{3r}} \nabla u \cdot \a \nabla u \, \d \mu\Rr].
\end{multline*}
We notice that $u_0$ can be seen as as a weighted sum of $\A_{s'} u$, for scales $s'$ satisfying ${s' \geq (1+\delta)r \geq r+2}$, by \cref{eq:CaccioppoliCondition1}. So we apply once Jensen's inequality for $u_0$ and obtain \cref{eq:Caccioppoli} by setting 
\begin{equation*}
C(d,\Lambda) := \frac{3^d}{2\delta^2}\sum_{n=0}^N(1+2\delta)^{-2n}, \qquad \theta := (\tilde{\theta})^{N+1}. \qedhere
\end{equation*}
Although we will not use this later, we now give more explicit estimates for the choice of the parameters in the proof above, resulting from the conditions listed in \cref{eq:CaccioppoliCondition1} and \cref{eq:CaccioppoliCondition2}. It suffices to pick an integer $N$ larger than $\Ll\lfloor \frac{d \log 3}{\log \Ll(1 + \frac{1}{2\Lambda}\Rr)}\Rr\rfloor$, and then in \cref{eq:CaccioppoliCondition2} use $\delta := \frac 1 2 (3^{\frac{1}{N+1}} - 1)$ to fix $\delta$, and in \cref{eq:CaccioppoliCondition1} we  require $r \ge 2\de^{-1}$, which gives the condition for the minimal scale $R_0$. A possible choice is the following 
\begin{equation*}
\begin{split}
&N := 2\Ll\lfloor \frac{d \log 3}{\log \Ll(1 + \frac{1}{2\Lambda}\Rr)}\Rr\rfloor+1,\quad \delta := \frac 1 2 (3^{\frac{1}{N+1}} - 1) \simeq \frac{1}{8d\Lambda}, \quad R_0 := 2\delta^{-1} \simeq  16d\Lambda,\\
&\tilde{\theta} := \theta'(1+2\delta)^d \simeq \Ll(1+\frac{1}{2\Lambda}\Rr)^{-\frac{1}{2}}, \ \ \theta := (\tilde{\theta})^{N+1} \simeq 3^{-d}, \ \ C:= \frac{3^d}{2\delta^2}\sum_{n=0}^N(1+2\delta)^{-2n} \simeq 2^8  3^d d^3 \Lambda^3. 
\end{split}
\end{equation*}
\end{proof}

\begin{figure}[ht]
\centering
\includegraphics[scale=0.5]{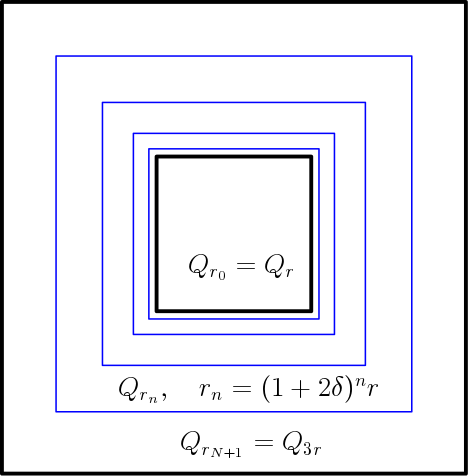}
\caption{An illustration of the iterative argument for the proof of Proposition~\ref{prop:Caccioppoli}. Since \Cref{lem:WeakCaccioppoli} can imply \Cref{prop:Caccioppoli} only for a comparison from scale $r$ to $(1+2\delta)r$ with $\delta$ very small, we add many intermediate scales $r_n = (1+2\delta)^n r$ between $r$ and $3r$. 
}
\label{fig:bootstrap}
\end{figure}

%
%
%
%
%
%
%
%

\section{Subadditive quantities}\label{sec:Subadditive}

We aim to adapt the strategy in \cite[Chapter 2]{AKMbook} for our model in continuum configuration space. In this section, we define several subadditive quantities, denoted by $\nu, \nu^*, J$, and develop their elementary properties. We then we use them and a renormalization argument to obtain a quantitative rate of convergence for $\ab$ in \Cref{subsec:ConvergenceJ}. 
\subsection{Subadditive quantities \texorpdfstring{$\nu$ and $\nu^*$}{nu and nu*}}
For every bounded domain $U \subset \Rd$ and $p,q \in \Rd$, we define the affine function in $U$ with slope~$p$ by
\begin{align}\label{eq:defAffine}
\ell_{p,U}(\mu) := \int_{U} p \cdot x \, \d \mu(x),
\end{align} 
and introduce the subadditive quantities
\begin{equation}\label{eq:defNu}
\begin{split}
\nu(U,p) &:= \inf_{v \in \ell_{p,U} + \cH^1_0(U)}\Er \Ll[ \frac{1}{\rho \vert U \vert} \int_{U} \frac{1}{2} \nabla v \cdot \a \nabla v \, \d \mu \Rr], \\
\nu^*(U,q) &:=  \sup_{u \in \cH^1(U)} \Er \Ll[\frac{1}{\rho \vert U \vert}\int_{U} \Ll( -\frac 1 2 \nabla u \cdot \a \nabla u + q \cdot \nabla u \Rr) \, \d \mu \Rr].
\end{split}
\end{equation}
The quantity $\nu$ can be thought of as the average energy per unit volume of the solution which matches with the behavior of the affine function $\ell_{p,U}$ when a particle leaves the domain $U$. The quantity $\nu^*$ is analogous to a Neumann problem with prescribed average flux of $q$. As will be seen below, the quantities $\nu$ and $\nu^*$ are approximately dual to one another; the quality of this approximation as the domain $U$ grows to $\Rd$ will be central to the proof of Theorem~\ref{thm:main}. If the matrix $\a$ were constant, then by \cref{eq:Integral0} the minimizer for $\nu(U,p)$ would be $\ell_{p,U}$, and we would have $\nu(U,p) = \frac 1 2 p \cdot \a p$; and similarly, were $\a$ constant, we would have $\nu^*(U,q) = \frac 1 2 q \cdot \a^{-1} q$.

We start by recording elementary properties satisfied by $\nu$ and $\nu^*$. We recall that $\G_U = \sigma(\mu(U), \mu \mres (\Rd \backslash U))$. For every $r > 0$, we denote by $B_r(U)$ the $r$-enlargement of $U$, that is, $B_r(U) := \{x \in \Rd: \dist(x, U) < r\}$. 
\begin{proposition}[Elementary properties of $\nu$ and $\nu^*$]\label{prop:NuBase}
The following properties hold for every bounded domain $U \subset \Rd$ with Lipschitz boundary and $p,q,p',q' \in \Rd$.

(1) There exists a unique solution for the optimization problem in the definition of $\nu(U,p)$ that satisfies ${\Er[v - \ell_{p,U}] = 0}$; we denote it by $v(\cdot,U,p)$. For the optimization problem in the definition of $\nu^*(U,q)$, there exists a maximizer $u(\cdot, U, q)$ that is $\mcl F_{B_1(U)}$-measurable and such that $\Er[u \, \vert \, \mcl \G_U] = 0$. They are $\a$-harmonic functions on $U$, i.e. $v(\cdot, U, p),  u(\cdot, U, q) \in \mcl{A}(U)$.

(2) There exist two $d \times d$ symmetric matrices $\ab(U)$ and $\ab_*(U)$ such that 
\begin{equation}\label{eq:NuMatrix}
\nu(U,p) = \frac{1}{2} p \cdot \ab(U) p, \qquad \nu^*(U,q) = \frac{1}{2} q \cdot \ab_*^{-1}(U) q,
\end{equation}
and these matrices satisfy ${\id \leq \ab(U) \leq \Lambda \id}$ and ${\id \leq \ab_*(U) \leq \Lambda \id}$.
Moreover,
\begin{align}
p' \cdot \ab(U) p &= \Er\Ll[\frac{1}{\rho \vert U\vert}\int_{U} p' \cdot \a(\mu, x) \nabla v(\mu, x, U, p) \, \d \mu(x)\Rr], \label{eq:flux1} \\
q' \cdot \ab_*^{-1}(U) q &= \Er\Ll[\frac{1}{\rho \vert U\vert} \int_{U}  q' \cdot \nabla u(\mu, x, U, q) \, \d \mu(x)\Rr] \label{eq:flux2}.
\end{align}

(3) Slope: $v(\mu,U,p)$ satisfies
\begin{equation}\label{eq:Slope1}
\Er\Ll[\fint_{U} \nabla v(\mu, x, U, p) \, \d \mu(x) \ \Bigg\vert \ \G_U\Rr] = \Er\Ll[\frac{1}{\rho \vert U\vert}\int_{U} \nabla v(\mu, x, U, p) \, \d \mu(x) \Rr] = p.
\end{equation}
For the function $u(\cdot,U,q)$, there exists a $d \times d$ symmetric matrix ${\id \leq \a_*(U; \G_U) \leq \Lambda \id}$ such that 
\begin{equation}\label{eq:Slope2}
\Er\Ll[\fint_{U} \nabla u(\mu, x, U, q) \, \d \mu(x) \ \Bigg\vert \ \mcl \G_U \Rr] = \a_*^{-1}(U; \G_U)q,
\end{equation}
and $\ab_*^{-1}(U) = \frac{1}{\rho \vert U \vert}\Er[\a_*^{-1}(U; \G_U)\mu(U)]$, so that
\begin{align}\label{eq:Slope3}
\Er\Ll[\frac{1}{\rho \vert U\vert}\int_{U} \nabla u(\mu, x, U, q) \, \d \mu(x) \Rr] = \ab_*^{-1}(U)q.
\end{align}

(4) Quadratic response: for every $v' \in \ell_{p,U} + \cH^1_0(U)$, we have  
\begin{multline}\label{eq:QuadraRep1}
\Er \Ll[\frac{1}{\rho \vert U \vert}\int_{U} \frac{1}{2} \nabla(v' - v(\mu,U,p)) \cdot \a \nabla(v' - v(\mu,U,p)) \, \d \mu \Rr] \\
= \Er \Ll[\frac{1}{\rho \vert U \vert}\int_{U} \frac{1}{2} \nabla v' \cdot \a \nabla v' \, \d \mu \Rr] - \nu(U,p).
\end{multline}
Similarly, for every $u' \in \cH^1(U)$, we have 
\begin{multline}\label{eq:QuadraRep2}
\Er \Ll[\frac{1}{\rho \vert U \vert}\int_{U} \frac{1}{2} \nabla(u' - u(\mu,U,q)) \cdot \a \nabla(u' - u(\mu,U,q)) \, \d \mu \Rr] \\
= \nu^*(U,q) - \Er \Ll[\frac{1}{\rho \vert U \vert}\int_{U}  \Ll( -\frac 1 2 \nabla u' \cdot \a \nabla u' + q \cdot \nabla u' \Rr) \, \d \mu \Rr].
\end{multline}

(5) The quantities $\nu$ and $\nu^*$ are subadditive: for every $n \in \N$,
\begin{equation}\label{eq:Subadditive}
\nu(\cu_{n+1}, p) \leq \nu(\cu_{n}, p), \qquad \nu^*(\cu_{n+1}, q) \leq \nu^*(\cu_{n}, q).
\end{equation}
\end{proposition}
\begin{proof}
We prove each of these points in turn.

(1)We study at first the maximizer for the problem $\nu^*(U,q)$. A first observation is that the maximizer can be found in  $\mcl F_{B_1(U)}$-measurable functions. Because for any $u \in \cH^1(U)$, its conditional expectation $\Er[u \, \vert \, \mcl F_{B_1(U)}]$ reaches a larger value for the functional in $\nu^*(U,q)$. We use Jensen's inequality that 
\begin{align*}
&\Er \Ll[\int_{U} \Ll( -\frac 1 2 \nabla  \Er[u \, \vert \, \mcl F_{B_1(U)}] \cdot \a \nabla \Er[u \, \vert \, \mcl F_{B_1(U)}] + q \cdot \nabla \Er[u \, \vert \, \mcl F_{B_1(U)}] \Rr) \, \d \mu \Rr] \\
= &\Er\Ll[\Er \Ll[\int_{U} \Ll( -\frac 1 2  \Er[\nabla u \, \vert \, \mcl F_{B_1(U)}] \cdot \a  \Er[\nabla u \, \vert \, \mcl F_{B_1(U)}] + q \cdot  \Er[\nabla u \, \vert \, \mcl F_{B_1(U)}] \Rr) \, \d \mu \, \Big\vert \, \mcl F_{B_1(U)} \Rr]\Rr] \\
\geq & \Er \Ll[\int_{U} \Ll( -\frac 1 2 \nabla u \cdot \a \nabla u + q \cdot \nabla u \Rr) \, \d \mu \Rr].
\end{align*}
By a variational calculus, we know the characterization of a maximizer with elliptic equation that for any $\phi \in \cH^1(U)$
\begin{align}\label{eq:EllipNuDual}
\Er \Ll[\int_{U}  \nabla u \cdot \a \nabla \phi  \, \d \mu \Rr] = \Er \Ll[\int_U q \cdot \nabla \phi \, \d \mu\Rr].
\end{align}
Similarly to the discussion in the proof of \Cref{prop:H2}, we know that a solution for this problem also satisfies the more precise equation
\begin{align}\label{eq:EllipNuDualSmall}
\Er \Ll[\int_{U}  \nabla u \cdot \a \nabla \phi  \, \d \mu \, \Big \vert \, \mcl \G_U\Rr] = \Er \Ll[\int_U q \cdot \nabla \phi \, \d \mu \, \Big \vert \,  \G_U\Rr],
\end{align}
and we can define its solution in the space 
\begin{align*}
W = \{f \in \cH^1(U) : \Er[f \, \vert \,  \mcl \G_U] = 0\}.
\end{align*}
In this space, we have  
\begin{align*}
\Er[f^2 \, \vert \, \mcl \G_U] \leq C \diam(U)^2\Er\Ll[\int_{U} \vert \nabla f \vert^2 \, \d \mu \, \big\vert \, \G_U\Rr],
\end{align*}
by the Poincar\'e inequality \Cref{prop:Poincare2}. Then the coercivity on left hand side in \cref{eq:EllipNuDualSmall} is ensured and we can apply the Lax-Milgram theorem. We call this maximizer $u(\mu, U, q)$. Testing \cref{eq:EllipNuDual} with $\phi \in \cH^{1}_0(U)$, \cref{eq:Integral0} implies that its right hand side is $0$, so we have $u(\mu,U,q) \in \mcl A(U)$.

Then we turn to $\nu(U,p)$. By a first order variation calculus, we know that a minimizer $v$ for $\nu(U,p)$ is characterized by an elliptic equation that for any $\phi \in \cH^1_0(U)$ 
\begin{align}\label{eq:EllipNu}
\Er\Ll[\int_{U} \nabla (v - \ell_{p,U}) \cdot \a \nabla \phi \, \d \mu\Rr] = \Er\Ll[\int_U -p \cdot \a \nabla \phi \, \d \mu\Rr].
\end{align}
We remark that one cannot treat this equation as \cref{eq:EllipNuDual}, because $\Er[v \vert \G_U]$ is not an element in $\cH^1_0(U)$ and we cannot subtract it. On the other hand, we can apply the Lax-Milgram theorem on the space  
\begin{align*}
V = \{f \in \cH^1_0(U): \Er[f] = 0\},
\end{align*}
to define the unique solution $v - \ell_{p,U} \in V$. We notice that the right hand side of \cref{eq:EllipNu} is clearly a bounded linear functional, and the coercivity of the left hand side of \cref{eq:EllipNu} is ensured by the Poincar\'e inequality \Cref{prop:Poincare1} on $V$.  We denote this minimizer by $v(\mu, U, p)$, and \cref{eq:EllipNu} implies that $v(\mu, U, p) \in \mcl A(U)$.

\medskip

(2) We test at first \cref{eq:EllipNu} with ${v(\mu, U, p') - \ell_{p',U} \in \cH_{0}^1(U)}$ and obtain that
\begin{multline}\label{eq:vari1TestAffine}
\Er\Ll[\int_{U} \nabla v(\mu,x,U,p) \cdot \a(\mu, x) \nabla v(\mu,x,U,p') \, \d \mu(x)\Rr] \\= \Er\Ll[\int_{U} \nabla v(\mu, x, U,p) \cdot \a(\mu, x) p'  \, \d \mu(x) \Rr],
\end{multline}
and this implies ${(p,p') \mapsto \Er\Ll[\frac{1}{\rho \vert U \vert}\int_{U} \nabla v(\mu,x,U,p) \cdot \a(\mu, x) \nabla v(\mu,x,U,p') \, \d \mu(x)\Rr]}$ is a bilinear map $p \cdot \ab(U) p'$. This definition with \cref{eq:vari1TestAffine}
proves \cref{eq:flux1}. We let $p = p'$ and obtain that $\nu(U, p) = \frac{1}{2}p \cdot \ab(U) p$. To obtain the bound of $\ab(U)$, we use the bound of $\a$ and the definition of \cref{eq:defNu}
\begin{multline*}
\inf_{v \in \ell_{p,U} + \cH^1_0(U)} \Er \Ll[\frac{1}{\rho \vert U \vert}\int_{U} \frac{1}{2} \vert \nabla v \vert^2 \, \d \mu \Rr] \leq \nu(U, p) = \frac{1}{2}p \cdot \ab(U) p  
\\ \leq \inf_{v \in \ell_{p,U} + \cH^1_0(U)} \Er \Ll[\frac{1}{\rho \vert U \vert}\int_{U} \frac{\Lambda}{2} \vert \nabla v \vert^2 \, \d \mu \Rr].
\end{multline*}
We can check that $\ell_{p,U}$ is the minimizer for $\inf_{v \in \ell_{p,U} + \cH^1_0(U)} \Er \Ll[\int_{\Rd} \frac{\Lambda}{2} \vert \nabla v \vert^2 \, \d \mu \Rr],$
then it concludes the proof of the bound ${\id \leq \ab(U) \leq \Lambda \id}$.

The same argument works for $\nu^*(U,q)$. We test \cref{eq:EllipNuDual} with $u(\mu, U, q')$ and obtain that
\begin{multline}\label{eq:vari2TestAffine}
\Er \Ll[\int_{U}  \nabla u(\mu, x, U, q) \cdot \a(\mu, x) \nabla u(\mu, x, U, q') \, \d \mu (x)\Rr] \\
= \Er \Ll[ \int_U q \cdot \nabla u(\mu, x, U, q') \, \d \mu(x) \Rr].
\end{multline}
This proves that $(q,q') \mapsto \Er \Ll[\frac{1}{\rho \vert U \vert}\int_{U}  \nabla u(\mu, x, U, q) \cdot \a(\mu, x) \nabla u(\mu, x, U, q') \Rr]$ is also bilinear and we denote it by $q \cdot \ab_*^{-1}(U) q'$, and this also concludes \cref{eq:flux2}. Then we put $q' = q$ and \cref{eq:vari2TestAffine} in the definition of \cref{eq:defNu} that
\begin{align*}
& \nu^*(U,q) 
\\&  =  \Er \Ll[\frac{1}{\rho \vert U \vert}\int_{U} \Ll( -\frac 1 2 \nabla u(\mu, x, U, q) \cdot \a(\mu, x) \nabla u(\mu, x, U, q) + q \cdot \nabla u(\mu, x, U, q) \Rr) \, \d \mu(x) \Rr] \\
& =\Er \Ll[\frac{1}{\rho \vert U \vert}\int_{U}  \frac 1 2 \nabla u(\mu, x, U, q) \cdot \a(\mu, x) \nabla u(\mu, x, U, q)  \, \d \mu(x) \Rr] \\
& = \frac{1}{2} q \cdot \ab_*^{-1}(U) q.
\end{align*}
This proves the bilinear map expression for $\nu^*(U, q)$. Concerning the bound for the matrix $\ab_*^{-1}(U)$, we use the bound for $\a$ and the equations above to obtain that 
\begin{multline}
\sup_{u \in \cH^1(U)} \Er \Ll[\frac{1}{\rho \vert U \vert}\int_{U} \Ll( -\frac \Lambda 2  \vert \nabla u \vert^2 + q\cdot \nabla u \Rr) \, \d \mu \Rr] \leq  \nu^*(U,q) \\
\leq \sup_{u \in \cH^1(U)} \Er \Ll[\frac{1}{\rho \vert U \vert}\int_{U} \Ll( -\frac 1 2 \vert \nabla u \vert^2 + q \cdot \nabla u \Rr) \, \d \mu \Rr].
\end{multline}
One can check for the lower bound, $\ell_{\frac{q}{\Lambda},U}$ attains the maximum and for the upper bound it is $\ell_{q,U}$ that attains the maximum. Then we put the expression $\nu^*(U, q) = \frac{1}{2} q \cdot \ab_*^{-1}(U) q$ and obtain that 
\begin{align*}
\frac{\Lambda^{-1}}{2}\vert q \vert^2 \leq \nu^*(U, q) = \frac{1}{2} q \cdot \ab_*^{-1}(U) q \leq \frac{1}{2} \vert q \vert^2,
\end{align*}
which implies the bound for $\ab_*(U)$.

\medskip 

(3) The slope identity \cref{eq:Slope1} for $v(\mu,U,p)$ is directly the result from \cref{eq:Integral0} that 
\begin{align*}
\Er\Ll[\fint_{U} \nabla v(\mu, x, U, p) \, \d \mu(x) \ \Bigg\vert \ \mu(U)\Rr] = \Er\Ll[\fint_{U} p \, \d \mu \ \Bigg\vert \ \mu(U)\Rr] = p.
\end{align*}
For the function $u(\mu, U, q)$, the identity \cref{eq:Slope3} comes directly from \cref{eq:flux2}, but conditioned $\G_U$, the averaged slope is not $\ab_*^{-1}(U)q$. In fact, we recall that $u(\mu, U, q)$ is also the conditioned maximizer for \cref{eq:EllipNuDualSmall}, so we can define the matrix ${\a_*^{-1}(U; \G_U)}$, the quenched slope \cref{eq:Slope2}. The estimate for this matrix is then obtained by repeating the argument in \cref{eq:NuMatrix} for \cref{eq:EllipNuDualSmall}.

\medskip

(4) We test \cref{eq:EllipNu} with $(v' - \ell_{p,U})$ and put it in the left hand side of \cref{eq:QuadraRep1}
\begin{equation}\label{eq:QuadraRep1Step1}
\begin{split}
& \Er \Ll[\frac{1}{\rho \vert U \vert}\int_{U} \frac{1}{2} \nabla(v' - v(\cdot,U,p)) \cdot \a \nabla(v' - v(\cdot,U,p)) \, \d \mu \Rr] \\
=& \Er \Ll[\frac{1}{\rho \vert U \vert}\int_{U} \frac{1}{2} \nabla v' \cdot \a \nabla v' \, \d \mu \Rr] + \Er \Ll[\frac{1}{\rho \vert U \vert}\int_{U} \frac{1}{2} \nabla  v(\cdot,U,p) \cdot \a \nabla  v(\cdot,U,p) \, \d \mu \Rr] \\
& \qquad - \Er \Ll[\frac{1}{\rho \vert U \vert}\int_{U}  \nabla v' \cdot \a \nabla  v(\cdot,U,p) \, \d \mu \Rr]\\
=& \Er \Ll[\frac{1}{\rho \vert U \vert}\int_{U} \frac{1}{2} \nabla v' \cdot \a \nabla v' \, \d \mu \Rr] + \Er \Ll[\frac{1}{\rho \vert U \vert}\int_{U} \frac{1}{2} \nabla  v(\cdot,U,p) \cdot \a \nabla  v(\cdot,U,p) \, \d \mu \Rr] \\
& \qquad - \Er \Ll[\frac{1}{\rho \vert U \vert}\int_{U}  p \cdot \a \nabla  v(\cdot,U,p) \, \d \mu \Rr].
\end{split}
\end{equation}
The term $ \Er \Ll[\frac{1}{\rho \vert U \vert}\int_{U}  p \cdot \a \nabla  v(\cdot,U,p) \, \d \mu \Rr]$ also appears on the right side of \cref{eq:flux1} with $p = p'$, thus we obtain that 
\begin{align*}
\Er \Ll[\frac{1}{\rho \vert U \vert}\int_{U}  p \cdot \a \nabla  v(\cdot,U,p) \, \d \mu \Rr] = p \cdot \ab(U) p = \Er \Ll[\frac{1}{\rho \vert U \vert}\int_{U}  \nabla  v(\cdot,U,p) \cdot \a \nabla  v(\cdot,U,p) \, \d \mu \Rr],
\end{align*}
and we put it back to \cref{eq:QuadraRep1Step1} to conclude for the validity of \cref{eq:QuadraRep1}.

Similarly, we develop the left hand side of \cref{eq:QuadraRep2} as \cref{eq:QuadraRep1Step1}, and use \cref{eq:EllipNuDual} with $ \phi = u'$ to treat the inner product term of $u'$ and $u(\cdot, U, q)$. 
\begin{align*}
\Er \Ll[\frac{1}{\rho \vert U \vert}\int_{U}  \nabla u' \cdot \a \nabla  u(\cdot,U,q) \, \d \mu \Rr] &=  \Er \Ll[\frac{1}{\rho \vert U \vert}\int_{U}  \nabla u' \cdot q   \, \d \mu \Rr]. 
\end{align*}
We put this  term in the left hand side of \cref{eq:QuadraRep2} and use the bilinear map expression of $\nu^*(U, q)$ to obtain that 
\begin{align*}
& \Er \Ll[\frac{1}{\rho \vert U \vert}\int_{U} \frac{1}{2} \nabla(u' - u(\cdot,U,q)) \cdot \a \nabla(u' - u(\cdot,U,q)) \, \d \mu \Rr] \\
=& \Er \Ll[\frac{1}{\rho \vert U \vert}\int_{U} \frac{1}{2} \nabla u' \cdot \a \nabla u' \, \d \mu \Rr] + \Er \Ll[\frac{1}{\rho \vert U \vert}\int_{U} \frac{1}{2} \nabla  u(\cdot,U,q) \cdot \a \nabla  u(\cdot,U,q) \, \d \mu \Rr] \\
& \qquad - \Er \Ll[\frac{1}{\rho \vert U \vert}\int_{U}  \nabla u' \cdot \a \nabla  u(\cdot,U,q) \, \d \mu \Rr]\\
=&  \Er \Ll[\frac{1}{\rho \vert U \vert}\int_{U} \frac{1}{2} \nabla u' \cdot \a \nabla u' \, \d \mu \Rr] + \nu^*(U,q) -   \Er \Ll[\frac{1}{\rho \vert U \vert}\int_{U}  \nabla u' \cdot q   \, \d \mu \Rr].
\end{align*}
This concludes the proof of \cref{eq:QuadraRep2}.

\medskip

(5) For $\nu(\cu_{n+1}, p)$, we test the associated variational problem with the candidate $v' = \sum_{z \in \Z_{n+1,n}}v(\cdot, z + \cu_{n}, p)$, which is an element of $\ell_{p, \cu_{n+1}} + \cH^1_0(\cu_{n+1})$, so that
\begin{align*}
\nu(\cu_{n+1}, p) &\leq \Er \Ll[\frac{1}{\rho \vert \cu_{n+1} \vert}\int_{\cu_{n+1}}  \nabla v' \cdot \a \nabla  v' \, \d \mu \Rr] \\
&= 3^{-d} \sum_{z \in \Z_{n+1,n} }  \Er \Ll[\frac{1}{\rho \vert\cu_n \vert}\int_{z + \cu_n}  \nabla v(\cdot, z + \cu_{n}, p)  \cdot \a  \nabla  v(\cdot, z + \cu_{n}, p)  \, \d \mu \Rr] \\
&= \nu(\cu_n, p).
\end{align*}
In the last step, we also use the stationarity of the coefficient field	 $\a$. 

For $\nu^{*}(\cu_{n+1}, p)$, we also use that, for every $z \in \Z_{n+1, n}$, we have the inclusion ${\cH^1(\cu_{n+1}) \subset \cH^1(z + \cu_n)}$,  so its unit energy on every small cube $z + \cu_n$ is less than the maximum $\nu^{*}(z + \cu_n, p)$, thus 
\begin{align*}
& \nu^*(\cu_{n+1}, q)\\
= & 3^{-d} \sum_{z \in \Z_{n+1,n}} \Er \Ll[\frac{1}{\rho \vert \cu_{n}\vert}\int_{z + \cu_n}  - \frac{1}{2} \nabla u(\cdot, \cu_{n+1}, q) \cdot \a \nabla  u(\cdot, \cu_{n+1}, q) + q \cdot \nabla u(\cdot, \cu_{n+1}, q)\, \d \mu \Rr] \\
\leq & 3^{-d} \sum_{z \in \Z_{n+1,n}} \nu^*(z + \cu_n, q)\\
= & \nu^*(\cu_n, q). \qedhere
\end{align*}
\end{proof}

\subsection{Subadditive quantity \texorpdfstring{$J$}{J}}

We now study the quantity $J$ defined by \begin{equation}\label{eq:defJ}
\begin{split}
J(U,p,q) :&= \nu(U,p) + \nu^*(U,q) - p \cdot q \\
    & = \frac{1}{2} p \cdot \ab(U)p + \frac{1}{2} q \cdot \ab_*^{-1}(U) q - p \cdot q.
\end{split}
\end{equation}
By the properties of $\nu$ and $\nu^*$, the quantity $J$ is also subadditive. We briefly explain why this quantity will be convenient for our purposes. If the functions $\nu(U,\cdot)$ and $\nu^*(U,\cdot)$ were exactly convex dual of one another, then we would have that $J \ge 0$ and that for every $p \in \Rd$, the infimum of $J(U,p,\cdot)$ is zero. This would correspond to the situation in which $\ab(U)$ and $\ab_*(U)$ are equal, and for every $p \in \Rd$, we would in fact have that $J(U,p,\ab(U)p) = 0$. Instead, we will show below that, for any symmetric matrix $\id \leq \tilde{\a} \leq \Lambda \id$, we have  
\begin{align*}
\vert \tilde{\a} - \ab(U)\vert + \vert \tilde{\a} - \a_{*}(U)\vert \leq \sup_{p \in B_1} C (J(U, p, \tilde{\a}p))^{\frac{1}{2}}.   
\end{align*}
The right side of the inequality above can be thought of as a measure of the defect in the convex duality relationship between $\nu$ and $\nu^*$. For $U = \cu_m$ and using $\tilde{\a} = \ab_*(\cu_m)$, we obtain that 
\begin{align*}
\vert \ab_{*}(\cu_m) - \ab(\cu_m)\vert \leq \sup_{p \in B_1}  C (J(U, p, \ab_*(\cu_m)p))^{\frac{1}{2}}.   
\end{align*}
Since we know that $\{\ab(\cu_m) \}_{m \geq 0}$ is a decreasing sequence while $\{\ab_{*}(\cu_m)\}_{m \geq 0}$ is a increasing sequence from \cref{eq:NuMatrix} and \cref{eq:Subadditive}, each sequence has a limit. 
Therefore, once we prove a rate of convergence to zero for $J(U, p, \ab_*(\cu_m)p)$, we get that the two limits coincide, and also a rate for the convergence of $\{\ab(\cu_m) \}_{m \geq 0}$.

The rest of this section will present this strategy in details. We establish at first a variational description for the quantity $J$ and the properties mentioned above.

\begin{lemma}
(1) For every $p,q \in \Rd$, we have the variational representation
\begin{equation}\label{eq:defJ2}
J(U,p,q) = \sup_{w \in \mcl{A}(U)} \Er \Ll[\frac{1}{\rho \vert U \vert}\int_{U} \Ll( -\frac 1 2 \nabla w \cdot \a \nabla w - p \cdot \a \nabla w + q \cdot \nabla w \Rr) \, \d \mu \Rr].
\end{equation}

(2) We have that $J(U, p ,q) \geq 0$ and $\ab(U) \geq \ab_*(U)$.

(3) There exists a constant $C(d,\Lambda) < \infty$ such that and for every symmetric matrix~$\tilde{\a}$ satisfying ${\id \leq \tilde{\a} \leq \Lambda \id}$, we have
\begin{align}\label{eq:JBoundError}
\vert \tilde{\a} - \ab(U)\vert + \vert \tilde{\a} - \ab_{*}(U)\vert \leq C \sup_{p \in B_1} (J(U, p, \tilde{\a}p))^{\frac{1}{2}}.
\end{align}
\end{lemma}
\begin{proof}
(1) We start by rewriting the expression of $J(U, p, q)$ using the definition of $\nu^*(U,q)$ and the quadratic expression of $\nu(U,p)$. Noting also that the maximizer of $\nu^*(U,q)$ belongs to $\mcl{A}(U)$, we can write
\begin{multline}\label{eq:JOneFree}
J(U,p,q) = \Er \Ll[\frac{1}{\rho \vert U \vert}\int_{U}  \frac 1 2 \nabla v(\cdot, U, p) \cdot \a \nabla v(\cdot, U, p)  \, \d \mu \Rr] \\
+ \sup_{u \in \mcl{A}(U)} \Er \Ll[\frac{1}{\rho \vert U \vert}\int_{U} \Ll( -\frac 1 2 \nabla u \cdot \a \nabla u  + q \cdot \nabla u \Rr) \, \d \mu \Rr] - p \cdot q.
\end{multline}
We claim that for any ${u \in  \mcl{A}(U)}$, with ${w : = u -  v(\cdot, U, p)}$, we have 
\begin{multline}\label{eq:Jidentity}
\Er \Ll[\frac{1}{\rho \vert U \vert}\int_{U} \Ll( \frac 1 2 \nabla v(\cdot, U, p) \cdot \a \nabla v(\cdot, U, p)  -\frac 1 2 \nabla u \cdot \a \nabla u  + q \cdot \nabla u \Rr) \, \d \mu \Rr] - p \cdot q\\
 = \Er \Ll[\frac{1}{\rho \vert U \vert}\int_{U} \Ll( -\frac 1 2 \nabla w \cdot \a \nabla w - p \cdot \a \nabla w + q \cdot \nabla w \Rr) \, \d \mu \Rr].
\end{multline}
To prove it, we can develop the right hand side of \cref{eq:Jidentity}
\begin{equation}\label{eq:JDifference}
\begin{split}
&\Er \Ll[\frac{1}{\rho \vert U \vert}\int_{U} \Ll( -\frac 1 2 \nabla w \cdot \a \nabla w - p \cdot \a \nabla w + q \cdot \nabla w \Rr) \, \d \mu \Rr] \\
= & \Er \Ll[\frac{1}{\rho \vert U \vert}\int_{U} \Ll( \frac 1 2 \nabla v(\cdot, U, p) \cdot \a \nabla v(\cdot, U, p)  -\frac 1 2 \nabla u \cdot \a \nabla u  + q \cdot \nabla u \Rr) \, \d \mu \Rr] - p \cdot q \\
& \qquad + \Er \Ll[\frac{1}{\rho \vert U \vert}\int_{U}  \nabla \Ll(v(\cdot, U, p) - \ell_{p,U} \Rr) \cdot \Ll(\a \nabla u  -  \a \nabla v(\cdot, U, p) - q \ \Rr) \, \d \mu \Rr]. \\
\end{split}
\end{equation}
Because $\Ll(v(\cdot, U, p) - \ell_{p,U}\Rr) \in \cH^1_0(U)$, we apply $u, v(\cdot, U, p) \in \mcl{A}(U)$ and \cref{eq:Integral0},  the last line of \cref{eq:JDifference} is $0$ and we prove \cref{eq:Jidentity}. Then we take the maximum as \cref{eq:JOneFree} and obtain the definition \cref{eq:defJ2}.

(2) The properties that $J(U, p, q) \geq 0$ comes from the definition of $\nu^*(U, q)$: we test the functional in the definition of $\nu^*(U,q)$ with the minimizer $v(\cdot, U, p)$ of $\nu(U, p)$ and obtain that 
\begin{align*}
& \nu^*(U, q) 
\\& \quad  \geq \Er \Ll[\frac{1}{\rho \vert U \vert}\int_{U} \Ll( -\frac 1 2 \nabla v(\mu, x, U, p) \cdot \a(\mu, x) \nabla v(\mu,x, U, p) + q \cdot \nabla v(\mu,x, U, p) \Rr) \, \d \mu \Rr] 
\\
& \quad=  p \cdot q - \nu(U, p),
\end{align*}
so that 
\begin{equation*}  
J(U, p, q) = \nu(U, p) + \nu^*(U, q) - p\cdot q \geq 0.
\end{equation*}
Then we test $J(U, p, q) \geq 0$ with that $q = \ab_*(U)p$ and obtain that 
\begin{align*}
0 \leq J(U, p, \ab_*(U)p) &= \frac{1}{2} p \cdot \ab(U)p + \frac{1}{2} (\ab_*(U)p) \cdot \ab_*^{-1}(U) (\ab_*(U)p) - p \cdot \ab_*(U)p,
\end{align*}
and therefore $\ab(U) \geq \ab_*(U)$.

(3) 
Using this property, we have 
\begin{align*}
J(U, p, q) &= \frac{1}{2} p \cdot \ab(U)p + \frac{1}{2} q \cdot \ab_*^{-1}(U) q - p \cdot q \\
& \geq \frac{1}{2} p \cdot \ab(U)p + \frac{1}{2} q \cdot \ab^{-1}(U) q - p \cdot q \\
& = \frac{1}{2}(\ab(U)p - q) \ab^{-1}(U) \cdot (\ab(U)p - q).
\end{align*}
We put $q = \tilde{\a} p$ and obtain $\vert \ab(U) - \tilde{\a}\vert \leq C \sup_{p \in B_1}  (J(U, p, \tilde{\a}p))^{\frac{1}{2}}$. The proof of the statement concerning $\vert \ab_*(U) - \tilde{\a}\vert$ is similar. 
\end{proof}

In view of the definition of $J$, this functional enjoys properties similar to those described in \Cref{prop:NuBase} for $\nu$ and $\nu^*$.
\begin{proposition}[Elementary properties of $J$]  For every bounded domain $U \subset \Rd$ with Lipschitz boundary and $p,q \in \Rd$, the quantity $J(U, p, q)$ defined in \cref{eq:defJ} satisfies the following properties:

(1) Characterization of optimizer: the optimization problem in \cref{eq:defJ2} admits a unique solution $v(\cdot, U, p, q) \in \cH^1(U)$ such that $\Er[v(\cdot, U, p, q) \, \vert \, \G_{U}] = 0$. This solution is such that for every $w \in \mcl A (U)$,
\begin{align}\label{eq:JVariation}
\Er\Ll[ \int_U \nabla v(\cdot, U, p, q) \cdot \a  \nabla w \, \d \mu \Rr] = \Er\Ll[ \int_U \Ll(-p \cdot \a \nabla w + q \cdot \nabla w\Rr)  \, \d \mu \Rr],
\end{align}
and $(p,q) \mapsto v(\cdot, U, p, q)$ a linear map. The function $v(\cdot, U, p, q)$ can be expressed in terms of the optimizers in \cref{eq:defNu} as
\begin{align}\label{eq:JVExpression}
v(\mu,U,p,q) = u(\mu,U,q) - v(\mu,U,p) - \Er[u(\mu,U,q) - v(\mu,U,p) \, \vert \, \G_U].
\end{align}
We have the quadratic expression
\begin{align}\label{eq:JEnergy}
J(U, p, q) = \Er\Ll[\frac{1}{\rho \vert U \vert} \int_U \frac{1}{2} \nabla v(\cdot, U, p, q) \cdot \a  \nabla v(\cdot, U, p, q) \, \d \mu \Rr].
\end{align}

(2) Slope: $v(\cdot, U, p , q)$ satisfies
\begin{equation}\label{eq:JSlope}
\begin{split}
\Er\Ll[\fint_U \nabla v(\cdot, U, p, q) \, \d \mu  \ \Bigg \vert \ \G_U\Rr] &= \a^{-1}_*(U; \G_U)q - p, \\
\Er\Ll[\frac{1}{\rho \vert U \vert} \int_U \nabla v(\cdot, U, p, q) \, \d \mu \Rr] &= \ab_*^{-1}(U)q - p,
\end{split}
\end{equation}
where the matrix $\a^{-1}_*(U; \G_U)$ is defined in \cref{eq:Slope2}.

(3) Quadratic response: for every $w \in \mcl{A}(U)$, we have 
\begin{multline}\label{eq:JQuadra}
\Er\Ll[\frac{1}{\rho \vert U \vert} \int_U \Ll(\frac{1}{2} \nabla(w - v(\cdot, U, p, q)) \cdot \a  \nabla (w - v(\cdot, U, p, q))\Rr) \, \d \mu \Rr] \\
= J(U, p, q) - \Er \Ll[\frac{1}{\rho \vert U \vert}\int_{U} \Ll( -\frac 1 2 \nabla w \cdot \a \nabla w - p \cdot \a \nabla w + q \cdot \nabla w \Rr) \, \d \mu \Rr].
\end{multline}

(4) Subadditivity: for every $n \in \N$,  we have
\begin{align}\label{eq:JSubadditive}
J(\cu_{n+1}, p, q) \leq J(\cu_{n}, p, q).
\end{align}
\end{proposition}
\begin{proof}
(1) The equation \cref{eq:JVariation} comes directly from the first order variation calculus. The proof of the existence and uniqueness of the solution $v(\cdot, U, p ,q)$ is similar as the one for $\nu^*(U,q)$. \Cref{eq:JVariation} also implies that the map $(p,q) \mapsto v(\cdot, U, p ,q)$ is linear because for any $p_1, p_2, q_1, q_2 \in \Rd$, and any $w \in \mcl{A}(U)$ we have 
\begin{align*}
& \Er\Ll[ \int_U \nabla v(\cdot, U, p_1 + p_2, q_1 + q_2) \cdot \a  \nabla w \, \d \mu \Rr] 
\\
& \qquad = \Er\Ll[ \int_U -(p_1 + p_2)\cdot \a \nabla w + (q_1 + q_2) \cdot \nabla w  \, \d \mu \Rr] \\
& \qquad = \Er\Ll[ \int_U \nabla (v(\cdot, U, p_1, q_1) + v(\cdot, U, p_2, q_2)) \cdot \a  \nabla w \, \d \mu \Rr].
\end{align*}
Then $(v(\cdot, U, p_1, q_1) + v(\cdot, U, p_2, q_2))$ is also a solution for the problem \cref{eq:JVariation} with parameter $(p_1 + p_2, q_1 + q_2)$. Notice that we have 
\begin{align*}
\Er[(v(\cdot, U, p_1, q_1) + v(\cdot, U, p_2, q_2)) \, \vert \, \G_U] = 0,
\end{align*}
it implies $v(\mu, U, p_1 + p_2, q_1 + q_2) = v(\mu, U, p_1, q_1) + v(\mu, U, p_2, q_2)$ and the linearity of the map.

The exact expression of $v(\mu, U, P, q)$ comes from the equivalent definition \cref{eq:defJ2} of $J(U,p,q)$ and its proof. We put $v(\mu, U, p, q)$ in the first order variation \cref{eq:JVariation}
\begin{multline*}
\Er \Ll[\frac{1}{\rho \vert U \vert}\int_{U} \Ll( - p \cdot \a \nabla v(\cdot, U, p, q) + q \cdot \nabla v(\cdot, U, p, q) \Rr) \, \d \mu \Rr] \\
= \Er \Ll[\frac{1}{\rho \vert U \vert}\int_{U}  \nabla v(\cdot,U,p,q) \cdot \a \nabla v(\cdot,U,p,q)  \, \d \mu \Rr].
\end{multline*}
Then we put this equation into \cref{eq:defJ2} to get \cref{eq:JEnergy}.

(2) The slope identity \cref{eq:JSlope} comes from \cref{eq:JVExpression},  \eqref{eq:Slope1}, \eqref{eq:Slope2}, and \eqref{eq:Slope3}.

(3) We use the expression in \cref{eq:JVExpression} with $w := u'-v(\cdot, U, p)$, then we use the quadratic response for $\nu^*(U,q)$ \cref{eq:QuadraRep2} that 
\begin{align*}
&\Er\Ll[\frac{1}{\rho \vert U \vert} \int_U \Ll(\frac{1}{2} \nabla(w - v(\cdot, U, p, q)) \cdot \a  \nabla (w - v(\cdot, U, p, q))\Rr) \, \d \mu \Rr] \\
& \qquad = \Er\Ll[\frac{1}{\rho \vert U \vert} \int_U \Ll(\frac{1}{2} \nabla(u' - u(\cdot, U, q)) \cdot \a  \nabla (u' - u(\cdot, U, q))\Rr) \, \d \mu \Rr]\\
& \qquad =  \nu^*(U, q) - \Er\Ll[\frac{1}{\rho \vert U \vert} \int_U \Ll(- \frac{1}{2} \nabla u' \cdot \a  \nabla u' + q \cdot \nabla u' \Rr) \, \d \mu \Rr].
\end{align*}
Then we add back the term $\nu(U,p)$ and it gives the desired result
\begin{align*}
&\Er\Ll[\frac{1}{\rho \vert U \vert} \int_U \Ll(\frac{1}{2} \nabla(w - v(\cdot, U, p, q)) \cdot \a  \nabla (w - v(\cdot, U, p, q))\Rr) \, \d \mu \Rr] \\
& \qquad = J(U,p,q) - \Ll(\nu(U,p) + \Er\Ll[\frac{1}{\rho \vert U \vert} \int_U \Ll(- \frac{1}{2} \nabla u' \cdot \a  \nabla u' + q \cdot \nabla u' \Rr) \, \d \mu \Rr] - p \cdot q \Rr) \\
& \qquad = J(U,p,q) - \Er \Ll[\frac{1}{\rho \vert U \vert}\int_{U} \Ll( -\frac 1 2 \nabla w \cdot \a \nabla w - p \cdot \a \nabla w + q \cdot \nabla w \Rr) \, \d \mu \Rr].
\end{align*}

(4) \Cref{eq:JSubadditive} is a consequence of \cref{eq:Subadditive} and \cref{eq:defJ}.
\end{proof}

We conclude this section with the following lemma.
\begin{lemma}[Comparison between two scales]\label{lem:JTwoScale}
For every $n,k \in \mathbb{N}$ with $k \leq n$, and $p,q \in \Rd$, writing $v(U)$ as shorthand for $v(\cdot,U, p, q)$, we have
\begin{multline}\label{eq:JTwoScale}
\frac{1}{\vert \Z_{n,k} \vert} \sum_{z \in \Z_{n,k}} \Er\Ll[\frac{1}{\rho \vert \cu_k \vert} \int_{z + \cu_k} \frac{1}{2} \vert \nabla v(\cu_n) - \nabla v(z+\cu_k) \vert^2 \, \d \mu \Rr]
\\
\leq J(\cu_k, p, q) - J(\cu_n, p, q).
\end{multline}
\end{lemma}
\begin{proof}
For any $z \in \Z_{n,k}$, since $v(\cu_n) \in \mcl A(z+\cu_k)$, we use the quadratic response \cref{eq:JQuadra} for $J(z+\cu_k, p, q)$ that 
\begin{align*}
&\Er\Ll[\frac{1}{\rho \vert \cu_k \vert} \int_{z + \cu_k} \frac{1}{2} \vert \nabla v(\cu_n) - \nabla v(z+\cu_k) \vert^2 \, \d \mu \Rr] \\
& \qquad \leq \Er\Ll[\frac{1}{\rho \vert \cu_k \vert} \int_{z + \cu_k} \frac{1}{2} (\nabla v(\cu_n) - \nabla v(z+\cu_k)) \cdot \a (\nabla v(\cu_n) - \nabla v(z+\cu_k)) \, \d \mu \Rr] \\
& \qquad =  J(z+\cu_k, p, q) 
\\ & \qquad \qquad - \Er \Ll[\frac{1}{\rho \vert \cu_k \vert}\int_{z+\cu_k} \Ll( -\frac 1 2 \nabla v(\cu_n) \cdot \a \nabla v(\cu_n) - p \cdot \a \nabla v(\cu_n) + q \cdot \nabla v(\cu_n) \Rr) \, \d \mu \Rr].
\end{align*}
We sum this expression over all $z \in \Z_{n,k}$ to obtain that 
\begin{align*}
&\frac{1}{\vert \Z_{n,k} \vert} \sum_{z \in \Z_{n,k}} \Er\Ll[\frac{1}{\rho \vert \cu_k \vert} \int_{z + \cu_k} \frac{1}{2} \vert \nabla v(\cu_n) - \nabla v(z+\cu_k) \vert^2 \, \d \mu \Rr] \\
& \quad \leq  \frac{1}{\vert \Z_{n,k} \vert} \sum_{z \in \Z_{n,k}} \bigg( J(z+\cu_k, p, q) 
\\
& \quad  \qquad - \Er \Ll[\frac{1}{\rho \vert \cu_k \vert}\int_{z+\cu_k} \Ll( -\frac 1 2 \nabla v(\cu_n) \cdot \a \nabla v(\cu_n) - p \cdot \a \nabla v(\cu_n) + q \cdot \nabla v(\cu_n) \Rr) \, \d \mu \Rr]\bigg)\\
& \quad =  J(\cu_k, p ,q) - J(\cu_n, p ,q).
\end{align*}
In the last step, we use the stationarity of $J$ and also \cref{eq:defJ2} for $v(\cu_n)$.
\end{proof}

%
%
%
%
%
%
%
%

\section{Quantitative rate of convergence}
\label{subsec:ConvergenceJ}
We are now ready to prove \Cref{thm:main}. We decompose the argument into a series of four steps.

\subsection{Step 1: setup}
We use the shorthand $\ab_n := \ab_*(\cu_n)$, so that by \cref{eq:JSlope}, the average slope of the function $v(\cdot, \cu_n, p, q)$ is $\ab_n^{-1} q - p$, in the sense that
\begin{align}\label{eq:SlopABN}
\Er\Ll[\frac{1}{\rho \vert \cu_n \vert} \int_U \nabla v(\cdot, \cu_n, p, q) \, \d \mu \Rr] &= \ab_n^{-1} q - p.
\end{align}
We let $\tau_n$ denote a measure of the defect in the subadditivity of $J$, precisely,
\begin{equation}\label{eq:defErrorJ}
\begin{split}
\tau_n :=& \sup_{p,q \in B_1} (J(\cu_n, p ,q) - J(\cu_{n+1}, p , q)) \\
=& \sup_{p \in B_1}(\nu(\cu_n, p) - \nu(\cu_{n+1},p)) + \sup_{q \in B_1}(\nu^*(\cu_n, q) - \nu^*(\cu_{n+1},q)).
\end{split}
\end{equation} 
A direct corollary from \cref{eq:defErrorJ} is that for any integers $n < m$,
\begin{align}\label{eq:ErrorABN}
\vert \ab_n^{-1} - \ab_m^{-1} \vert = \sup_{q \in B_1} q \cdot (\ab_n^{-1} - \ab_m^{-1}) q = \sup_{q \in B_1} \Ll(\nu^*(\cu_{n},q) - \nu^*(\cu_{m},q)\Rr) \leq C \sum_{k=n}^{m-1} \tau_k. 
\end{align}
We recall that $\{\ab(\cu_m)\}_{m \geq 0}$ is decreasing and $\{\ab_*(\cu_m)\}_{m \geq 0}$ is increasing, with the comparison $\ab_*(\cu_m) \leq \ab(\cu_m)$.  From \cref{eq:JBoundError}, we know that 
\begin{align*}
\vert \ab(\cu_m) - \ab\vert \le \vert \ab(\cu_m) - \ab_*(\cu_m)\vert \leq C\sup_{p \in B_1} (J(\cu_m, p, \ab_m p))^{\frac{1}{2}}.
\end{align*}
From now on, we thus fix $p \in B_1$, and focus on estimating $J(\cu_m, p, \ab_m p)$. We also assume without further notification that $m$ is sufficiently large that $3^m \ge R_0$, for the constant $R_0$ appearing in \Cref{prop:Caccioppoli}. We use ${\A_{3^m+2} v(\cdot, \cu_{m+1}, p, \ab_m p)}$ to compare with \cref{eq:JEnergy} and apply the quadratic response \cref{eq:JQuadra}. In the rest of \text{Step 1}, we write $v(U)$ as a shorthand for $v(\cdot, U, p, \ab_m p)$, and decompose
\begin{equation}\label{eq:JamDecom}
\begin{split}
(J(\cu_m, p, \ab_m p))^{\frac{1}{2}} &=  \Ll(\Er\Ll[\frac{1}{\rho \vert \cu_m \vert} \int_{\cu_m} \frac{1}{2} \nabla v(\cu_m) \cdot \a  \nabla v(\cu_m) \, \d \mu \Rr]\Rr)^{\frac{1}{2}} \\
&\leq \text{\cref{eq:JamDecom}-a} + \text{\cref{eq:JamDecom}-b},
\end{split}
\end{equation}
with
\begin{multline*}  
\text{\cref{eq:JamDecom}-a} 
\\
= \Ll(\Er\Ll[\frac{1}{\rho \vert \cu_m \vert} \int_{\cu_m} \frac{1}{2} (\nabla v(\cu_m) - \nabla \A_{3^m+2} v(\cu_{m+1})) \cdot \a   (\nabla v(\cu_m) - \nabla \A_{3^m+2} v(\cu_{m+1})) \, \d \mu \Rr]\Rr)^{\frac{1}{2}},
\end{multline*}
and
\begin{equation*}  
\text{\cref{eq:JamDecom}-b} = \Ll(\Er\Ll[\frac{1}{\rho \vert \cu_m \vert} \int_{\cu_m} \frac{1}{2}  \nabla \A_{3^m+2} v(\cu_{m+1}) \cdot \a   \nabla \A_{3^m+2} v(\cu_{m+1}) \, \d \mu \Rr]\Rr)^{\frac{1}{2}}.
\end{equation*}
We treat the two terms separately. For \cref{eq:JamDecom}-a, since ${\A_{3^m+2} v(\cu_{m+1}) \in \mcl A(\cu_m)}$ (see \Cref{prop:AppendixLocalization} for details),  we use \cref{eq:JQuadra} to get
\begin{align*}
&  \vert \text{\cref{eq:JamDecom}-a} \vert^2 
 \\
 & \qquad  = J(\cu_m, p, \ab_m p) 
- \Er \Ll[\frac{1}{\rho \vert \cu_m \vert}\int_{\cu_m} \Ll( -\frac 1 2 \nabla \A_{3^m+2} v(\cu_{m+1}) \cdot \a \nabla \A_{3^m+2} v(\cu_{m+1}) \Rr) \, \d \mu \Rr] 
 \\
 & \qquad  \qquad - \Er \Ll[\frac{1}{\rho \vert \cu_m \vert}\int_{\cu_m} \Ll( - p \cdot \a \nabla \A_{3^m+2} v(\cu_{m+1}) + \ab_m p \cdot \nabla \A_{3^m+2} v(\cu_{m+1}) \Rr) \, \d \mu \Rr].
\end{align*}
Using Jensen's inequality, we have 
\begin{multline*}
\Er \Ll[\int_{\cu_m} \Ll( \frac 1 2 \nabla \A_{3^m+2} v(\cu_{m+1}) \cdot \a \nabla \A_{3^m+2} v(\cu_{m+1}) \Rr) \, \d \mu \Rr] 
\\
\leq \Er \Ll[\int_{\cu_m} \Ll( \frac 1 2 \nabla  v(\cu_{m+1}) \cdot \a \nabla v(\cu_{m+1}) \Rr) \, \d \mu \Rr],
\end{multline*}
and the conditional expectation also implies that 
\begin{multline*}
\Er \Ll[\int_{\cu_m} \Ll( - p \cdot \a \nabla \A_{3^m+2} v(\cu_{m+1}) + \ab_m p \cdot \nabla \A_{3^m+2} v(\cu_{m+1}) \Rr) \, \d \mu \Rr] \\
= \Er \Ll[\int_{\cu_m} \Ll( - p \cdot \a \nabla v(\cu_{m+1}) + \ab_m p \cdot \nabla  v(\cu_{m+1}) \Rr) \, \d \mu \Rr].
\end{multline*}
Thus we combine these terms with the quadratic response \cref{eq:JQuadra} to obtain 
\begin{align*}
\vert \text{\cref{eq:JamDecom}-a} \vert^2 &\leq J(\cu_m, p, \ab_m p) - \Er \Ll[\frac{1}{\rho \vert \cu_m \vert}\int_{\cu_m} \Ll( -\frac 1 2 \nabla  v(\cu_{m+1}) \cdot \a \nabla  v(\cu_{m+1}) \Rr) \, \d \mu \Rr] \\
& \qquad - \Er \Ll[\frac{1}{\rho \vert \cu_m \vert}\int_{\cu_m} \Ll( - p \cdot \a \nabla  v(\cu_{m+1}) + \ab_m p \cdot \nabla  v(\cu_{m+1}) \Rr) \, \d \mu \Rr]\\
& = \Er \Ll[\frac{1}{\rho \vert \cu_m \vert}\int_{\cu_m} \Ll( \frac 1 2 \nabla  (v(\cu_{m+1}) - v(\cu_{m})) \cdot \a \nabla  (v(\cu_{m+1}) - v(\cu_{m})) \Rr) \, \d \mu \Rr],
\end{align*}
and we use \Cref{lem:JTwoScale} between $\cu_m$ and $\cu_{m+1}$ to get
\begin{align}\label{eq:JamDecomBoundA}
\vert \text{\cref{eq:JamDecom}-a} \vert^2 &\leq 3^d(J(\cu_m, p, \ab_m p) - J(\cu_{m+1}, p, \ab_m p))  \leq C(d,\Lambda) \tau_{m},
\end{align}
where the quantity $\tau_m$ is defined in \cref{eq:defErrorJ}.

For the term \cref{eq:JamDecom}-b, we can apply the modified Caccioppoli inequality \cref{eq:Caccioppoli}: there exist two finite positive constants $C(d,\Lambda)$ and $\theta(d, \Lambda) \in (0,1)$ such that  
\begin{multline}\label{eq:JamDecomB}
\Er\Ll[\frac{1}{\rho \vert \cu_m \vert}\int_{\cu_m} \nabla (\A_{3^m+2} v(\cu_{m+1})) \cdot \a \nabla (\A_{3^m+2} v(\cu_{m+1})) \, \d \mu\Rr] \\
\leq  \frac{C}{3^{2m} \rho \vert \cu_{m+1} \vert} \Er[(v(\cu_{m+1}))^2] + \theta \Er\Ll[\frac{1}{\rho \vert \cu_{m+1} \vert}\int_{\cu_{m+1}} \nabla v(\cu_{m+1}) \cdot \a \nabla v(\cu_{m+1}) \, \d \mu\Rr].
\end{multline}
Using \cref{eq:JEnergy}, we see that the averaged gradient term on the right side of \cref{eq:JamDecomB} is $J(\cu_{m+1}, p, \ab_m p)$, and \cref{eq:JSubadditive} asserts that ${J(\cu_{m+1}, p, \ab_m p) \leq J(\cu_{m}, p, \ab_m p)}$. Therefore, we get the bound for \cref{eq:JamDecom}-b
\begin{align}\label{eq:JamDecomBoundB}
\vert \text{\cref{eq:JamDecom}-b} \vert^2 \leq \frac{C}{3^{2m} \rho \vert \cu_{m+1} \vert} \Er[(v(\cu_{m+1}))^2] + \theta J(\cu_{m}, p, \ab_m p).
\end{align}
We put \cref{eq:JamDecomBoundA} and \cref{eq:JamDecomBoundB} back to \cref{eq:JamDecom}, obtaining
\begin{align*}
(J(\cu_{m}, p, \ab_m p))^{\frac{1}{2}} & \leq C \tau_m^{\frac{1}{2}} + \Ll( \frac{C}{3^{2m} \rho \vert \cu_{m+1} \vert} \norm{v(\cu_{m+1})}^2_{\cL^2} + \theta J(\cu_{m}, p, \ab_m p)\Rr)^{\frac{1}{2}}\\
&\leq C \tau_m^{\frac{1}{2}} + \frac{C}{3^{m} (\rho \vert \cu_{m+1} \vert)^{\frac{1}{2}}}\norm{v(\cu_{m+1})}_{\cL^2} + \theta^{\frac{1}{2}}(J(\cu_{m}, p, \ab_m p))^{\frac{1}{2}}.
\end{align*}
Since $\theta < 1$, this gives
\begin{align}\label{eq:JamL2}
J(\cu_{m}, p, \ab_m p) \leq C\Ll(\tau_m + \frac{1}{3^{2m} \rho \vert \cu_{m+1} \vert}\norm{v(\mu, \cu_{m+1}, p, \ab_m p)}^2_{\cL^2}\Rr).
\end{align}

\subsection{Step 2: flatness estimate}
In this step, we estimate the $\cL^2$-flatness of optimizers of $J$. Notice that, using the result of Lemma~\ref{eq:mainFlatness} with $v(\cdot, \cu_{m+1}, p, \ab_m p)$, the corresponding affine function is $0$ and we obtain from \cref{eq:JamL2} that 
\begin{align}\label{eq:JamTau}
J(\cu_{m}, p, \ab_m p) \leq C \Ll(3^{-\beta m} +  \sum_{n=0}^m3^{-\beta(m-n)}\tau_n\Rr).
\end{align}

\begin{lemma}[$\cL^2$-flatness estimate]
There exist $\beta(d) > 0$ and $C(d,\Lambda, \rho) < \infty$  such that for every $p,q \in B_1$ and $m \in \mathbb{N}$,
\begin{align}\label{eq:mainFlatness}
\frac{1}{\rho \vert \cu_{m+1}\vert} \norm{ v(\cdot, \cu_{m+1}, p, q) - \ell_{\bar \a_m^{-1}q - p, \cu_{m+1}}}^2_{\cL^2} \leq C 3^{2m} \Ll(3^{-\beta m} +  \sum_{n=0}^m 3^{-\beta(m-n)}\tau_n\Rr).
\end{align}
\end{lemma}
\begin{proof}
In the rest of the proof, we write $v(U) := v(\cdot, U, p, q)$ as we will not change $p,q$ in the proof. Since $\Er\Ll[v(\cu_{m+1}) - \ell_{\bar \a_m^{-1}q - p, \cu_{m+1}} \vert \G_{m+1}\Rr] = 0$, we can use the multiscale Poincar\'e inequality \cref{eq:Multi} 
\begin{equation}\label{eq:mainFlatnessDecom}
\begin{split}
& \frac{1}{(\rho \vert \cu_{m+1}\vert)^{\frac{1}{2}}} \norm{ v(\cu_{m+1}) - \ell_{\bar \a_m^{-1}q - p, \cu_{m+1}}}_{\cL^2} \\
& \leq  C \Ll(\Er\Ll[\frac{1}{\rho \vert\cu_{m+1}\vert} \int_{\cu_{m+1}} \vert\nabla v(\cu_{m+1}) - (\ab_m^{-1}q - p) \vert^2 \, \d \mu \Rr]\Rr)^{\frac{1}{2}} \\
&\qquad   + C \sum_{n=0}^{m+1} 3^n \Ll(\Er\Ll[\frac{1}{\rho \vert\cu_{m+1}\vert} \int_{\cu_{m+1}} \vert \S_{m+1,n} \nabla v(\cu_{m+1}) - (\ab_m^{-1}q - p)\vert^2 
 \, \d \mu \Rr]\Rr)^{\frac{1}{2}}.
\end{split}
\end{equation}
The first term on the right side above is of constant order, by \cref{eq:JEnergy}. For the second term, we use a two-scale comparison for every $0 \leq n \leq m+1$ that 
\begin{equation}\label{eq:mainFlatnessDecom2}
\begin{split}
&\Ll(\Er\Ll[\frac{1}{\rho \vert\cu_{m+1}\vert} \int_{\cu_{m+1}} \vert \S_{m+1,n} \nabla v(\cu_{m+1}) - (\ab_m^{-1}q - p)\vert^2 
 \, \d \mu \Rr]\Rr)^{\frac{1}{2}} \\
& \leq \vert \ab_m^{-1} - \ab_n^{-1} \vert + 
\Ll(\Er\Ll[\frac{1}{\rho \vert\cu_{m+1}\vert} \sum_{z \in \Z_{m+1,n}}\int_{z+\cu_n} \vert \S_{m+1,n} \nabla v(\cu_{m+1}) - \S_{m+1,n} \nabla v(z+\cu_n)\vert^2 
 \, \d \mu \Rr]\Rr)^{\frac{1}{2}} \\
& \qquad + \Ll(\Er\Ll[\frac{1}{\rho \vert\cu_{m+1}\vert} \sum_{z \in \Z_{m+1,n}}\int_{z+\cu_n} \vert \S_{m+1,n} \nabla v(z+\cu_n) - (\ab_n^{-1}q - p)\vert^2 
 \, \d \mu \Rr]\Rr)^{\frac{1}{2}}.
\end{split}
\end{equation}
For the first term $\vert \ab_m^{-1} - \ab_n^{-1} \vert$ we have 
\begin{align*}
\vert \ab_m^{-1} - \ab_n^{-1} \vert^2 \leq C(d,\Lambda)\vert \ab_m^{-1} - \ab_n^{-1} \vert \leq \sum_{k=n}^{m-1} \tau_k.
\end{align*}
For the second term in \cref{eq:mainFlatnessDecom2}, recalling \cref{eq:defSnk}, we use Jensen's inequality 
and \cref{eq:JTwoScale} to get 
\begin{align*}
&\Er\Ll[\frac{1}{\rho \vert\cu_{m+1}\vert} \sum_{z \in \Z_{m+1,n}}\int_{z+\cu_n} \vert \S_{m+1,n} \nabla v(\cu_{m+1}) - \S_{m+1,n} \nabla v(z+\cu_n)\vert^2 
 \, \d \mu \Rr] \\
& \quad \leq  \Er\Ll[\frac{1}{\rho \vert\cu_{m+1}\vert} \sum_{z \in \Z_{m+1,n}}\int_{z+\cu_n} \vert \nabla v(\cu_{m+1}) - \nabla v(z+\cu_n)\vert^2 
 \, \d \mu \Rr] \\
& \quad \leq  \sum_{k=n}^m \tau_k.
\end{align*}
For the third term \cref{eq:mainFlatnessDecom2}, we use \cref{eq:SnkMt}, Jensen's inequality,
and stationarity. Here we remark that the operator $\S_{n,n}^z$ is a conditional expectation with more information than $\S_{m+1, n}$.
\begin{align*}
&\Er\Ll[\frac{1}{\rho \vert\cu_{m+1}\vert} \sum_{z \in \Z_{m+1,n}}\int_{z+\cu_n} \vert \S_{m+1,n} \nabla v(z+\cu_n) - (\ab_n^{-1}q - p)\vert^2 
 \, \d \mu \Rr] \\
& \quad \leq \Er\Ll[\frac{1}{\rho \vert\cu_{m+1}\vert} \sum_{z \in \Z_{m+1,n}}\int_{z+\cu_n} \vert \S_{n,n}^z \nabla v(z+\cu_n) - (\ab_n^{-1}q - p)\vert^2 
 \, \d \mu \Rr] \\
& \quad =  \Er\Ll[\frac{1}{\rho \vert\cu_{n}\vert} \int_{\cu_{n}} \vert \S_{n} \nabla v(\cu_{n}) - (\ab_n^{-1}q - p)\vert^2 
 \, \d \mu \Rr].
\end{align*}
The estimation of this term is postponed to the next step. We will prove in \Cref{l:mainVariance} below that
\begin{align*}
\Er\Ll[\frac{1}{\rho \vert\cu_{n}\vert} \int_{\cu_{n}} \vert \S_{n} \nabla v(\cu_{n}) - (\ab_n^{-1}q - p)\vert^2 
 \, \d \mu \Rr] \leq  C 3^{-\beta n} + \sum_{k=0}^{n-1}3^{-\beta(n-k)}\tau_k. 
\end{align*}
We put these estimates back to \cref{eq:mainFlatnessDecom} and obtain that 
\begin{align*}
\frac{1}{(\rho \vert \cu_{m+1}\vert)^{\frac{1}{2}}} \norm{ v(\cu_{m+1}) - \ell_{\ab_m^{-1}q - p, \cu_{m+1}}}_{\cL^2} \leq C \sum_{n=0}^{m} 3^n \Ll(3^{-\beta n} + \sum_{k=0}^{n-1}3^{-\beta(n-k)}\tau_k + \sum_{k=n}^m \tau_k\Rr)^\frac{1}{2}.
\end{align*}
We square the two sides and use the Cauchy-Schwarz inequality to obtain
\begin{align*}
& \frac{1}{\rho \vert \cu_{m+1}\vert}\norm{ v(\cu_{m+1}) - \ell_{\ab_m^{-1}q - p, \cu_{m+1}}}_{\cL^2}^2 
\\
& \qquad \leq C \Ll(\sum_{n=0}^{m} 3^n\Rr) \Ll(\sum_{n=0}^{m} 3^n \Ll(3^{-\beta n} + \sum_{k=0}^{n-1}3^{-\beta(n-k)}\tau_k + \sum_{k=n}^m \tau_k\Rr)\Rr)
\\
& \qquad \leq  C 3^{2m} \Ll(3^{-\beta m} +  \sum_{n=0}^m 3^{-\beta(m-n)}\tau_n\Rr),
\end{align*}
as announced.
\end{proof}

\subsection{Step 3: variance estimate} In this part, we prove the following variance estimate, which was used in Step 2.
\begin{lemma}[Variance estimate]
\label{l:mainVariance}
There exist $\beta(d) > 0$ and $C(d,\Lambda, \rho) < \infty$ such that for every $p,q \in B_1$ and $n \in \mathbb{N}$,
\begin{align}
\label{eq:mainVariance}
\Er\Ll[\frac{1}{\rho \vert \cu_n \vert} \int_{\cu_n} \vert\S_{n} \nabla v(\mu, \cu_n, p, q) - (\ab_n^{-1}q - p)\vert^2 \, \d \mu\Rr] \leq C 3^{-\beta n} + \sum_{k=0}^{n-1}3^{-\beta(n-k)}\tau_k.
\end{align}
\end{lemma}
\begin{proof}
In the rest of the proof, we write $v(U) := v(\cdot, U, p, q)$, as we will not change $p,q$ in the proof. From \cref{eq:SlopABN}, we know that the average slope of $v(\cu_n)$ is $(\ab_n^{-1}q - p)$, and notice that $v(\cu_n)$ is $\mcl F_{B_1(\cu_n)}$-measurable. Thus the idea is to use $\{v(z+\cu_k)\}_{z \in \Z_{n,k}}$ to approximate  \cref{eq:mainVariance} in scale $3^k$ with some error, and then apply the independence for $v(z+\cu_k)$ and $v(z'+\cu_k)$ for $\dist(z, z')$ large. However, different from the standard elliptic setting, here we will see a renormalization with random weights.

We start by relaxing \cref{eq:mainVariance} to $\G_{n,n-2}$. We observe that in fact $\S_{n}\nabla v(\cu_n)$ is constant in $\cu_n$, so
\begin{align*}
\int_{\cu_n} \vert\S_{n} \nabla v(\cu_n) - (\ab_n^{-1}q - p)\vert^2 \, \d \mu = \frac{1}{\mu(\cu_n)} \Ll\vert\int_{\cu_n}\Ll( \S_{n} \nabla v(\cu_n) - (\ab_n^{-1}q - p) \Rr)\, \d \mu \Rr\vert^2.
\end{align*}
We denote by $V_n$ the left hand side of \cref{eq:mainVariance}. By triangle inequality, we have 
\begin{align}\label{eq:mainVarianceDecom}
(V_n)^{\frac{1}{2}} \leq \text{\cref{eq:mainVarianceDecom}-a} + \text{\cref{eq:mainVarianceDecom}-b} + \text{\cref{eq:mainVarianceDecom}-c},
\end{align}
with
\begin{align*}
\text{\cref{eq:mainVarianceDecom}-a} &= \vert\ab_n^{-1} - \ab_{n-2}^{-1}\vert, \\
\text{\cref{eq:mainVarianceDecom}-b} &=  \Ll(\Er\Ll[\frac{1}{\rho \vert \cu_n \vert}\frac{1}{\mu(\cu_{n})} \Ll\vert\sum_{z \in \Z_{n,n-2}}\int_{z+\cu_{n-2}} \Ll(\S_{n,n} \nabla v(\cu_n) - \S_{n,n-2} \nabla v(z+\cu_{n-2}) \Rr)\, \d \mu \Rr\vert^2 \Rr]\Rr)^{\frac{1}{2}},\\
\text{\cref{eq:mainVarianceDecom}-c} &=  \Ll(\Er\Ll[\frac{1}{\rho \vert \cu_n \vert}\frac{1}{\mu(\cu_{n})} \Ll\vert\sum_{z \in \Z_{n,n-2}}\int_{z+\cu_{n-2}} \Ll(\S_{n,n-2} \nabla v(z+\cu_{n-2}) - (\ab_{n-2}^{-1}q - p) \Rr) \, \d \mu \Rr\vert^2 \Rr]\Rr)^{\frac{1}{2}}.
\end{align*}

The term \cref{eq:mainVarianceDecom}-a can be controlled by \cref{eq:ErrorABN}:
\begin{align}\label{eq:mainVarianceDecomBoundA}
\text{\cref{eq:mainVarianceDecom}-a} \leq C(\tau_{n-2} + \tau_{n-1})^{\frac{1}{2}}.
\end{align}

For the term \cref{eq:mainVarianceDecom}-b, recalling \cref{eq:SnkMt} and \cref{eq:defSnk}, we use Jensen's inequality 
and the two-scale comparison  \cref{eq:JTwoScale} to get 
\begin{equation}\label{eq:mainVarianceDecomBoundB}
\begin{split}
\text{\cref{eq:mainVarianceDecom}-b} &\leq  \Ll(\Er\Ll[\frac{1}{\rho \vert \cu_n \vert} \sum_{z \in \Z_{n,n-2}}\int_{z+\cu_{n-2}} \vert\S_{n,n} \nabla v(\cu_n) - \S_{n,n-2} \nabla v(z+\cu_{n-2})\vert^2 \, \d \mu \Rr]\Rr)^{\frac{1}{2}}\\
&\leq   \Ll(\Er\Ll[\frac{1}{\rho \vert \cu_n \vert} \sum_{z \in \Z_{n,n-2}}\int_{z+\cu_{n-2}} \vert \nabla v(\cu_n) -  \nabla v(z+\cu_{n-2})\vert^2 \, \d \mu \Rr]\Rr)^{\frac{1}{2}}\\
&\leq (\tau_{n-2} + \tau_{n-1})^{\frac{1}{2}}.
\end{split}
\end{equation}

The term \cref{eq:mainVarianceDecom}-c is the key for our result. To simplify a little more the notation, we write
\begin{equation}\label{eq:defXm}
\Ll\{
\begin{array}{ll}
X_{z} :=  \S_{n,n-2} \nabla v(z+\cu_{n-2})(\mu, z) - (\ab_{n-2}^{-1}q - p) ,\\
m_z := \mu(z+\cu_{n-2}).
\end{array}
\Rr.
\end{equation} 
Notice that $X_z, m_z$ are $\mcl F_{z+\cu_{n-1}}$-measurable. With this notation in place, we have
\begin{align*}
\int_{z+\cu_{n-2}} \Ll(\S_{n,n-2} \nabla v(z+\cu_{n-2}) - (\ab_{n-2}^{-1}q - p) \Rr)\, \d \mu &= m_z X_z,
\end{align*}
and by \cref{eq:SlopABN}, 
\begin{equation*}  
\Er[m_z X_z] = 0.
\end{equation*}
The term \cref{eq:mainVarianceDecom}-c we want to estimate can be rewritten as
\begin{align*}
\text{\cref{eq:mainVarianceDecom}-c} &=  \Ll(\Er\Ll[\frac{1}{\rho \vert \cu_n \vert}\frac{\Ll(\sum_{z \in \Z_{n,n-2}} m_z X_z \Rr)^2}{\sum_{z \in \Z_{n,n-2} }m_z} \Rr]\Rr)^{\frac{1}{2}}. 
\end{align*}
If the coefficients $m_z$ were deterministic, then we would be able to leverage on the finite range of dependence of $X_z$ in this variance term.
However, since the number of particles $m_z$ is random, we introduce the event
\begin{align}\label{eq:defGoodConfiguration}
\mcl C_{n,\rho,\delta} := \Ll\{\mu \in \mmd(\Rd): \forall z \in \Z_{n,n-2}, \Ll\vert\frac{\mu(z+\cu_{n-2})}{\rho\vert \cu_{n-2} \vert} - 1\Rr\vert \leq \delta, \text{ and } \Ll\vert\frac{\mu(\cu_{n})}{\rho\vert \cu_{n} \vert} - 1\Rr\vert \leq \delta\Rr\},
\end{align}
thus we can divide \cref{eq:mainVarianceDecom}-c into two terms 
\begin{align*}
\text{\cref{eq:mainVarianceDecom}-c} &\leq \text{\cref{eq:mainVarianceDecom}-c1} + \text{\cref{eq:mainVarianceDecom}-c2},\\
\text{\cref{eq:mainVarianceDecom}-c1} &=  \Ll(\Er\Ll[\frac{\Ind{(\mcl C_{n,\rho,\delta})^c}}{\rho \vert \cu_n \vert}\frac{\Ll(\sum_{z \in \Z_{n,n-2}} m_z X_z \Rr)^2}{\sum_{z \in \Z_{n,n-2} }m_z}  \Rr]\Rr)^{\frac{1}{2}}, \\
\text{\cref{eq:mainVarianceDecom}-c2} &= \Ll(\Er\Ll[\frac{\Ind{\mcl C_{n,\rho,\delta}}}{\rho \vert \cu_n \vert}\frac{\Ll(\sum_{z \in \Z_{n,n-2}} m_z X_z \Rr)^2}{\sum_{z \in \Z_{n,n-2} }m_z}  \Rr]\Rr)^{\frac{1}{2}}.
\end{align*} 
For the term \cref{eq:mainVarianceDecom}-c1, we know that $(\mcl C_{n,\rho,\delta})^c$ is not typical in large scales, and we have the Chernoff bound 
\begin{align*}
\Pr[\mu \notin \mcl C_{n,\rho,\delta}] \leq 3^{2d+1} \exp\Ll(-\frac{\rho \vert\cu_{n-2}\vert \delta^2}{4}\Rr).
\end{align*}
Moreover, by the Cauchy-Schwarz inequality,
\begin{align*}
\frac{\Ll(\sum_{z \in \Z_{n,n-2}} m_z X_z \Rr)^2}{\sum_{z \in \Z_{n,n-2} }m_z} \leq \sum_{z \in \Z_{n,n-2}} m_z \vert X_z \vert^2.
\end{align*}
We need a bound for the term $\vert X_z \vert^2$: recalling the definition in  \cref{eq:defSnk} and \cref{eq:JSlope}, 
\begin{align*}
\S_{n-2,n-2}^z \nabla v(z+\cu_{n-2})(\mu, z) &= \Er\Ll[\fint_{z+\cu_{n-2}} \nabla v(z+\cu_{n-2}) \, \d\mu \ \Bigg\vert \ \G_{n-2,n-2}^z  \Rr]\\
&= \a(z+\cu_{n-2}; \G_{z+\cu_{n-2}})^{-1}q - p.
\end{align*}
Using the martingale structure of \cref{eq:SnkMt}, we have
\begin{align*}
X_z &= \S_{n,n-2} \nabla v(z+\cu_{n-2})(\mu, z) - (\ab_{n-2}^{-1}q - p) \\
&= \Er\Ll[\fint_{z+\cu_{n-2}} \S_{n-2,n-2}^z \nabla v(z+\cu_{n-2}) \, \d\mu \ \Bigg\vert \ \G_{n,n-2}  \Rr] - (\ab_{n-2}^{-1}q - p)\\
&= \Er\Ll[ \a(z+\cu_{n-2}; \G_{z+\cu_{n-2}})^{-1} - \ab_{n-2}^{-1} \ \Bigg\vert \ \G_{n,n-2}  \Rr]q.
\end{align*} 
Then we use Jensen's inequality and the bound of ${\id \leq \a(z+\cu_{n-2}; \G_{z+\cu_{n-2}}) \leq \Lambda \id}$ 
\begin{equation}\label{eq:XmTrivialBound}
\begin{split}
\vert X_z \vert^2 &= \vert\S_{n,n-2} \nabla v(z+\cu_{n-2}) - (\ab_{n-2}^{-1}q - p) \vert^2 \\
&= \Er\Ll[ \vert\a(z+\cu_{n-2}; \G_{z+\cu_{n-2}}) - \ab_{n-2}^{-1}\vert^2 \, \vert \, \G_{n,n-2}\Rr] \leq \Lambda^2.
\end{split}
\end{equation}
This concludes that 
\begin{equation}\label{eq:mainVarianceDecomBoundC1}
\begin{split}
\text{\cref{eq:mainVarianceDecom}-c1} \leq \Lambda^2\Er\Ll[\frac{\Ind{(\mcl C_{n,\rho,\delta})^c}}{\rho \vert \cu_n \vert}\mu(\cu_n)\Rr] &\leq C(d,\Lambda)\frac{1}{\rho \vert\cu_{n-2}\vert}\exp\Ll(-\frac{\rho \vert\cu_{n-2}\vert \delta^2}{4}\Rr)\\
&\leq C(d, \Lambda, \rho) 3^{-dn}.
\end{split}
\end{equation}

Finally, we treat \cref{eq:mainVarianceDecom}-c2. We calculate \cref{eq:mainVarianceDecom}-c2 at first with the conditional expectation with respect to $\G_{n,n-2}$. Clearly, $\mcl C_{n,\rho,\delta}$ is $\G_{n,n-2}$-measurable, and under this condition ${\mu(\cu_n) \geq (1-\delta)\rho \vert \cu_n \vert}$, so we have 
\begin{equation}\label{eq:mainVarianceDecomC2}
\begin{split}
\vert\text{\cref{eq:mainVarianceDecom}-c2}\vert^2 &= \frac{1}{\rho \vert \cu_n \vert}\Er\Ll[\frac{\Ind{\mcl C_{n,\rho,\delta}}}{\mu(\cu_{n})} \Er\Ll[\Ll(\sum_{z \in \Z_{n,n-2}} m_z X_z \Rr)^2 \ \Bigg \vert \ \G_{n,n-2} \Rr]\Rr] \\
&\leq \frac{1}{\rho \vert \cu_n \vert}\Er\Ll[\frac{1}{(1-\delta)\rho \vert \cu_n \vert} \Er\Ll[\Ind{\mcl C_{n,\rho,\delta}} \Ll(\sum_{z \in \Z_{n,n-2}} m_z X_z \Rr)^2 \ \Bigg \vert \ \G_{n,n-2} \Rr]\Rr].
\end{split}
\end{equation}
We would like to develop the term $\vert\sum_{z \in \Z_{n,n-2}} m_z X_z \vert^2$ and also drop out the indicator term. The argument here is deterministic
\begin{align*}
\Ll\vert\sum_{z \in \Z_{n,n-2}} m_z X_z \Rr\vert^2 &= \sum_{\substack{z,z' \in \Z_{n,n-2}\\
\vert z - z'\vert_{\infty} < 3^{n-1}}}  m_z m_{z'} X_z \cdot X_{z'} +  \sum_{\substack{z,z' \in \Z_{n,n-2}\\
\vert z - z'\vert_{\infty} \geq 3^{n-1}}}   m_z m_{z'} X_z \cdot X_{z'}\\
&\leq \frac{1}{2}\sum_{\substack{z,z' \in \Z_{n,n-2}\\
\vert z - z'\vert_{\infty} < 3^{n-1}}} \Ll( (m_z)^2\vert X_z\vert^2 + (m_{z'})^2\vert X_{z'}\vert^2 \Rr) +  \sum_{\substack{z,z' \in \Z_{n,n-2}\\
\vert z - z'\vert_{\infty} \geq 3^{n-1}}}  m_z m_{z'} X_z \cdot X_{z'},
\end{align*} 
where $\vert z - z'\vert_{\infty} := \max_{1 \leq i \leq d} \vert z_i - z'_i\vert$. We now add back the indicator $\Ind{\mcl C_{n,\rho,\delta}}$ and develop it
\begin{equation}\label{eq:mainVarianceTrick}
\begin{split}
&\Ind{\mcl C_{n,\rho,\delta}}\Ll\vert\sum_{z \in \Z_{n,n-2}} m_z X_z \Rr\vert^2 \\
&\leq \Ind{\mcl C_{n,\rho,\delta}}\Ll( \frac{(1+\delta)\rho \vert \cu_{n-2} \vert}{2}\sum_{\substack{z,z' \in \Z_{n,n-2}\\
\vert z - z'\vert < 3^{n-1}}} \Ll(m_z \vert X_z\vert^2 + m_{z'}\vert X_{z'}\vert^2\Rr) + \sum_{\substack{z,z' \in \Z_{n,n-2}\\
\vert z - z'\vert \geq 3^{n-1}}}  m_z m_{z'} X_z \cdot X_{z'}\Rr) \\
&\leq \frac{(1+\delta)\rho \vert \cu_{n-2} \vert}{2}\sum_{\substack{z,z' \in \Z_{n,n-2}\\
\vert z - z'\vert < 3^{n-1}}} \Ll(m_z \vert X_z\vert^2 + m_{z'}\vert X_{z'}\vert^2\Rr) + \sum_{\substack{z,z' \in \Z_{n,n-2}\\
\vert z - z'\vert \geq 3^{n-1}}}  m_z m_{z'} X_z \cdot X_{z'}.    
\end{split}
\end{equation} 
From the first line to the second line above, we use that $m_z \leq (1+\delta)\rho \vert \cu_{n-2} \vert$ under the event $\mcl C_{n,\rho,\delta}$. We then keep in mind that the quantity in $\Big( \cdots\Big)$ on the second line of \cref{eq:mainVarianceTrick} is always larger than $\Ll\vert\sum_{z \in \Z_{n,n-2}} m_z X_z \Rr\vert^2$, so it is nonnegative. Therefore, from the second line to the third line, we can drop the indicator function in front. Inserting this estimate into \cref{eq:mainVarianceDecomC2}, we obtain that 
\begin{align*}
\vert\text{\cref{eq:mainVarianceDecom}-c2}\vert^2 \leq \frac{1}{\rho \vert \cu_n \vert} \frac{(1+\delta) \vert \cu_{n-2} \vert}{(1-\delta) \vert \cu_n \vert} \sum_{\substack{z,z' \in \Z_{n,n-2}\\
\vert z - z'\vert_{\infty} < 3^{n-1}}}\Er\Ll[  \frac{1}{2}\Ll(m_z \vert X_z\vert^2 + m_{z'}\vert X_{z'}\vert^2\Rr)\Rr] \\
+ \frac{1}{\rho \vert \cu_n \vert} \frac{1}{(1-\delta) \vert \cu_n \vert} \sum_{\substack{z,z' \in \Z_{n,n-2}\\
\vert z - z'\vert_{\infty} \geq 3^{n-1}}} \Er\Ll[ m_z m_{z'} X_z \cdot X_{z'}\Rr].
\end{align*}
The sum in the second line is $0$, because for $\vert z - z'\vert_{\infty} \geq 3^{n-1}$, $m_z X_z$ and $m_{z'}X_{z'}$ are independent, 
\begin{align*}
\Er\Ll[ m_z m_{z'} X_z \cdot X_{z'}\Rr] = \Er\Ll[ m_z X_z \Rr] \cdot \Er\Ll[ m_{z'} X_{z'}\Rr] = 0.
\end{align*}
For the sum in the first line, $\Er[m_z \vert X_z\vert^2]$ is nothing but
\begin{align*}
\Er\Ll[{\int_{z+\cu_{n-2}} \vert \S_{n,n-2} \nabla v(z+\cu_{n-2}) - (\ab_{n-2}^{-1}q - p) \vert^2 \, d\mu}\Rr].
\end{align*} 
We use Jensen's inequality 
to shrink the operator to $\S_{n-2,n-2}^z$ that 
\begin{align*}
&\Er\Ll[{\int_{z+\cu_{n-2}} \vert \S_{n,n-2} \nabla v(z+\cu_{n-2}) - (\ab_{n-2}^{-1}q - p) \vert^2 \, d\mu}\Rr]\\ 
& \quad \leq \Er\Ll[{\int_{z+\cu_{n-2}} \vert \S_{n-2,n-2}^z \nabla v(z+\cu_{n-2}) - (\ab_{n-2}^{-1}q - p) \vert^2 \, d\mu}\Rr] \\
& \quad = \Er\Ll[{\int_{\cu_{n-2}} \vert \S_{n-2} \nabla v(\cu_{n-2}) - (\ab_{n-2}^{-1}q - p) \vert^2 \, d\mu}\Rr].
\end{align*}
There are at most $9^d \times 5^d$ pairs $z,z' \in \Z_{n,n-2}$ such that $\vert z - z'\vert_{\infty} < 3^{n-1}$; see \Cref{fig:variance} for an illustration. Therefore, we obtain 
\begin{align*}
\vert\text{\cref{eq:mainVarianceDecom}-c2}\vert^2 &\leq \Ll(\frac{5}{9}\Rr)^d \Ll(\frac{1+\delta}{1-\delta}\Rr) \Er\Ll[\frac{1}{\rho \vert \cu_{n-2}\vert}\int_{\cu_{n-2}} \vert \S_{n-2} \nabla v(\cu_{n-2}) - (\ab_{n-2}^{-1}q - p) \vert^2 \, d\mu\Rr]\\
&=  \Ll(\frac{5}{9}\Rr)^d \Ll(\frac{1+\delta}{1-\delta}\Rr) V_{n-2},
\end{align*}
where we recall that $V_n$ is the left hand side of \cref{eq:mainVariance}. We put this estimate together with \cref{eq:mainVarianceDecomBoundA}, \eqref{eq:mainVarianceDecomBoundB}, \eqref{eq:mainVarianceDecomBoundC1} back to \cref{eq:mainVarianceDecom} to obtain the recurrence relation
\begin{align*}
(V_{n})^{\frac{1}{2}} \leq \Ll(\frac{5}{9}\Rr)^\frac{d}{2} \Ll(\frac{1+\delta}{1-\delta}\Rr)^{\frac{1}{2}} (V_{n-2})^{\frac{1}{2}} + C(\tau_{n-2} + \tau_{n-1})^{\frac{1}{2}} + C3^{-dn}.
\end{align*}
By choosing $\delta(d)> 0$ sufficiently small, we obtain the desired result \cref{eq:mainVariance}.
\end{proof}
\begin{figure}[h!]
\centering
\includegraphics[scale=0.5]{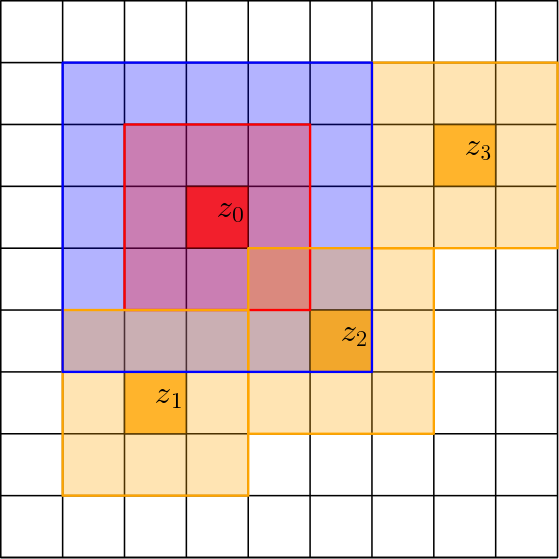}
\caption{In the cube $\square_n$ and all its sub-cubes $\{z+\square_{n-2}\}_{z \in \Z_{n,n-2}}$, for a chosen sub-cube $z_0 + \square_{n-2}$ (the cube in dark red), the support of ${v(z_0 + \square_{n-2})}$ is in $z_0 + \square_{n-1}$ (the cube in light red), so it has at most $5^d$ cubes of scale $3^{n-2}$ whose associated function has a support intersecting with $z_1 + \square_{n-1}$ (the cube in blue). For example, $v(z_2+\square_{n-2})$ has correlation with $v(z_0+\square_{n-2})$, while ${v(z_1+\square_{n-2})}, {v(z_3+\square_{n-2})}$ do not. This gives us the contraction factor $\Ll(\frac{5}{9}\Rr)^{d}$.}\label{fig:variance}
\end{figure}

\subsection{Step 4: iterations} Once we obtain the estimate \cref{eq:JamTau}, it remains to do some numerical iterations, similarly to  \cite[Page 59-60]{AKMbook}. For the reader's convenience, we recall the main steps here. Let $\{\e_i\}_{1 \leq i \leq d}$ denote the canonical basis in $\Rd$, and define 
\begin{align*}
F_m := \sum_{i=1}^d J(\cu_m, \e_i, \ab_m \e_i).
\end{align*}
In order to obtain an exponential decay for $(F_m)_{m \geq 0}$, we first introduce a weighted version of this quantity:
\begin{align*}
\tilde{F}_m := \sum_{n=0}^{m}3^{-\frac{\beta}{2}(m-n)}F_n.
\end{align*}  
Here the exponent $\beta$ is the same as in \cref{eq:JamTau}. It is clear that $F_m \leq \tilde{F}_m$, so it suffices to prove an exponential decay for $(\tilde{F}_m)_{m \geq 0}$. We will do so by proving a recurrence equation of type $ \tilde{F}_{m+1} \le C(\tilde{F}_m - \tilde{F}_{m+1})$ for some constant $C(d,\Lambda) < \infty$. Thus in the following, we calculate some bounds for $(\tilde{F}_m - \tilde{F}_{m+1})$ and $\tilde{F}_{m+1}$.

Starting with $(\tilde{F}_m - \tilde{F}_{m+1})$, we write
\begin{align*}
\tilde{F}_m - \tilde{F}_{m+1} \geq \sum_{n=0}^n 3^{-\frac{\beta }{2}(m-n)} (F_n - F_{n+1}) - C3^{-\frac{\beta m}{2}}.
\end{align*} 
Noticing that $\ab_{n+1} p$ is the minimizer for the mapping ${q \mapsto J(\cu_{n+1}, p, q)}$ in \cref{eq:defJ}, we have  
\begin{align}\label{eq:Fbound}
F_{n+1} = \sum_{i=1}^d J(\cu_{n+1}, \e_i, \ab_{n+1} \e_i) \leq \sum_{i=1}^d J(\cu_{n+1}, \e_i, \ab_n \e_i).
\end{align} 
Using also \cref{eq:defJ}, that ${\id \le \ab_n \le \Lambda \id}$, and that $p \mapsto \nu(\cu_n,p) -\nu(\cu_{n+1},p)$ and $q \mapsto \nu^*(\cu_n,q) -\nu^*(\cu_{n+1},q)$ are positive semidefinite quadratic forms, we get
\begin{align*}
F_n - F_{n+1} 
& \geq \sum_{i=1}^d (J(\cu_n, \e_i, \ab_n \e_i) - J(\cu_{n+1}, \e_i, \ab_n \e_i)) \\
& = \sum_{i = 1}^d (\nu(\cu_n,\e_i) - \nu(\cu_{n+1},\e_i)) + \sum_{i = 1}^d (\nu^*(\cu_n,\ab_n \e_i) - \nu^*(\cu_{n+1},\ab_n \e_i))\\
& \ge C^{-1} \Ll( \sup_{p \in B_1} \Ll( \nu(\cu_n,p) - \nu(\cu_{n+1},p) \Rr)  + \sup_{q \in B_1} \Ll( \nu^*(\cu_n,q) - \nu^*(\cu_{n+1},q) \Rr) \Rr) \\
& \geq C^{-1} \tau_n,
\end{align*}
and thus 
\begin{equation}\label{eq:FLower}
\tilde{F}_m - \tilde{F}_{m+1} \geq C^{-1}\sum_{n=0}^m 3^{-\frac{\beta}{2}(m-n)} \tau_n - C3^{-\frac{\beta m}{2}}.
\end{equation}
For the upper bound of $\tilde{F}_{m+1}$, we use \cref{eq:Fbound} to see that $F_n \leq F_{n+1}$, so 
\begin{align*}
\tilde{F}_{m+1} &= 3^{-\frac{\beta}{2}(m+1)} F_0 + \sum_{n=0}^{m}3^{-\frac{\beta}{2}(m-n)}F_{n+1} \\
&\leq C3^{-\frac{\beta m}{2}} + \sum_{n=0}^{m}3^{-\frac{\beta}{2}(m-n)}F_{n}.
\end{align*}
Then we apply \cref{eq:JamTau} into the result above to get 
\begin{equation}\label{eq:FUpper}
\begin{split}
\tilde{F}_{m+1} &\leq  C3^{-\frac{\beta m}{2}} + \sum_{n=0}^{m}3^{-\frac{\beta}{2}(m-n)}\Ll(3^{-\beta n} +  \sum_{k=0}^n 3^{-\beta(n-k)}\tau_k\Rr) \\
&\leq C3^{-\frac{\beta m}{2}} + 3^{-\frac{\beta}{2}m}\sum_{k=0}^m \tau_k \sum_{n=k}^m 3^{\frac{\beta}{2}(2k - n)} \\
&\leq C3^{-\frac{\beta m}{2}} + C\sum_{k=0}^m 3^{-\frac{\beta}{2}(m-k)}\tau_k.  
\end{split}
\end{equation}
We combine \cref{eq:FLower} and \cref{eq:FUpper}, to obtain $C(\tilde{F}_m - \tilde{F}_{m+1} + \tilde{C}3^{-\frac{\beta m}{2}}) \geq \tilde{F}_{m+1}$, which implies 
\begin{align*}
\tilde{F}_{m+1} \leq \theta \tilde{F}_m + C 3^{-\frac{\beta m}{2}},
\end{align*}
for some $\theta(d,\Lambda) \in (0,1)$. We thus conclude for the exponential decay of $(\tilde{F}_m)_{m \geq 0}$, and thus also of $F_m$, since $F_m \le \tilde{F}_m$. By \cref{eq:JBoundError}, this completes the proof of \Cref{thm:main}.

%
%
%
%
%
%
%
%

\appendix
\section{Some elementary properties of the function spaces}\label{sec:Appendix}

\begin{lemma}[Canonical projection]\label{lem:AppendixProjection}
Let $f : \mmd(\Rd) \to \R$ be a function, and for every Borel set $U$, measure $\mu \in \mmd(\Rd)$, and $n \in \N$, let $f_n(\cdot, \mu \mres U^c)$ denote the (permutation-invariant) function
\begin{equation*}  
f_n(\cdot, \mu \mres U^c) :
\Ll\{
\begin{array}{rcl}  
U^n  & \to & \R \\
(x_1, \ldots, x_n) & \mapsto & f \Ll( \sum_{i = 1}^n \de_{x_i} + \mu \mres U^c \Rr). 
\end{array}
\Rr.
\end{equation*}

The following statements are equivalent. 

(1) The function $f$ is $\mcl F$-measurable. 

(2) For every $n \in \N$, the function $f_n$ is $\mcl B_U^{\otimes n} \otimes \mcl F_{U^c}$-measurable. 

\end{lemma}

\begin{proof}
We start from $(1) \Rightarrow (2)$. Because ${\mcl F = \mcl F_U \otimes \mcl F_{U^c}}$, it suffices to study the product function
\begin{align*}
{f = \Ind{\mu(V_1)=n_1}\Ind{\mu(V_2)=n_2}\Ind{\mu(U)=n}},    
\end{align*}
for some Borel sets ${V_1 \subset U}, {V_2 \subset U^c}$. In this case, we have 
\begin{align*}
\{f_n = 1\} &= \{\mu(V_1)=n_1\} \cap \{\mu(V_2)=n_2\} \cap\{\mu(U)=n\} \\
&= \bigcup_{\sigma \in S_n} \Ll(\bigcap_{i=1}^{n_1} \{x_{\sigma(i) } \in V_1\} \bigcap_{j={n_1 + 1}}^n \{x_{\sigma(j)} \in (U \backslash V_1)\} \bigcap  \{\mu(V_2)=n_2\} \Rr),
\end{align*}
where $S_n$ is the symmetric group. This proves that $f_n$ is $\mcl B_U^{\otimes n} \otimes \mcl F_{U^c}$-measurable.

\smallskip
We turn to $(2) \Rightarrow (1)$. Let us pick a suitable $f_n$ and $\mu \mres U = \sum_{i=1}^n \delta_{x_i}$, then the main point is to establish the $\mcl F$-measurable property. Since $f_n$ is $\mcl B_U^{\otimes n} \otimes \mcl F_{U^c}$-measurable and permutation-invariant, it suffices to study the function of type 
\begin{align}\label{eq:ProjectionBasis}
f_n = \sum_{\sigma \in S_n} \Ll(\prod_{i=1}^n \Ind{x_{\sigma(i)} \in V_i}\Rr) \Ind{\mu \mressmall U^c(V_0) = n_0}\Ind{\mu(U)=n},
\end{align}
for $\{V_i\}_{0 \leq i\leq n}$ Borel sets. This is still a complicated function, but we can add one more condition
\begin{align}\label{eq:ProjectionHypo}
\forall 1 \leq i,j \leq n, \quad V_i = V_j \text{ or } V_i \cap V_j = \emptyset.
\end{align}
For example, let $\{\tilde{V}_j\}_{0 \leq j\leq m}$ be all the different elements in $\{V_i\}_{0 \leq i\leq n}$, and $\tilde{V}_j$ appears $n_j$ times. For the functions of type \cref{eq:ProjectionBasis} satisfying the condition \cref{eq:ProjectionHypo}, the $\mcl F$-measurable property is easy to treat since we have 
\begin{multline*}
\sum_{\sigma \in S_n} \Ll(\prod_{i=1}^n \Ind{x_{\sigma(i)} \in V_i}\Rr) \Ind{\mu \mressmall U^c(V_0) = n_0}\Ind{\mu(U)=n} 
\\
= \Ll(\prod_{j=1}^m \Ind{\mu(\tilde{V}_j) = n_j}\Rr)\Ind{\mu \mressmall U^c(V_0) = n_0}\Ind{\mu(U)=n},
\end{multline*}
which is an $\mcl F$-measurable function.

Finally, let us conclude that for a general $f_n$ in \cref{eq:ProjectionBasis}, they can be decomposed into the sum of the functions with the propriety \cref{eq:ProjectionHypo}. Let us see the case $n = 2$, where we have the following decomposition
\begin{align*}
\Ind{x_1 \in V_1}\Ind{x_2 \in V_2} &= (\Ind{x_1 \in (V_1 \backslash V_2)} + \Ind{x_1 \in (V_1 \cap V_2)})(\Ind{x_2 \in (V_2 \backslash V_1)} + \Ind{x_2 \in (V_1 \cap V_2)})\\
&= \Ind{x_1 \in (V_1 \backslash V_2)}\Ind{x_2 \in (V_2 \backslash V_1)} + \Ind{x_1 \in (V_1 \backslash V_2)}\Ind{x_2 \in (V_1 \cap V_2)} \\
& \qquad + \Ind{x_1 \in (V_1 \cap V_2)}\Ind{x_2 \in (V_2 \backslash V_1)} + \Ind{x_1 \in (V_1 \cap V_2)}\Ind{x_2 \in (V_1 \cap V_2)}.
\end{align*}
For a general $n$, one can use induction and this concludes the proof.  
\end{proof}

\begin{proposition}\label{prop:AppendixLocalization}
For every $s > 0$ and $f \in \cH^1(Q_s)$, we have $\A_s f \in \cH^1(Q_s)$, and for every $x \in \supp(\mu) \cap Q_s$
\begin{align}\label{eq:Commute2}
\nabla (\A_s f)(\mu, x) = \A_s (\nabla f)(\mu, x).
\end{align}
Moreover, if $s > 2$ and $f \in \mcl A(Q_s)$, then  $\A_s f \in \mcl A(Q_{s-2})$.
\end{proposition}

\begin{proof}
At first, we should remark the well-definedness of the right side of \cref{eq:Commute2}. Notice that the Poisson measure can be decomposed as a sum of the independent parts ${\mu = \mu \mres \ov{Q}_s + \mu \mres \ov{Q}_s^c}$, we have 
\begin{align*}
\A_s f = \int_{\mmd(\Rd)} f(\mu \mres \ov{Q}_s + \mu' \mres \ov{Q}_s^c) \, \d \Pr(\mu').
\end{align*}
Thus the right-hand side of \cref{eq:Commute2} is defined as 
\begin{align}\label{eq:Commute2Right}
\A_s (\nabla f)(\mu, x) := \int_{\mmd(\Rd)} \nabla f(\mu \mres \ov{Q}_s + \mu' \mres \ov{Q}_s^c, x)\, \d \Pr(\mu').
\end{align}
%
We prove \cref{eq:Commute2} and $\A_s f \in \cH^1(Q_s)$ for the functions in ${\cC^{\infty}(Q_s) \cap \cH^1(Q_s)}$ as they are dense, and we can focus on the case $\mu(\ov{Q}_s) = n$ fixed. We use Lemma~\ref{lem:AppendixProjection} to write 
$$
f\Ll(\sum_{i = 1}^n \de_{x_i} + \mu \mres \ov{Q}_s^c\Rr)= f_n(x_1, \cdots, x_n, \mu \mres \ov{Q}_s^c).
$$ 
Following the property of product measure, for every $(x_1, x_2, \cdots, x_n) \in (\ov{Q}_s)^n$, the mapping 
\begin{align*}
{\mu \mres \ov{Q}_s^c \mapsto f_n(x_1, \cdots, x_n, \mu \mres \ov{Q}_s^c)},    
\end{align*}
is $\mcl F_{\ov{Q}_s^c}$-measurable. Thus for every $(x_1, x_2, \cdots, x_n) \in (Q_s)^n$, the mapping 
\begin{align*}
{\mu \mres \ov{Q}_s^c \mapsto \nabla_{x_k} f_n(x_1, \cdots, x_n, \mu \mres \ov{Q}_s^c)},    
\end{align*}
is also $\mcl F_{\ov{Q}_s^c}$-measurable because it is the limit of $\mcl F_{\ov{Q}_s^c}$-measurable functions. Then we observe that
\begin{align*}
\norm{\nabla f_n}_{L^{\infty}((Q_s)^n)} &= \sup_{(Q_s)^n}\Ll(\sum_{k=1}^n\vert \nabla_{x_k}f_n(x_1, \cdots, x_n, \mu \mres \ov{Q}_s^c)\vert^2\Rr)^{\frac{1}{2}}\\
&= \sup_{(\Q \cap Q_s)^n}\Ll(\sum_{k=1}^n\vert \nabla_{x_k}f_n(x_1, \cdots, x_n, \mu \mres \ov{Q}_s^c)\vert^2\Rr)^{\frac{1}{2}},
\end{align*}
as a supremum of a countable number of $\mcl F_{\ov{Q}_s^c}$-measurable functions, is finite and $\mcl F_{\ov{Q}_s^c}$-measurable. Thus we can define a cut-off version of $f$ that 
\begin{align*}
f^{n,M} = f \Ind{\mu(\ov{Q}_s)=n}\Ind{\norm{\nabla f_n}_{L^{\infty}((Q_s)^n)} \leq M},
\end{align*}
and we can establish \cref{eq:Commute2} at first for $f^{n,M}$. For every $x \in Q_s \cap \supp(\mu)$, we have 
\begin{align*}
&\partial_k (\A_s f^{n,M})(\mu, x)  \\
&= \lim_{h \to 0} \int_{\mmd(\Rd)}\frac{f((\mu - \delta_x + \delta_{x+h\e_k})\mres \ov{Q}_s + \mu' \mres \ov{Q}_s^c) - f(\mu \mres \ov{Q}_s + \mu' \mres \ov{Q}_s^c)}{h}\\
& \qquad \qquad \times \Ind{\norm{\nabla f_n}_{L^{\infty}((Q_s)^n)} \leq M}  \, \d \Pr(\mu')\Ind{\mu(\ov{Q}_s)=n},
\end{align*}
for $h$ small enough such that $x + h \e_k \in Q_s$. Since $f \in \cC^{\infty}(Q_s)$, we use \Cref{lem:AppendixProjection} and the mean value theorem 
\begin{align*}
\frac{f(\mu - \delta_x + \delta_{x+h\e_k}) - f(\mu)}{h} = \partial_k f(\mu  - \delta_x + \delta_{x+ \theta\e_k}, x+  \theta \e_k),
\end{align*}
for some $\theta \in (0,1)$. With the indicator $\Ind{\norm{\nabla f_n}_{L^{\infty}((Q_s)^n)} \leq M}$, this term is bounded by $M$, so we can use the dominated convergence theorem that 
\begin{align*}
&\partial_k (\A_s f^{n,M})(\mu, x)\\
&= \int_{\mmd(\Rd)}\lim_{h \to 0}\frac{f((\mu - \delta_x + \delta_{x+h\e_k})\mres \ov{Q}_s + \mu' \mres \ov{Q}_s^c) - f(\mu \mres \ov{Q}_s + \mu' \mres \ov{Q}_s^c)}{h}\\
& \qquad \qquad \times \Ind{\norm{\nabla f_n}_{L^{\infty}((Q_s)^n)} \leq M}  \, \d \Pr(\mu')\Ind{\mu(\ov{Q}_s)=n}\\
&=\int_{\mmd(\Rd)}\partial_k f(\mu \mres \ov{Q}_s + \mu' \mres \ov{Q}_s^c, x)\Ind{\norm{\nabla f_n}_{L^{\infty}((Q_s)^n)} \leq M}  \, \d \Pr(\mu')\Ind{\mu(\ov{Q}_s)=n},
\end{align*}
which establishes the \cref{eq:Commute2} in the sense \cref{eq:Commute2Right}. By Jensen's inequality and Fubini's lemma, we observe that 
\begin{align*}
&\Er\Ll[\int_{Q_s} \vert\nabla (\A_s f^{n,M}) \vert^2(\mu, x) \, \d \mu(x)\Rr]\\
& = \Er\Ll[\int_{Q_s} \Ll\vert\int_{\mmd(\Rd)}\nabla  f^{n,M}(\mu \mres \ov{Q}_s + \mu' \mres \ov{Q}_s^c, x) \, \d \Pr(\mu')\Rr\vert^2 \, \d \mu(x)\Rr]\\
& \leq \Er\Ll[\int_{Q_s} \int_{\mmd(\Rd)} \Ll\vert\nabla  f^{n,M} (\mu \mres \ov{Q}_s + \mu' \mres \ov{Q}_s^c, x) \Rr\vert^2\, \d \Pr(\mu') \, \d \mu(x)\Rr]\\
& = \Er\Ll[\int_{Q_s} \Ll\vert \nabla  f^{n,M} \Rr\vert^2 (\mu, x)  \, \d \mu(x)\Rr],
\end{align*}
which implies that $\A_s f^{n,M} \in \cH^1(Q_s)$. Then we use once again Jensen's inequality for $f^{n,M}$ and $f^{n,M'}$ with $M < M'$
\begin{align*}
&\Er\Ll[\int_{Q_s} \vert\nabla (\A_s f^{n,M}) - \nabla (\A_s f^{n,M'})\vert^2(\mu, x) \, \d \mu(x)\Rr] \\
& \leq \Er\Ll[\int_{Q_s} \Ll\vert \nabla  f^{n,M} - \nabla f^{n,M'}\Rr\vert^2 (\mu, x)  \, \d \mu(x)\Rr]\\
& = \Er\Ll[\int_{Q_s} \Ll\vert \nabla  f \Rr\vert^2 (\mu, x)  \, \d \mu(x) \Ind{\mu(\ov{Q}_s) = n}\Ind{M < \norm{\nabla f_n}_{L^{\infty}((Q_s)^n)} \leq M'}\Rr].
\end{align*}
So $\{f^{n,M}\}_{M \geq 0}$ gives a Cauchy sequence in $\cH^1(Q_s)$, and the only candidate is $f\Ind{\mu(\ov{Q}_s) = n}$ because it is the limit in $\cL^2$. By this and a linear combination, we establish \cref{eq:Commute2} for $f$ in ${\cC^{\infty}(Q_s) \cap \cH^1(Q_s)}$, and we can then extend to a general function in $\cH^1(Q_s)$ by the density argument.  

\smallskip
For the part of $\a$-harmonic function, we suppose $f \in \mcl A(Q_s)$ and test $\phi \in \cH^1_0(Q_{s-2})$ with \cref{eq:Commute2}, 
\begin{align*}
&\Er\Ll[\int_{Q_{s-2}} (\nabla \A_s f)(\mu, x) \cdot \a(\mu, x) \nabla \phi(\mu, x) \, \d \mu(x)\Rr]\\
&=\Er\Ll[\int_{Q_{s-2}}  \A_s (\nabla f)(\mu, x) \cdot \a(\mu, x) \nabla \phi(\mu, x) \, \d \mu(x)\Rr]\\
&=\Er\Ll[\int_{Q_{s-2}}  \Ll(\int_{\mmd(\Rd)} \nabla f(\mu \mres \ov{Q}_s + \mu' \mres \ov{Q}_s^c, x) \, \d \Pr(\mu')\Rr) \cdot \a(\mu, x) \nabla \phi(\mu, x) \, \d \mu(x)\Rr]. 
\end{align*}
Restricted on $x \in Q_{s-2}$, we have $\a(\mu, x), \nabla \phi(\mu, x)$ are $\mcl F_{Q_s} \otimes \mcl B_{Q_s}$-measurable,  so we have 
\begin{align*}
\forall x \in \supp(\mu) \cap Q_{s-2}, \qquad \a(\mu, x) \nabla \phi(\mu, x) =  \a(\mu \mres \ov{Q}_s, x) \nabla \phi(\mu \mres \ov{Q}_s, x).    
\end{align*}
We can enter the part in the integration, and then use Fubini's lemma
\begin{align*}
&\Er\Ll[\int_{Q_{s-2}} (\nabla \A_s f)(\mu, x) \cdot \a(\mu, x) \nabla \phi(\mu, x) \, \d \mu(x)\Rr]\\
&=\Er\Ll[\int_{Q_{s-2}}  \Ll(\int_{\mmd(\Rd)} \nabla f(\mu \mres \ov{Q}_s + \mu' \mres \ov{Q}_s^c, x) \cdot  \a(\mu \mres \ov{Q}_s, x) \nabla \phi(\mu \mres \ov{Q}_s, x) \, \d \Pr(\mu')\Rr)  \, \d \mu(x)\Rr] \\
&=\Er\Ll[\int_{Q_{s-2}}  \nabla f(\mu, x) \cdot \a(\mu, x) \nabla \phi(\mu, x)   \, \d \mu(x)\Rr] \\
&= 0.
\end{align*}
In the last step, we use $f \in \mcl A(Q_s)$ and this finishes the proof.
\end{proof}

\section{Equivalent definitions of the effective diffusion matrix}
\label{sec:app.equiv.def}

Recall that we defined $\ab(U)$ and $\ab_*(U)$ according to \cref{eq:defAffine}-\cref{eq:NuMatrix}. The proof of Theorem~\ref{thm:main} ensures the existence of a constant $C < \infty$ and an exponent $\al > 0$ such that for every $m \in \N$,
\begin{equation}
\label{eq:rate}
\Ll| \ab(\cu_m) - \ab \Rr| + \Ll| \ab_*(\cu_m) - \ab \Rr|  \leq C 3^{-\alpha m}.
\end{equation}
Throughout this appendix, we will only rely on the qualitative statement that
\begin{equation}  
\label{e.identities.ab}
\ab = \lim_{m \to \infty} \ab(\cu_m) = \lim_{m \to \infty} \ab_*(\cu_m).
\end{equation}

The first main goal of this appendix is to demonstrate that the definition we chose for the bulk diffusion matrix indeed coincides with the ``stationary'' definition appearing in works such as \cite{varadhanII,fuy}. Adapted to our context, this alternative definition takes the following form. For
\begin{align}
\label{eq:defGamma}
\Gamma := \Ll\{\mu \mapsto \int_\Rd  \tau_x g(\mu) \, \d x, \quad g \in \cC_c^\infty(\Rd) \cap \cH^1_0(\Rd) \Rr\},
\end{align}
we let $\at$ be the $d$-by-$d$ matrix such that for every $p \in \Rd$, 
\begin{align}
\label{eq:defab2}
p \cdot \at p := \inf_{u \in \Gamma} \Er \Ll[ (p + \nabla u(\mu + \delta_0, 0)) \cdot \a(\mu + \delta_0, 0) (p + \nabla u(\mu+\delta_0, 0)) \Rr],
\end{align}
where in \cref{eq:defGamma}, we used the notation $\tau_x g(\mu) := g(\tau_{-x} \mu)$. Notice that for $g \in \cC_c^\infty(\Rd) \cap \cH^1_0(\Rd)$ and $u  : \mu \mapsto  \int_\Rd  \tau_x g(\mu) \, \d x\in \Gamma$, the function $\mu \mapsto u(\mu)$ is typically not well-defined unless $\mu$ is of finite support. However, the quantity $\nabla u(\mu,\cdot)$ makes sense whenever the measure $\mu$ is $\sigma$-finite, since in the sum $\nabla u(\mu, y) = \int_\Rd  \nabla (\tau_x g)(\mu, y) \, \d x$, the function~$g$ is local, and thus the integrand $\nabla (\tau_x g)(\mu, y)$ is non-zero only for $x$ in a bounded set. With this interpretation of $\nabla u$, the right side of \cref{eq:defab2} is well-defined.

\begin{theorem}
\label{thm:TwoAb}
We have $\ab = \at$. 
\end{theorem}

The second main goal of this appendix is to demonstrate that the infimum in~\cref{eq:defab2} is achieved in a suitable completion of the space $\Gamma$. We also show that the optimizer, which we call the (stationary) corrector, can be obtained as a limit of approximations based on the finite-volume optimizers for $\nu$ or $\nu^*$. Denoting
\begin{equation*}  
\mmb(\Rd) := \Ll\{ (\mu,x) \in \mmd(\Rd) \times \Rd \ : \ x \in \supp \mu \Rr\} ,
\end{equation*}
we introduce the space
\begin{equation*}
\cL^2_\bullet 
:= \Ll\{f: \mmb(\Rd)\to \R \ : \ f \text{ is measurable and } \Er\Ll[\int_\Rd \vert f(\mu, x) \vert^2 \, \d \mu(x)\Rr] < \infty \Rr\},
\end{equation*}
and its local version
\begin{multline}
\cL^2_{\bullet,\mathrm{loc}}  := \Big\{f: \mmb(\Rd)\to \R \ : \ \mbox{$f$ is measurable and }
\\
\mbox{for every compact $K \subset \Rd$, } \ 
 \Er\Ll[\int_K \vert f(\mu, x) \vert^2 \, \d \mu(x)\Rr] < \infty\Big\}.
\end{multline}
In these definitions, we say that $f : \mmb(\Rd) \to \R$ is measurable provided that the mapping
\begin{equation*}  
\Ll\{
\begin{array}{rcl}  
\mmd(\Rd) \times \Rd & \to & \R \\
(\mu,x) & \mapsto & f(\mu,x) \1_{\{x \in \supp \mu\}},
\end{array}
\Rr.
\end{equation*}
is $\mcl F \otimes \mcl B$-measurable. The space $\cL^2_{\bullet,\mathrm{loc}}$ is naturally endowed with the family of seminorms
\begin{equation*}  
\Ll\{
\begin{array}{rcl}  
\cL^2_{\bullet,\mathrm{loc}} & \to & \R \\
f & \mapsto & \Er\Ll[\int_K \vert f(\mu, x) \vert^2 \, \d \mu(x)\Rr]^\frac 1 2 ,
\end{array}
\Rr.
\end{equation*}
indexed by all the compact sets $K \subset \Rd$. This family of seminorms turns $\cL^2_{\bullet,\mathrm{loc}}$ into a complete space. 

For every $p \in \Rd$ and $m \in \N$, we let $\phi(\cdot, \cu_m,p)$ be such that the minimizer in the definition of $\nu(\cu_m,p)$ is $\ell_{p,\cu_m} + \phi(\cdot,\cu_m,p)$, where we recall that $\ell_{p,\cu_m}$ was defined in~\cref{eq:defAffine}. Similarly, we let $\phi^*(\cdot,\cu_m,p)$ be such that $\ell_{p,\cu_m} + \phi^*(\cdot,\cu_m,p)$ is the maximizer in the definition of $\nu^*(\cu_m,\ab_*(\cu_m) p)$. More precisely, using the notation introduced in Proposition~\ref{prop:NuBase}, we have
\begin{equation*}  
v(\cdot,\cu_m,p) = \ell_{p,U} + \phi(\cdot,\cu_m,p),
\end{equation*}
and
\begin{equation*}  
u(\cdot,\cu_m, \ab_*(\cu_m) p) = \ell_{p,\cu_m} +\phi^*(\cdot,\cu_m,p) .
\end{equation*}
Finally, we define $\nabla \td{\phi}_{p,m}$ according to the formula
\begin{equation}
\label{e.def.tdphi}
\tilde{\phi}_{p,m} : \mu \mapsto \frac{1}{ \vert \cu_m \vert} \int_{\Rd} \tau_x \phi(\mu, \cu_m, p) \, \d x 
= \frac{1}{ \vert \cu_m \vert} \int_{\Rd} \phi(\mu, x + \cu_m, p) \, \d x.
\end{equation}
Notice that, while $\tilde{\phi}_{p,m}(\mu)$ is ill-defined when $\mu \sim \poi(\rho)$, the quantity $\nabla\tilde{\phi}_{p,m}(\mu,\cdot)$ is still well-defined, for the same reason as in the discussion following \cref{eq:defab2}. Our second main result is as follows.
\begin{theorem}
\label{t.local.conv}
The following statements hold for every $p \in \Rd$. 

(1) The sequence $(\nabla \td{\phi}_{p,m})_{m \in \N}$ is a Cauchy sequence in $(\cL^2_{\bullet,\mathrm{loc}})^d$. Its limit, which we denote by $\nabla \phi_p$, satisfies
\begin{align}
\label{eq:CorrectorHarmonic}
\forall v \in \cH^1_0(\Rd), \qquad \Er\Ll[\int_{\Rd} \nabla v \cdot \a (p + \nabla \phi_{p}) \, \d \mu \Rr] = 0.
\end{align}

(2) We have
\begin{align}\label{eq:AppCorrector}
\lim_{m \to \infty}\Er\Ll[ \frac{1}{\rho \vert \cu_m \vert} \int_{\cu_m} \vert \nabla \phi(\mu, \cdot, \cu_m, p) -  \nabla \phi_{p}(\mu, \cdot)\vert^2 \, \d \mu \Rr] = 0,
\end{align}
as well as
\begin{align}\label{eq:AppCorrectorDual}
\lim_{m \to \infty}\Er\Ll[ \frac{1}{\rho \vert \cu_m \vert} \int_{\cu_m} \vert \nabla \phi^*(\mu, \cdot, \cu_m, p) -  \nabla \phi_{p}(\mu, \cdot)\vert^2 \, \d \mu \Rr] = 0.
\end{align}

(3) The effective diffusion matrix $\ab$ satisfies
\begin{equation}\label{eq:defab4}
 p \cdot \ab p = \Er \Ll[ (p + \nabla \phi_p(\mu + \delta_0, 0)) \cdot \a(\mu + \delta_0, 0) (p + \nabla \phi_p(\mu+\delta_0, 0)) \Rr],
\end{equation}
as well as
\begin{equation}
\label{eq:defab5}
 \ab p = \Er \Ll[ \a(\mu + \delta_0, 0) (p + \nabla \phi_p(\mu+\delta_0, 0)) \Rr].
\end{equation}
\end{theorem}
As a preparation towards the proof of these results, we state in the following proposition a number of elementary properties about the function space $\Gamma$.
\begin{proposition}\label{prop:GammaProp}
Let $g \in \cC_c^\infty(\Rd) \cap \cH^1_0(\Rd)$ be an $\mcl F_{\cu_n}$-measurable function, and let $u := \int_\Rd  \tau_x g \, \d x \in \Gamma$. The following properties hold:
\begin{enumerate}
\item $y \mapsto \nabla u(\mu + \delta_y, y)$ is a stationary field, i.e. $\nabla u(\mu + \delta_y, y) = \nabla u(\tau_{-y}\mu + \delta_0, 0)$.
\item $\nabla u$ has mean zero, that is, 
\begin{align}\label{eq:meanZero}
\Er\Ll[\nabla u(\mu + \delta_0, 0)\Rr] = 0.
\end{align}
\item $\nabla u$ satisfies the estimate 
\begin{multline}\label{eq:GammaComparison}
\Er \Ll[ \nabla u(\mu + \delta_0, 0) \cdot \a(\mu + \delta_0, 0) \nabla u(\mu+\delta_0, 0) \Rr] \\
\leq 3^{2dn}\Er\Ll[\frac{1}{\rho \vert \cu_n \vert}\int_{\cu_n} \nabla g(\mu, y) \cdot \a(\mu, y)  \nabla g(\mu, y) \, \d \mu(y) \Rr].
\end{multline} 
\end{enumerate}
\end{proposition}
\begin{proof}
(1)  We use the definition to write 
\begin{align*}
\partial_k u(\mu + \delta_y,y) &= \lim_{h \to 0}\frac{1}{h} \Ll(\int_\Rd  \tau_x g(\mu + \delta_{y + h \e_k}) - \tau_x g(\mu + \delta_{y}) \, \d x\Rr) \\
&=  \lim_{h \to 0}\frac{1}{h} \Ll(\int_{y + \cu_n}  \tau_x g(\mu + \delta_{y + h \e_k}) - \tau_x g(\mu + \delta_{y}) \, \d x\Rr) \\
&= \lim_{h \to 0}\frac{1}{h} \Ll(\int_{y + \cu_n}  \tau_{x-y} g(\tau_{-y}\mu + \delta_{h \e_k}) - \tau_{x-y} g(\tau_{-y}\mu + \delta_{0}) \, \d x\Rr).
\end{align*}
From the first line to the second line, we used the fact that $g$ is $\mcl F_{\cu_n}$-measurable, so if the transport vector $x$ does not belong to $y + \cu_n$, then the integrand vanishes (up to a boundary layer that vanishes in the limit $h \to 0$). We then do the change of variables $z = x-y$ to get
\begin{align*}
\partial_k u(\mu + \delta_y,y) &= \lim_{h \to \infty}\frac{1}{h} \Ll(\int_{ \cu_n}  \tau_{z} g(\tau_{-y}\mu + \delta_{h \e_k}) - \tau_{z} g(\tau_{-y}\mu + \delta_{0}) \, \d z\Rr) \\
&=  \partial_k u(\tau_{-y}\mu + \delta_0,0),
\end{align*}
which means that $y \mapsto \nabla u(\mu + \delta_y,y)$ is a stationary gradient field.

(2) We use the equations developed in the last question. Since $g \in \cC_c^\infty(\Rd)$, we can exchange the integration and derivative and get 
\begin{align}\label{eq:DeriGamma}
\nabla u(\mu + \delta_y,y) &= \int_{\Rd}  \nabla \tau_{z} g(\tau_{-y}\mu + \delta_0, 0)  \, \d z . 
\end{align}
We evaluate this gradient at $y=0$
\begin{align*}
\Er\Ll[\nabla u(\mu + \delta_0, 0)\Rr] &= \Er\Ll[\int_{\Rd}  \nabla \tau_{z} g(\mu + \delta_0, 0)  \, \d z\Rr] \\
&= \Er\Ll[\int_{\Rd}  \nabla  g(\tau_{-z}\mu + \delta_{-z}, -z)  \, \d z\Rr]\\
&= \Er\Ll[\int_{\Rd}  \nabla  g(\mu + \delta_{z}, z)  \, \d z\Rr] \\
&= 0.
\end{align*}
From the second line to the third line, we used the stationarity of the Poisson point process. Because $g \in \cC_c^\infty(\Rd)$,  then $\int_{\Rd}  \nabla  g(\mu + \delta_{z}, z)  \, \d z = 0$ and the integration in the third line vanishes.

(3) We pick a cube $Q_L = \Ll(-\frac{L}{2}, \frac{L}{2}\Rr)^d$ with $L > 0$, make use of the stationarity of ${y \mapsto \nabla u(\mu + \delta_y, y)}$ and Mecke's identity (see \cite[Theorem 4.1]{bookPoisson})
\begin{align*}
&\Er \Ll[  \nabla u(\mu + \delta_0, 0) \cdot \a(\mu + \delta_0, 0)  \nabla u(\mu+\delta_0, 0) \Rr]\\
& = \Er \Ll[ \frac{1}{ \vert Q_L \vert} \int_{Q_L} \nabla u(\mu + \delta_y, y) \cdot \a(\mu + \delta_y, y)  \nabla u(\mu+\delta_y, y) \, \d y \Rr] \\
& = \Er \Ll[ \frac{1}{ \rho \vert Q_L \vert} \int_{Q_L} \nabla u(\mu, y) \cdot \a(\mu, y) \nabla u(\mu, y) \, \d \mu(y) \Rr].
\end{align*}
We put the definition $u = \int_\Rd  \tau_x g \, \d x$ into the equation. For the gradient at $y$, as $g$ is $\mcl F_{\cu_n}$-measurable, thus only the term $\nabla (\tau_x g)(\mu, y)$ for $x \in y + \cu_n$ contributes. This gives  
\begin{multline*}
\Er \Ll[ \nabla u(\mu + \delta_0, 0) \cdot \a(\mu + \delta_0, 0) \nabla u(\mu+\delta_0, 0) \Rr] \\
= \Er \Ll[ \frac{1}{ \rho \vert Q_L \vert} \int_{Q_L} \Ll(\int_{y + \cu_n}  \nabla (\tau_x g)(\mu, y) \, \d x\Rr) \cdot \a(\mu, y) \Ll(\int_{y + \cu_n}  \nabla(\tau_x g)(\mu, y) \, \d x\Rr) \, \d \mu(y) \Rr].
\end{multline*}
We next apply Jensen's inequality to obtain that
\begin{equation}\label{eq:TrickJensen}
\begin{split}
&\Er \Ll[ \nabla u(\mu + \delta_0, 0) \cdot \a(\mu + \delta_0, 0)  \nabla u(\mu+\delta_0, 0) \Rr] \\
&= \Er \Ll[ \frac{\vert \cu_n \vert^2}{ \rho \vert Q_L \vert} \int_{Q_L} \Ll(\fint_{y + \cu_n}  \nabla (\tau_x g)(\mu, y) \, \d x\Rr) \cdot \a(\mu, y) \Ll(\fint_{y + \cu_n}  \nabla (\tau_x g)(\mu, y) \, \d x\Rr) \, \d \mu(y) \Rr]\\
&\leq \Er \Ll[ \frac{\vert \cu_n \vert^2}{ \rho \vert Q_L \vert} \int_{Q_L} \Ll(\fint_{y + \cu_n}  \nabla (\tau_x g)(\mu, y) \cdot \a(\mu, y) \nabla (\tau_x g)(\mu, y)\, \d x\Rr)   \, \d \mu(y) \Rr]\\
&\leq \Er \Ll[ \frac{\vert \cu_n \vert}{ \rho \vert Q_L \vert} \int_{Q_{L+3^n}} \Ll(\int_{x + \cu_n}  \nabla (\tau_x g)(\mu, y) \cdot \a(\mu, y) \nabla (\tau_x g)(\mu, y) \, \d \mu(y)\Rr)   \, \d x \Rr].
\end{split}
\end{equation}
In the last line, we use Fubini's lemma and exchange ${\int (\cdots) \, \d x}$ with ${\int (\cdots) \, \d \mu(y)}$. In this procedure, we have to enlarge the domain from $Q_L$ to $Q_{L+3^n}$, because for the gradient at $y \in Q_L$, $\nabla (\tau_x g)(\mu, y )$ contributes for the transport $x \in Q_{L+3^n}$ (see \Cref{fig:cubeContributes1} as an illustration). Using the stationarity of the Poisson point process, we have 
\begin{multline*}
\Er \Ll[ \int_{x + \cu_n}  \nabla (\tau_x g)(\mu, y) \cdot \a(\mu, y) \nabla (\tau_x g)(\mu, y) \, \d \mu(y) \Rr] \\
= \Er \Ll[ \int_{\cu_n}  \nabla g(\mu, y) \cdot \a(\mu, y) \nabla g(\mu, y) \, \d \mu(y) \Rr],
\end{multline*}
which helps us conclude that 
\begin{align*}
&\Er \Ll[ (\xi + \nabla u(\mu + \delta_0, 0)) \cdot \a(\mu + \delta_0, 0) (\xi + \nabla u(\mu+\delta_0, 0)) \Rr] \\
& \leq 3^{2dn} \frac{\vert Q_{L+3^n} \vert}{  \vert Q_L \vert} \Er \Ll[  \frac{1}{\rho \vert \cu_n \vert}\int_{\cu_n}  \nabla g(\mu, y) \cdot \a(\mu, y) \nabla  g(\mu, y) \, \d \mu(y) \Rr].
\end{align*}
We take $L \to \infty$ and obtain the desired result. 
\end{proof}
\begin{remark}\label{rmk:GammaComparison}
The inequality \cref{eq:GammaComparison} is essentially sharp when $\nabla g$ itself is close to a stationary field. Indeed, if $g$ is close to a stationary field, then 
\begin{align*}
\nabla \tau_x g(\mu, y) = \nabla g(\tau_{-x}\mu, y-x) \simeq \nabla g(\mu, y),
\end{align*}
which implies that the application of Jensen's inequality in \cref{eq:TrickJensen} is essentially sharp. The error introduced by a boundary layer in a subsequent step of the proof disappears as we take $L \to \infty$ at the end.
\end{remark}

As a corollary of \Cref{prop:GammaProp}, we can also propose the following equivalent definition of $\at$.
\begin{corollary}
For any open set $U \subset \Rd$, we have 
\begin{align}\label{eq:defab3}
\xi \cdot \at \xi = \inf_{u \in \Gamma} \Er \Ll[ \frac{1}{\rho \vert U \vert} \int_{U} (\xi + \nabla u) \cdot \a (\xi + \nabla u) \, \d \mu \Rr].
\end{align}
\end{corollary}
\begin{proof}
It is a direct result of Mecke's identity (see \cite[Theorem 4.1]{bookPoisson}) and the stationarity of ${y \mapsto \nabla u(\mu + \delta_y, y)}$.
\end{proof}

\begin{figure}[h!]
\centering
\includegraphics[scale=0.5]{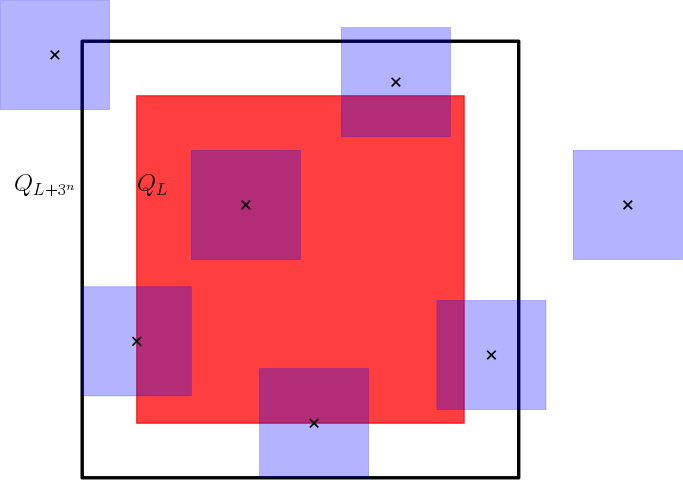}
\caption{The red cube represents the cube $Q_L$ and the blue cubes represent the transport of small cubes $x + \square_n$. It is clear that for $x \in Q_{L + 3^n}$, the blue cubes and the red cube intersect.}\label{fig:cubeContributes1}
\end{figure}

With the help of \Cref{prop:GammaProp}, we can now prove the first main theorem of this appendix. 
\begin{proof}[Proof of Theorem~\ref{thm:TwoAb}]
We decompose the proof into two steps. 

\textit{Step 1: Bound from below $\at \geq \ab$.} We fix $m \in \N$ and a sequence of approximate minimizers $\{\phi_{{p}}^{(i)}\}_{i \geq 1}$ for the variational problem in \cref{eq:defab2}, which we write in the form
\begin{align*}
\phi_{{p}}^{(i)}(\mu) = \int_{\Rd} \tau_x g_i(\mu) \, \d x
\end{align*}
for some $g_i \in \cC_c^\infty(\Rd) \cap \cH^1_0(\Rd)$. 
Now we propose a modified version in $\cH^1(\cu_m)$ defined by
\begin{align*}
\tilde{\phi}_{{p}}^{(i)}(\mu) := \int_{K_i} \tau_x g_i(\mu) \, \d x,
\end{align*}
with $K_i \subset \Rd$ a large compact set so that $\tilde{\phi}_{{p}}^{(i)} \in \cH^1(\cu_m)$ and 
\begin{align*}
\forall y \in \cu_m, \qquad \nabla \tilde{\phi}_{{p}}^{(i)}(\mu, y) = \nabla \phi_{{p}}^{(i)}(\mu, y).
\end{align*}
Then we test ${p} + \nabla \tilde{\phi}_{{p}}^{(i)}$ in the optimization problem for $\nu^*(\cu_m, q)$ to get that 
\begin{align*}
\frac{1}{2} q \cdot \ab_*^{-1}(\cu_m) q &\geq \Er \Ll[\frac{1}{\rho \vert \cu_m \vert}\int_{\cu_m} \Ll( -\frac 1 2 ({p} + \nabla \tilde{\phi}_{{p}}^{(i)}) \cdot \a ({p} + \nabla \tilde{\phi}_{{p}}^{(i)}) + q \cdot ({p} + \nabla \tilde{\phi}_{{p}}^{(i)}) \Rr) \, \d \mu \Rr]\\
&= \Er \Ll[\frac{1}{\rho \vert \cu_m \vert}\int_{\cu_m} \Ll( -\frac 1 2 ({p} + \nabla \phi_{{p}}^{(i)}) \cdot \a ({p} + \nabla \phi_{{p}}^{(i)}) + q \cdot ({p} + \nabla \phi_{{p}}^{(i)}) \Rr) \, \d \mu \Rr].
\end{align*}
We use the stationarity of $y \mapsto  \nabla \phi_{{p}}^{(i)}(\mu + \delta_y, y)$, \cref{eq:meanZero} and let $i \to \infty$ to obtain  
\begin{align*}
\frac{1}{2} q \cdot \ab_*^{-1}(\cu_m) q \geq - \frac{1}{2} {p} \cdot \at {p} + {p} \cdot q.
\end{align*}
Taking $q = \ab_*(\cu_m) {p}$ leads to
\begin{align*}
{p} \cdot \at {p} \geq {p} \cdot \ab_*(\cu_m) {p}.
\end{align*}
Finally, we let $m \to \infty$ and conclude that $\at \geq \ab$.

\textit{Step 2: Bound from above $\at \leq \ab$.} We hope to prove $\at \leq \ab$ by testing the variational formula \cref{eq:defab2} with a suitable candidate, namely the function $\td \phi_{p,m}$ introduced in \cref{e.def.tdphi}.
Since $\phi(\cdot,\cu_m,p) \in \cH^1_0(\cu_m)$, and using \cref{eq:GammaComparison}, we can approximate $\nabla \td \phi_{p,m}$ in $(\cL^2_{\bullet,\mathrm{loc}})^d$ arbitrarily closely with elements of $\Gamma$. It thus follows that we can use $\td \phi_{p,m}$ as a candidate in the variational problem in \cref{eq:defab2}, and use the comparison inequality \cref{eq:GammaComparison} to get that   
\begin{align}
\label{e.ineq.at.ab}
{p} \cdot \at {p} & \leq  \Er \Ll[ ({p} + \nabla \tilde{\phi}_{{p},m}(\mu + \delta_0, 0)) \cdot \a(\mu + \delta_0, 0) ({p} + \nabla \tilde{\phi}_{{p},m}(\mu+\delta_0, 0)) \Rr] \\
\notag
& \leq \Er\Ll[\frac{1}{\rho \vert \cu_m \vert}\int_{\cu_m} ({p} + \nabla\phi(\mu, y, \cu_m, {p})) \cdot \a  ({p} + \nabla\phi(\mu, y, \cu_m, {p})) \, \d \mu(y) \Rr]\\
\notag
& = {p} \cdot \ab(\cu_m) {p}.
\end{align}
Finally, we let $m \to \infty$ and conclude that $\at \leq \ab$.
\end{proof}

In the proof above, we used $\{\tilde{\phi}_{{p},m}\}_{m \geq 1}$ as a sequence of approximate minimizers for the variational problem in \cref{eq:defab2}. This already gives us a good hint for the validity of at least some of the statements in Theorem~\ref{t.local.conv}. We now turn to the proof of the first part of this result.

\begin{proof}[Proof of part (1) of Theorem~\ref{t.local.conv}]
We decompose the proof into four steps.

\textit{Step 1: $\{\nabla \tilde{\phi}_{{p},m}\}_{m \geq 1}$ is a Cauchy sequence in ${(\cL^2_{\bullet,\mathrm{loc}})^d}$.}
We fix $n < m$ and, recalling the notation $\Z_{m,n} := 3^n \Zd \cap \cu_m$, we observe that 
\begin{align*}
\nabla \tilde{\phi}_{{p},n}(\mu, y) &=  \frac{1}{ \vert \cu_n \vert} \int_{\Rd} \nabla \phi(\mu, y, x + \cu_n, {p}) \, \d x \\
&= \frac{1}{ \vert \cu_m \vert} \int_{\Rd} \sum_{z \in \Z_{m,n}} \nabla \phi(\mu, y, x + z + \cu_n, {p}) \, \d x \\
&= \frac{1}{ \vert \cu_m \vert} \int_{\Rd} \nabla \tau_x \phi_{{p}, m, n}(\mu, y) \, \d x ,
\end{align*}
where the function $\phi_{{p}, m, n}$ is defined as 
\begin{align}
\label{e.def.phi.pmn}
\phi_{{p}, m, n}(\mu) := \sum_{z \in \Z_{m,n}} \phi(\mu, z+\cu_n, {p}).
\end{align}
For any compact set $K$, we use Mecke's identity (see \cite[Theorem 4.1]{bookPoisson}) and the stationarity of $\nabla \tilde{\phi}_{{p},m}$ 
\begin{align*}
&\Er\Ll[\int_K \vert \nabla \tilde{\phi}_{{p},m} - \nabla \tilde{\phi}_{{p},n}\vert^2(\mu, y)  \, \d \mu(y)\Rr] \\
&= \rho \Er\Ll[\int_K \vert \nabla \tilde{\phi}_{{p},m} - \nabla \tilde{\phi}_{{p},n}\vert^2(\mu+\delta_y, y)  \, \d y\Rr] \\
&= \rho \vert K \vert \Er\Ll[\vert \nabla \tilde{\phi}_{{p},m} - \nabla \tilde{\phi}_{{p},n}\vert^2(\mu+\delta_0, 0) \Rr].
\end{align*}
Then we use the comparison inequality \cref{eq:GammaComparison} and obtain that 
\begin{align*}
&\Er\Ll[\int_K \vert \nabla \tilde{\phi}_{{p},m} - \nabla \tilde{\phi}_{{p},n}\vert^2(\mu, y)  \, \d \mu(y)\Rr]\\
&\leq \rho \vert K \vert \Er\Ll[\frac{1}{\rho \vert \cu_m \vert}\int_{\cu_m}\Ll \vert \nabla \phi(\mu, y, \cu_m, {p}) - \nabla \phi_{{p},m,n}(\mu, y) \Rr\vert^2 \, \d \mu(y) \Rr]\\
&\leq \rho \vert K \vert (\nu(\cu_n, {p}) - \nu(\cu_m, {p})).
\end{align*}
By \cref{e.identities.ab}, this shows that $\{\nabla \tilde{\phi}_{{p},m}\}_{m \geq 1}$ is a Cauchy sequence in ${(\cL^2_{\bullet,\mathrm{loc}})^d}$. 

\textit{Step 2: Harmonic property - setting up.} Denote the limit by $\nabla \phi_{p}$, we set things up to prove the harmonic property \cref{eq:CorrectorHarmonic} by approximation. We fix $s > 0$ and ${v \in \cC^\infty_c(\Rd) \cap \cH^1_0(\Rd)}$ which is $\mcl F_{Q_s}$-measurable, and observe that
\begin{align*}
& \Er\Ll[\int_{\Rd} \nabla v \cdot \a ({p} + \nabla \phi_{{p}}) \, \d \mu \Rr]\\
& = \lim_{m \to \infty}\Er\Ll[\int_{Q_s} \nabla v \cdot \a ({p} + \nabla \tilde{\phi}_{{p},m}) \, \d \mu \Rr]\\
& = \lim_{m \to \infty}\Er\Ll[\int_{Q_s} \nabla v(\mu, y) \cdot \a(\mu, y) \Ll( \fint_{y + \cu_m} {p} + \nabla \phi(\mu, y, x + \cu_m, {p}) \, \d x\Rr) \, \d \mu(y) \Rr].
\end{align*}
We then use Fubini's lemma to exchange the order of integration, 
\begin{align*}
& \Er\Ll[\int_{\Rd} \nabla v \cdot \a ({p} + \nabla \phi_{{p}}) \, \d \mu \Rr]   = \lim_{m \to \infty}\frac{1}{\vert \cu_m \vert} \\
&\times \Er\Ll[\int_{Q_{3^m + s}} \Ll(\int_{x+\cu_m} \nabla v(\mu, y) \cdot \a(\mu, y) \Ll({p} + \nabla \phi(\mu, y, x + \cu_m, {p})\Rr)  \Ind{y \in Q_s}  \, \d \mu(y)\Rr) \, \d x \Rr].
\end{align*}
For $m$ sufficiently large, we can decompose the domain of integration in $x$ in the expression above into $Q_{3^m - s}$ and $Q_{3^m + s} \backslash Q_{3^m - s}$. We analyse the contribution of each of these quantities in each of the following two steps. 

\textit{Step 3: Integration in $Q_{3^m - s}$.} Notice that for $x \in Q_{3^m - s}$, we have $Q_s \subset x + \cu_m$ (see \Cref{fig:cubeContributes2} for an illustration),  thus we can drop the indicator $\Ind{y \in Q_s}$ in the inner integral and use the $\a$-harmonic property of ${{p} + \nabla \phi(x + \cu_m, {p})}$ to get that
\begin{align*}
&\frac{1}{\vert \cu_m \vert}\Er\Ll[\int_{Q_{3^m - s}} \Ll(\int_{x+\cu_m} \nabla v(\mu, y) \cdot \a(\mu, y) \Ll({p} + \nabla \phi(\mu, y, x + \cu_m, {p})\Rr) \, \d \mu(y)\Rr) \, \d x \Rr] \\
&\quad = \frac{1}{\vert \cu_m \vert}\int_{Q_{3^m - s}} \Er\Ll[\int_{x+\cu_m} \nabla v(\mu, y) \cdot \a(\mu, y) \Ll({p} + \nabla \phi(\mu, y, x + \cu_m, {p})\Rr) \, \d \mu(y) \Rr] \, \d x  \\
&\quad = 0.
\end{align*} 
\begin{figure}[h!]
\centering
\includegraphics[scale=0.5]{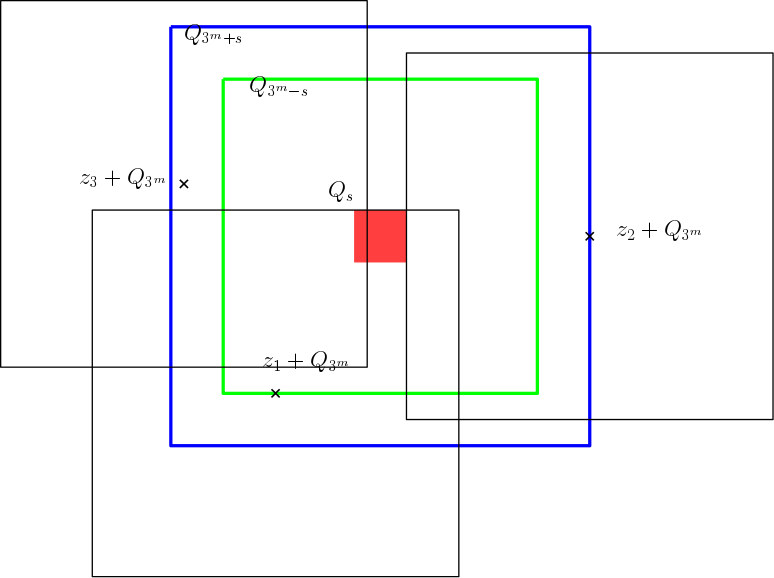}
\caption{The red cube represents $Q_s$, and the blue and green cubes represent respectively $Q_{3^m + s}$ and $Q_{3^m - s}$. For $x \in Q_{3^m - s}$, we have $Q_s \subset x + \square_m$; for $x \notin Q_{3^m + s}$, ${(x + \square_m) \cap Q_s = \emptyset}$; for $x \in Q_{3^m + s} \backslash Q_{3^m - s}$, $x + \square_m$ and $Q_s$ have non-empty intersection but $Q_s$ is not totally contained in $x + \square_m$. These three cases are represented by $z_1, z_2, z_3$.}\label{fig:cubeContributes2}
\end{figure}

\textit{Step 4: Boundary layer $Q_{3^m + s} \backslash Q_{3^m - s}$.} We use Young's inequality to bound the term with the integral over $x \in Q_{3^m + s} \backslash Q_{3^m - s}$ by
\begin{align*}
&\frac{\Lambda}{\vert \cu_m \vert}\Er\Ll[\int_{Q_{3^m + s} \backslash Q_{3^m - s}} \Ll(\int_{(x+\cu_m) \cap Q_s} \vert\nabla v(\mu, y)\vert^2 + \Ll \vert {p} + \nabla \phi(\mu, y, x + \cu_m, {p})\Rr\vert^2  \, \d \mu(y)\Rr) \, \d x \Rr] \\
&\leq \underbrace{\frac{ \Lambda \vert Q_{3^m + s} \backslash Q_{3^m - s} \vert}{\vert \cu_m\vert} \Er\Ll[\int_{Q_s} \vert\nabla v(\mu, y)\vert^2  \, \d \mu(y) \Rr]}_{\mathbf{(A)}}\\
& \quad + \underbrace{\frac{\Lambda}{\vert \cu_m \vert}\Er\Ll[\int_{Q_{3^m + s} \backslash Q_{3^m - s}} \Ll(\int_{(x+\cu_m) \cap Q_s} \Ll \vert {p} + \nabla \phi(\mu, y, x + \cu_m, {p})\Rr\vert^2  \, \d \mu(y)\Rr) \, \d x \Rr]}_{\mathbf{(B)}}.
\end{align*}
For the term $\mathbf{(A)}$, we have, for a constant $C$ that may depend on $s$,
\begin{align*}
\mathbf{(A)} \le   C \Lambda 3^{-m}\Er\Ll[\int_{Q_s} \vert\nabla v(\mu, y)\vert^2  \, \d \mu(y) \Rr] \xrightarrow[m \to \infty]{} 0.
\end{align*}
For the term $\mathbf{(B)}$, we use the stationarity to observe that 
\begin{multline*}
\Er\Ll[\int_{(x+\cu_m) \cap Q_s} \Ll \vert {p} + \nabla \phi(\mu, y, x + \cu_m, {p})\Rr\vert^2  \, \d \mu(y) \Rr] \\
= \Er\Ll[\int_{\cu_m \cap (-x + Q_s)} \Ll \vert {p} + \nabla \phi(\mu, y, \cu_m, {p})\Rr\vert^2  \, \d \mu(y) \Rr].
\end{multline*}
We apply once again Fubini's lemma to $\mathbf{(B)}$ and get that 
\begin{align*}
\mathbf{(B)} & = \frac{\Lambda}{\vert \cu_m \vert}\Er\Ll[\int_{Q_{3^m + s} \backslash Q_{3^m - s}} \Ll(\int_{\cu_m} \Ll \vert {p} + \nabla \phi(\mu, y, \cu_m, {p})\Rr\vert^2 \Ind{y \in (-x + Q_s)} \, \d \mu(y) \Rr) \, \d x \Rr] \\
&= \frac{\Lambda}{\vert \cu_m \vert}\Er\Ll[\int_{\cu_m} \Ll \vert {p} + \nabla \phi(\mu, y, \cu_m, {p})\Rr\vert^2 \Ll(\int_{Q_{3^m + s} \backslash Q_{3^m - s}} \Ind{y \in (-x + Q_s)} \, \d x \Rr)\, \d \mu(y)   \Rr] \\
&\leq \frac{\Lambda \vert Q_s \vert}{\vert \cu_m \vert}\Er\Ll[\int_{\cu_m} \Ll \vert {p} + \nabla \phi(\mu, y, \cu_m, {p})\Rr\vert^2 \Ind{\dist(y, \partial \cu_m) \leq s}\, \d \mu(y)   \Rr].
\end{align*}
That this term converges to zero is a consequence of the stronger estimate given by \Cref{lem:BoundaryLayer} below. This concludes the proof for ${v \in \cC^\infty_c(\Rd) \cap \cH^1_0(\Rd)}$, and then we can use density argument to extend to general ${v \in \cH^1_0(\Rd)}$.
\end{proof}

In the proof above, we appealed to the following boundary layer estimate, which we state as a separate lemma for future reference (and which is stronger than what was needed for the purpose of the proof above, since the boundary layer size is allowed to increase with $m$).

\begin{lemma}[Boundary layer estimate]
For every sequence $(s_m)_{m \in \N}$ such that $s_m \leq 3^m$ and $\lim_{m \to \infty} 3^{-m}s_m =0$, we have  
\label{lem:BoundaryLayer}
\begin{align}\label{eq:BoundaryLayer}
\lim_{m \to \infty}\Er\Ll[\frac{1}{\rho \vert \cu_m \vert} \int_{\cu_m} \Ll \vert {p} + \nabla \phi(\mu, y, \cu_m, {p})\Rr\vert^2 \Ind{\dist(y, \partial \cu_m) \leq s_m}\, \d \mu(y)  \Rr] = 0.
\end{align}
\end{lemma}
\begin{proof}
The idea is to make use of the renormalization argument. We define a mesoscopic scale $n$ associated to $m$ such that $s_m \leq 3^n$, $n \to \infty$ and $m-n \to \infty$. Then we immediately have
\begin{multline*}
\Er\Ll[\frac{1}{\rho \vert \cu_m \vert} \int_{\cu_m} \Ll \vert {p} + \nabla \phi(\mu, y, \cu_m, {p})\Rr\vert^2 \Ind{\dist(y, \partial \cu_m) \leq s_m}\, \d \mu(y)  \Rr] \\
 \leq \Er\Ll[\frac{1}{\rho \vert \cu_m \vert} \int_{\cu_m} \Ll \vert {p} + \nabla \phi(\mu, y, \cu_m, {p})\Rr\vert^2 \Ind{\dist(y, \partial \cu_m) \leq 3^n}\, \d \mu(y)  \Rr]. 
\end{multline*}
We propose to compare $\phi(\cdot, \cu_m,p)$ with $\phi_{{p},m,n} \in \cH^1_0(\cu_m)$ defined as in \cref{e.def.phi.pmn}:
\begin{align*}
\phi_{{p},m,n}(\mu) = \sum_{z \Z_{m,n}} \phi(\mu, z + \cu_n, {p}). 
\end{align*}
Then we have 
\begin{equation}\label{eq:BdDecom}
\begin{split}
&\Er\Ll[\frac{1}{\rho \vert \cu_m \vert} \int_{\cu_m} \Ll \vert {p} + \nabla \phi(\mu, y, \cu_m, {p})\Rr\vert^2 \Ind{\dist(y, \partial \cu_m) \leq 3^n}\, \d \mu(y)  \Rr]\\
& \leq \underbrace{ 2\Er\Ll[\frac{1}{\rho \vert \cu_m \vert} \int_{Q_{3^m} \backslash Q_{3^m-2 \times 3^n}} \Ll \vert {p} + \nabla \phi_{{p},m,n}(\mu, y)\Rr\vert^2 \, \d \mu(y)  \Rr]}_{\text{\cref{eq:BdDecom}-a}} \\
&\qquad + \underbrace{2\Er\Ll[\frac{1}{\rho \vert \cu_m \vert} \int_{Q_{3^m} \backslash Q_{3^m-2 \times 3^n}} \Ll \vert \nabla \phi_{{p},m,n}(\mu, y)- \nabla \phi(\mu, y, \cu_m, {p}) \Rr\vert^2 \, \d \mu(y)  \Rr]}_{\text{\cref{eq:BdDecom}-b}}.
\end{split}
\end{equation}
For the first term \cref{eq:BdDecom}-a, we do partition of sum into cubes of size $3^n$, then we have 
\begin{align*}
&\Er\Ll[\frac{1}{\rho \vert \cu_m \vert} \int_{Q_{3^m} \backslash Q_{3^m-2 \times 3^n}} \Ll \vert {p} + \nabla \phi_{{p},m,n}(\mu, y)\Rr\vert^2 \, \d \mu(y)  \Rr]\\
& = \frac{\vert \cu_n \vert}{\vert \cu_m \vert}\sum_{z \in \Z_{m,n} \cap (Q_{3^m} \backslash Q_{3^m-2 \times 3^n})} \Er\Ll[\frac{1}{\rho \vert \cu_n \vert} \int_{z + \cu_n} \Ll \vert {p} + \nabla \phi(\mu, y, z + \cu_n, {p})\Rr\vert^2 \, \d \mu(y)  \Rr] \\
& = \frac{\vert Q_{3^m} \backslash Q_{3^m-2 \times 3^n}  \vert}{\vert Q_{3^m} \vert} \nu(\cu_n, {p}) \\
& \leq  3^{-(m-n)}\Lambda \vert {p} \vert^2.
\end{align*}
For the second term \cref{eq:BdDecom}-b, we have 
\begin{align*}
&\Er\Ll[\frac{1}{\rho \vert \cu_m \vert} \int_{Q_{3^m} \backslash Q_{3^m-2 \times 3^n}} \Ll \vert \nabla \phi_{{p},m,n}(\mu, y)- \nabla \phi(\mu, y, \cu_m, {p}) \Rr\vert^2 \, \d \mu(y)  \Rr] \\
& \leq \Er\Ll[\frac{1}{\rho \vert \cu_m \vert} \int_{Q_{3^m}} \Ll \vert \nabla \phi_{{p},m,n}(\mu, y)- \nabla \phi(\mu, y, \cu_m, {p}) \Rr\vert^2 \, \d \mu(y)  \Rr]  \\
& = \nu(\cu_n, {p})-\nu(\cu_m, {p}).
\end{align*}
Therefore, when we take $m,n \to \infty$, both  \cref{eq:BdDecom}-a and  \cref{eq:BdDecom}-b go to $0$, so the boundary layer in a mecroscopic scale can be neglected.
\end{proof}

Now that the gradient of the whole-space corrector $\nabla \phi_{p}$ is well-defined, we can proceed to complete the proof of Theorem~\ref{t.local.conv}. 
\begin{proof}[Proof of parts (2) and (3) of Theorem~\ref{t.local.conv}]
We start by discussing the validity of the identities \cref{eq:defab4}  and \cref{eq:defab5}. We use the stationary approximate corrector~$\tilde{\phi}_{{p},n}$ defined in \cref{e.def.tdphi}, and observe from \cref{e.ineq.at.ab} and Theorem~\ref{thm:TwoAb} that 
\begin{align*}
{p} \cdot \ab {p} & 
= \lim_{m \to \infty}\Er \Ll[ ({p} + \nabla \tilde{\phi}_{{p},m}(\mu + \delta_0, 0)) \cdot \a(\mu + \delta_0, 0) ({p} + \nabla \tilde{\phi}_{{p},m}(\mu+\delta_0, 0)) \Rr]
\\
&  = \lim_{m \to \infty}\Er \Ll[ \frac{1}{\rho \vert \cu_0 \vert} \int_{\cu_0} ({p} + \nabla \tilde{\phi}_{{p},m}) \cdot \a ({p} + \nabla \tilde{\phi}_{{p},m}) \, \d \mu \Rr].
\end{align*} 
The identity \cref{eq:defab4} then follows from the convergence of ${\nabla \tilde{\phi}_{{p},m}}$ to ${\nabla \phi_{{p}}}$ in $(\cL^2_{\bullet,\mathrm{loc}})^d$. For the second identity, we can use the fact that
\begin{align*}
\lim_{m \to \infty}\Er\Ll[ \frac{1}{\rho \vert \cu_m \vert} \int_{\cu_m} \a \Ll({p} + \nabla \phi(\mu, \cdot, \cu_m, {p})\Rr)\, \d \mu \Rr] = \ab {p},
\end{align*}
the estimate \cref{eq:AppCorrector}, and the stationarity of $\nabla \phi_p$. The identity \cref{eq:AppCorrectorDual} can also be deduced from \cref{eq:AppCorrector} because
\begin{align*}
&\Er\Ll[ \frac{1}{\rho \vert \cu_m \vert} \int_{\cu_m} \vert \nabla \phi^*(\mu, \cdot, \cu_m, {p}) -  \nabla \phi_{{p}}(\mu, \cdot)\vert^2 \, \d \mu \Rr] \\
& \leq 2\Er\Ll[ \frac{1}{\rho \vert \cu_m \vert} \int_{\cu_m} \vert \nabla \phi(\mu, \cdot, \cu_m, {p}) -  \nabla \phi_{{p}}(\mu, \cdot)\vert^2 \, \d \mu \Rr] \\
& \qquad + 2\Er\Ll[ \frac{1}{\rho \vert \cu_m \vert} \int_{\cu_m} \vert \nabla \phi^*(\mu, \cdot, \cu_m, {p}) -  \nabla \phi(\mu, \cdot, \cu_m, {p})\vert^2 \, \d \mu \Rr].
\end{align*}
By \cref{eq:JVExpression} and \cref{eq:JEnergy}, the second term can be bounded by  $J(\cu_m, {p}, \ab_*(\cu_m){p})$; and by \cref{eq:defJ} and \cref{e.identities.ab}, this quantity converges to $0$ as $m \to \infty$. From now on, we thus focus on the proof of \cref{eq:AppCorrector}.

The idea of the proof of \cref{eq:AppCorrector} is very close to that for \cref{eq:CorrectorHarmonic} and \cref{eq:BoundaryLayer}. We fix a mescroscopic scale $n = \lfloor \frac{m}{3}\rfloor$, and then use the decomposition 
\begin{equation}\label{eq:AppCorrectorDecom}
\begin{split}
&\Er\Ll[ \frac{1}{\rho \vert \cu_m \vert} \int_{\cu_m} \vert  \nabla \phi_{{p}}(\mu, \cdot) - \nabla \phi(\mu, \cdot, \cu_m, {p}) \vert^2 \, \d \mu \Rr] \\
& \leq \underbrace{2\Er\Ll[ \frac{1}{\rho \vert \cu_m \vert} \int_{\cu_m} \vert \nabla \phi_{{p}}(\mu, \cdot) - \nabla \tilde{\phi}_{{p},n}(\mu, \cdot)  \vert^2 \, \d \mu \Rr]}_{\text{\cref{eq:AppCorrectorDecom}-a}} \\
& \qquad +  \underbrace{2\Er\Ll[ \frac{1}{\rho \vert \cu_m \vert} \int_{\cu_m} \vert  \nabla \tilde{\phi}_{{p},n}(\mu, \cdot) - \nabla \phi(\mu, \cdot, \cu_m, {p})\vert^2 \, \d \mu \Rr]}_{\text{\cref{eq:AppCorrectorDecom}-b}}.
\end{split}
\end{equation} 
For the first term \cref{eq:AppCorrectorDecom}-a, we use the stationarity to transform to the integration on unit cube $\cu_0$, and then use the fact that ${\nabla \tilde{\phi}_{{p},n}}$ converges to ${\nabla \phi_{{p}}}$ in $(\cL^2_{\bullet,\mathrm{loc}})^d$ to get that 
\begin{align*}
\lim_{m \to \infty}\text{\cref{eq:AppCorrectorDecom}-a} = \lim_{n \to \infty}\Er\Ll[ \frac{1}{\rho \vert \cu_0 \vert} \int_{\cu_0} \vert \nabla \tilde{\phi}_{{p},n}(\mu, \cdot) -  \nabla \phi_{{p}}(\mu, \cdot)\vert^2 \, \d \mu \Rr] = 0.
\end{align*}
Thus it suffices to finish the second term \cref{eq:AppCorrectorDecom}-b. We use the definition in \cref{e.def.tdphi} and Jensen's inequality to get that
\begin{align*}
&\text{\cref{eq:AppCorrectorDecom}-b} \\
&= 2\Er\Ll[ \frac{1}{\rho \vert \cu_m \vert} \int_{\cu_m} \Ll\vert  \Ll( \fint_{y + \cu_n}\nabla \phi(\mu, y, x +\cu_n, {p}) \, \d x\Rr) - \nabla \phi(\mu, y, \cu_m, {p})\Rr\vert^2 \, \d \mu(y) \Rr]\\
&\leq 2\Er\Ll[ \frac{1}{\rho \vert \cu_m \vert} \int_{\cu_m} \fint_{y + \cu_n}  \Ll\vert \nabla \phi(\mu, y, x +\cu_n, {p})  - \nabla \phi(\mu, y, \cu_m, {p})\Rr\vert^2  \, \d x \, \d \mu(y) \Rr]\\
&\leq 2 \times 3^{-dn}\Er\Ll[ \frac{1}{\rho \vert \cu_m \vert} \int_{Q_{3^m + 3^n}} \int_{x + \cu_n}  \Ll\vert \nabla \phi(\mu, y, x +\cu_n, {p})  - \nabla \phi(\mu, y, \cu_m, {p})\Rr\vert^2  \, \d \mu(y)  \, \d x  \Rr]\\
& \leq 2 \times \Ll(\text{\cref{eq:AppCorrectorDecom}-b1} + \text{\cref{eq:AppCorrectorDecom}-b2} + \text{\cref{eq:AppCorrectorDecom}-b3}\Rr),
\end{align*}
where in the last line we decompose once again the integration into three terms with respect to the domain
\begin{align*}
\text{\cref{eq:AppCorrectorDecom}-b1} &=  3^{-dn}\Er\Ll[ \frac{1}{\rho \vert \cu_m \vert} \int_{Q_{3^m - 10 \times 3^n}} \int_{x + \cu_n}  \Ll\vert \nabla  \phi(\mu, y, x +\cu_n, {p})  - \nabla \phi(\mu, y, \cu_m, {p})\Rr\vert^2  \, \d \mu(y)  \, \d x  \Rr], \\ 
\text{\cref{eq:AppCorrectorDecom}-b2} &=  3^{-dn}\Er\Ll[ \frac{1}{\rho \vert \cu_m \vert} \int_{Q_{3^m + 3^n} \backslash Q_{3^m - 10 \times 3^n}} \int_{x + \cu_n}  \Ll\vert {p} + \nabla  \phi(\mu, y, x +\cu_n, {p})  \Rr\vert^2  \, \d \mu(y)  \, \d x  \Rr], \\
\text{\cref{eq:AppCorrectorDecom}-b3} &=  3^{-dn}\Er\Ll[ \frac{1}{\rho \vert \cu_m \vert} \int_{Q_{3^m + 3^n} \backslash Q_{3^m - 10 \times 3^n}} \int_{x + \cu_n}  \Ll\vert {p} + \nabla \phi(\mu, y, \cu_m, {p}) \Rr\vert^2  \, \d \mu(y)  \, \d x  \Rr].
\end{align*}
The terms \cref{eq:AppCorrectorDecom}-b2 and \cref{eq:AppCorrectorDecom}-b3 are easy to treat as they are boundary layer terms. For \cref{eq:AppCorrectorDecom}-b2 we can use the energy bound
\begin{align*}
\text{\cref{eq:AppCorrectorDecom}-b2} \leq C 3^{-(m-n)} \nu(\cu_n, {p}) \xrightarrow[m \to \infty]{} 0.
\end{align*}
For \cref{eq:AppCorrectorDecom}-b3, since the function $\Ll\vert {p} + \nabla \phi(\mu, y, \cu_m, {p}) \Rr\vert^2$ does not involve $x$, we use Fubini's lemma that 
\begin{align*}
&\text{\cref{eq:AppCorrectorDecom}-b3} \\
&=  3^{-dn}\Er\Ll[ \frac{1}{\rho \vert \cu_m \vert} \int_{Q_{3^m + 3^n}} \Ll(\int_{Q_{3^m + 3^n} \backslash Q_{3^m - 10 \times 3^n}} \Ind{x \in y + \cu}  \, \d x\Rr)  \Ll\vert {p} + \nabla \phi(\mu, y, \cu_m, {p}) \Rr\vert^2  \, \d \mu(y)   \Rr] \\
&\leq  \Er\Ll[ \frac{1}{\rho \vert \cu_m \vert} \int_{Q_{3^m + 3^n} \backslash Q_{3^m - 20 \times 3^n}}   \Ll\vert {p} + \nabla \phi(\mu, y, \cu_m, {p}) \Rr\vert^2  \, \d \mu(y)   \Rr]\\
&= \Er\Ll[ \frac{1}{\rho \vert \cu_m \vert} \int_{Q_{3^m} \backslash Q_{3^m - 20 \times 3^n}}   \Ll\vert {p} + \nabla \phi(\mu, y, \cu_m, {p}) \Rr\vert^2  \, \d \mu(y)   \Rr] + \Er\Ll[ \frac{1}{\rho \vert \cu_m \vert} \int_{Q_{3^m+3^n} \backslash Q_{3^m}}   \Ll\vert p \Rr\vert^2  \, \d \mu(y)   \Rr]\\
&\xrightarrow[m \to \infty]{} 0.
\end{align*}
Here from the second line to the third line, we use the fact that the gradient contributes only on $Q_{3^m+3^n} \backslash Q_{3^m - 20 \times 3^n}$. Then we do a decomposition: the integration on $Q_{3^m+3^n} \backslash Q_{3^m}$ can be calculated directly, since $\phi(\mu, \cu_m, {p})$ is $\mcl F_{\cu_m}$-measurable and the gradient vanishes; the integration on $Q_{3^m + 3^n} \backslash Q_{3^m - 20 \times 3^n}$ can be bounded by the boundary layer estimate in \Cref{lem:BoundaryLayer}.

Finally, we focus on \cref{eq:AppCorrectorDecom}-b1. We rewrite the integration ${\int_{Q_{3^m - 10 \times 3^n}}}$ as
\begin{multline*}
\text{\cref{eq:AppCorrectorDecom}-b1} \\
\leq \Er\Ll[ \frac{1}{\rho \vert \cu_m \vert}  \fint_{\cu_n} \Ll(\sum_{\substack{z \in \Z_{m,n} \\ \dist(z, \partial \cu_m) \geq 5 \times 3^n}}\int_{x + z + \cu_n}  \Ll\vert \nabla  \phi(\mu, y, x + z +\cu_n, {p})  - \nabla \phi(\mu, y, \cu_m, {p})\Rr\vert^2  \, \d \mu(y)  \, \Rr)\d x  \Rr].
\end{multline*} 
For each fixed $x \in \cu_n$, we can propose a sub-minimizer $\ell_{{p}, \cu_m} + w_x$ for $\nu(\cu_m, {p})$ defined by (see \Cref{fig:cubeContributes4} for an illustration)
\begin{align*}
w_x :=  \phi(\cdot, \cu_m \backslash U_x, {p}) + \sum_{\substack{z \in \Z_{m,n} \\ \dist(z, \partial \cu_m) \geq 5 \times 3^n}}  \phi(\cdot, x + z +\cu_n, {p}) ,
\end{align*}
where
\begin{align*}
U_x := \overline{\bigcup_{\substack{z \in \Z_{m,n} \\ \dist(z, \partial \cu_m) \geq 5 \times 3^n} }(x + z +\cu_n)}.  
\end{align*}
The gradient of $w_x$ and $ \phi(\cdot, x + z +\cu_n, {p})$ coincides on every cube $x+z+\cu_n$, so we can write
\begin{align*}
& \Er\Ll[\frac{1}{\rho \vert \cu_m \vert} \Ll(\sum_{\substack{z \in \Z_{m,n} \\ \dist(z, \partial \cu_m) \geq 5 \times 3^n}}\int_{x + z + \cu_n}  \Ll\vert \nabla  \phi(\mu, y, x + z +\cu_n, {p})  - \nabla \phi(\mu, y, \cu_m, {p})\Rr\vert^2  \, \d \mu(y)  \, \Rr)\Rr] \\
& \leq \Er\Ll[\frac{1}{\rho \vert \cu_m \vert} \int_{\cu_m} \vert \nabla w_x(\mu, y) - \nabla \phi(\mu, y, \cu_m, {p}) \vert^2  \, \d \mu(y)  \Rr] \\
& \leq \Ll( \sum_{\substack{z \in \Z_{m,n} \\ \dist(z, \partial \cu_m) \geq 5 \times 3^n}} \frac{\vert \cu_n \vert}{\vert \cu_m \vert} \nu(\cu_n, {p}) +  \frac{\vert  \cu_m \backslash U_x \vert}{\vert \cu_m \vert} \nu(\cu_m \backslash U_x, {p})\Rr) - \nu(\cu_m, {p}) \\
& \leq \nu(\cu_n, {p}) - \nu(\cu_m, {p}) + 5 \times 3^{-\frac{2m}{3}} \Lambda \vert {p} \vert^2,
\end{align*}
where we used the quadratic response \eqref{eq:QuadraRep1} from the second line to the third line. This implies that $\lim_{m \to \infty}\text{\cref{eq:AppCorrectorDecom}-b1} = 0$, and thus completes the proof of \cref{eq:AppCorrector}.
\end{proof}

\begin{figure}[h!]
\centering
\includegraphics[scale=0.5]{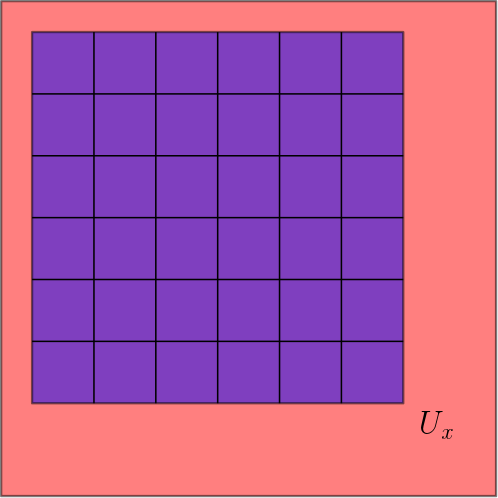}
\caption{The function $\ell_{{p}, \square_m} + w_x$ is a sub-minimizer for the problem $\nu(\square_m, {p})$, which combines the minimizer in cubes of scale $3^n$ biased by a vector $x$ (the cubes in blue), and a minimizer in $U_x$ (the domain in red).}\label{fig:cubeContributes4}
\end{figure}

\subsection*{Acknowledgements}
Part of this project was developed while AG was affiliated to the University of Bonn and supported through the CRC 1060 (The Mathematics of Emergent Effects) that is funded through the German Science Foundation (DFG), and the Hausdorff Center for Mathematics (HCM). CG was supported by a PhD scholarship from Ecole Polytechnique. Part of this project was developed while CG was an academic visitor at the Courant Institute, NYU. JCM was partially supported by the NSF grant DMS-1954357. Part of this project was developed while JCM was affiliated at CNRS and ENS Paris, PSL University, and was partially supported by the ANR grants LSD (ANR-15-CE40-0020-03) and Malin (ANR-16-CE93-0003).

\bibliographystyle{abbrv}
\bibliography{KawasakiRef}

\end{document}